\documentclass[10pt]{article}

\usepackage{authblk}
\usepackage[symbol]{footmisc}

\usepackage{amsfonts}
\usepackage{amsthm}
\usepackage{amsmath}
\usepackage{parskip}
\usepackage{graphicx}
\usepackage{amssymb}
\usepackage{cite}
\usepackage{moresize}
\usepackage{graphicx}
\usepackage{subcaption}
\usepackage{mathtools}
\usepackage{lmodern} 

\usepackage{tikz}
\usepackage{pgfplotstable}
\usetikzlibrary{calc,shapes,arrows,arrows.meta,automata,positioning}
\pgfplotsset{compat=1.18}

\numberwithin{equation}{section}
\mathtoolsset{showonlyrefs}

\usepackage{geometry}
\newgeometry{
    top=3cm,
    bottom=3cm,
    outer=2.5cm,
    inner=2.5cm,
}

\usepackage{color}
\usepackage{xcolor}

\usepackage{makecell}

\usepackage{cite}
\usepackage[colorlinks=true,linkcolor=blue,citecolor=green!65!black]{hyperref}

\newcommand{\A}{\mathcal{A}}

\newcommand{\R}{\mathbb{R}}

\newcommand{\N}{\mathbb{N}}

\newcommand{\Z}{\mathcal{Z}}

\newcommand{\Nash}{\mathcal{N}}

\newcommand{\vareps}{\varepsilon}

\newcommand{\abs}[1]{\left\vert #1 \right\vert}
\newcommand{\norm}[1]{\|#1\|}
\newcommand{\rmd}{\mathrm{d}}

\newtheorem{theorem}{Theorem}[section]
\newtheorem{lemma}[theorem]{Lemma}
\newtheorem{proposition}[theorem]{Proposition}
\newtheorem{corollary}[theorem]{Corollary}

\newtheorem{definition}[theorem]{Definition}

\theoremstyle{remark}
\newtheorem{remark}[theorem]{Remark}
\newtheorem{example}[theorem]{Example}

\begin{document}
    \title{A heterogeneous nonlocal advection--diffusion system}

    

    \author{Joseph McCusker$^{\ast \dagger}$, John Christopher Meyer$^\ast$, Mabel Lizzy Rajendran$^\ast$}

    \date{}
    \maketitle

    \begingroup
    \footnotetext[1]{School of Mathematics, Watson Building, University of Birmingham, Edgbaston, Birmingham, B15 2TT, UK}
    \footnotetext[2]{Corresponding author, jxm1601@bham.ac.uk}
    \endgroup

    
	\begin{abstract}
		We present a self-contained investigation of the local and global well-posedness of a system of nonlocal advection--diffusion equations for a heterogeneous population over $\mathbb{R}^d$, $d \in \mathbb{N}$. Each convolution kernel $K_{ij}$, which describes the nonlocal advection of species $i$ according to the distribution of species $j$, is assumed to have its own regularity $\nabla K_{ij} \in L^{q_{ij}}(\mathbb{R}^d)$, $1 < q_{ij} < \infty$. 
		Local well-posedness of the mild solution and its regularity is obtained using semigroup theory and contraction mapping arguments. 
		For kernels that satisfy a given interaction cycle condition, global existence is established using a Nash-type inequality to show an \textit{a priori} energy bound. For a class of kernels that need not satisfy the interaction cycle condition, a dynamical property satisfied by a normalised \(L^p\) norm of each species is used to obtain a uniform-in-time bound under a smallness condition on the initial data. An example family of kernels with varying regularities, supplemented by numerical simulations, is then considered to illustrate the theory developed in this work and how it may be generalised.
    \end{abstract} 



	\section{Introduction}
    Partial differential equations with nonlocal advection are widely used to model phenomena such as aggregation, segregation, and adhesion, as they capture how the dynamics of a given population is influenced by its spatial distribution. In recent years, heterogeneous systems of such equations have accumulated interest for modelling territory formation between different species in ecology~\cite{fagan2017perceptual, potts2019spatial, painter2024biological, wang2023open}, cell migration involving subpopulations with different adhesion strengths~\cite{armstrong2006continuum, chen2020mathematical} and chemotaxis between species~\cite{he2021multi, tello2012stabilization}.
    
	We consider the following Cauchy problem for a nonlocal advection-diffusion system in \(d \in \N\) spatial dimensions which can be written in component form as
	\begin{equation}\label{eqn:agg_diff} \tag{NLADS}
		\begin{dcases}
			\frac{\partial u_i}{\partial t} = D_i \Delta u_i - \nabla \cdot (u_i \nabla F_i[u]), &\quad (t,x) \in (0,T] \times \R^d, \\
			F_i[u](t,x) :=  \sum_{j=1}^{N} (K_{ij} \ast u_j) (t,x), &\quad (t,x) \in (0,T] \times \R^d,\\
			u_{i}(0,x) = u_{0i}(x), &\quad x \in \R^d,
		\end{dcases}
	\end{equation}
	  where \(\ast\) denotes the spatial convolution
    \begin{equation}
        (K \ast v)(t,x) := \int_{\R^d} K(x-y) v(t,y) \rmd y
    \end{equation}
    and \(T > 0\). For each \(i \in \{1,\dots, N\}\), \(u_i(t,\cdot)\) denotes the population density of the \(i\)-th species in the model at time \(t\), each with a corresponding initial data \(u_{0i}(x)\) and diffusion constant \(D_i > 0\). For each \(i,\, j \in \{1,\dots, N\}\), \(K_{ij}(x)\) denotes a perception kernel.

	In population and opinion dynamics, perception kernels of the form \(K_{ij} = \gamma_{ij} W_{ij}\) are considered, where \(\gamma_{ij} \in \R\) denotes the attractive (when \(\gamma_{ij} > 0\)) or repulsive (when \(\gamma_{ij} < 0\)) strengths and \(W_{ij}\) is a probability density function describing how species \(i\) perceives the spatial or opinion distribution of species \(j\) within some neighbourhood \cite{giunta2024weakly}. The weighting of the perception at each location is determined via a combination of how clearly species \(i\) observes species \(j\) and how relatively significant the presence of species \(j\) is at said location. Under this interpretation, the nonlocal advection term for species \(i\) describes its movement in the direction of \(\nabla F_i[u]\). Some simple examples of perception kernels include:
    \begin{itemize}
        \item \emph{Uniform/top hat distribution}~\cite{potts2019spatial, fagan2017perceptual}. For some sensing radius \(R > 0\), consider
		\begin{equation}
			x \mapsto \begin{cases}
				1/(\omega_d R^d), &\quad \abs{x} \leq R, \\
				0, & \quad \text{otherwise},
			\end{cases}
		\end{equation}
		where \(\omega_d\) is the volume of a \(d\)-dimensional unit ball.
        \item \emph{Raised cosine distribution}~\cite{giunta2025phylogeny}. For some sensing radius \(R > 0\), consider
		\begin{equation}
			x \mapsto \begin{cases}
				\displaystyle{c_d \left(1 + \cos\left(\frac{\pi\abs{x}}{R}\right)\right)}, &\quad \abs{x} \leq R, \\
				0, & \quad \text{otherwise},
			\end{cases}
		\end{equation}
		where \(c_d > 0\) is some normalisation constant.
        \item \emph{Exponential distribution}~\cite{bernoff2013nonlocal, fagan2017perceptual}. For some \(\lambda > 0\),
		\begin{equation}
			x \mapsto \frac{\lambda^d}{\omega_d} \exp(-\lambda \abs{x}), \quad \abs{x} \geq 0.
		\end{equation}
    \end{itemize}
	While the examples stated above are radially symmetric probability density functions, it is also natural to consider other options for a perception kernel. For instance, kernels without radial symmetry arise in models with anisotropic observations such as in biology, where organisms that can only detect other individuals in their direction of travel~\cite{milewski2008simple} and in opinion dynamics models, where a group of individuals is biased towards a particular direction in an opinion space~\cite{koertje2024collective}. Kernels that are not probability density functions may be chosen for models which account for over-crowding in swarms by including a short-range repulsion in an otherwise attractive perception kernel~\cite{fetecau2011swarm}. 
	
	Another example of kernels that are not probability density functions are kernels of the form \(K_{ij} = \gamma_{ij} W_{ij}\), where \(\gamma_{ij} > 0\) for each \(i,\,j\) and \(W_{ij}\) is the Green's function for the elliptic operator \(\alpha_{ij} I - \Delta\). In this case,~\eqref{eqn:agg_diff} is a multi-species parabolic--elliptic Patlak--Keller--Segel system, which is used for modelling chemotaxis phenomena. Such an example is a biofilm including different species or phenotypes of bacteria that communicate through chemical signalling~\cite{he2021multi}. In literature, \(W_{ij}\) is referred to as the \emph{Newtonian potential} for \(\alpha_{ij} = 0\), and the \emph{Bessel potential} for \(\alpha_{ij} > 0\)~\cite[\S V.3]{stein1970singular}.

	The analysis of~\eqref{eqn:agg_diff} in this work falls into three parts. The first part establishes a broad class of kernels \(K\) and initial data \(u_0\) for which the system is locally well-posed, preserves the non-negativity and conserves the mass of each species in the system. The second part focuses on determining a class of kernels such that solutions of~\eqref{eqn:agg_diff} can be continued indefinitely in time so long as \(u_0 \in L^1 \cap L^\infty(\R^d)^N\), in particular, determining for which kernels a finite-time blow-up of the \(L^p\) norms of the solution cannot occur, hence preventing a \emph{chemotactic collapse}, as studied in~\cite{blanchet2006two,herrero1996chemotactic}. The final part considers the dynamical properties of an auxiliary concentration functional, which measures the ratio between the \(L^p\) and \(L^1\) norms of each species. Consequently, such properties unlock machinery for establishing uniform-in-time bounds of the \(L^p\) norms of solutions of~\eqref{eqn:agg_diff} without requiring the solution to tend to a steady state.

	By considering a family of radially symmetric kernels such that, for each \(i,\,j\), \(\nabla K_{ij} \in L^r(\R^2),\, r > 2\), compactly supported, or \(\nabla K_{ij} \in L^3(\R^3)\) and initial data \(u_{0i} \in W^{3,1} \cap W^{3,\infty}(\R^d)\) for either \(d = 2\) or \(d=3\), for each \(i,\, j\),~\cite{cozzi2025global} demonstrate global well-posedness of~\eqref{eqn:agg_diff}, given a small mass estimate on the initial data. By instead considering~\eqref{eqn:agg_diff} over the $d$-dimensional torus and assuming each kernel \(K_{ij}\) is twice differentiable over the \(d\)-dimensional torus with bounded derivatives and with \(u_{0i}\) twice differentiable for each \(i,\, j \in \{1,\dots,N\}\), \cite{giuntaLocalGlobalExistence2022,giunta2025positivity} establish global existence and uniqueness of local solutions for arbitrary spatial dimension, and classical global solutions in one dimension. Note that the analysis for both results utilises the assumption that all the kernels are (weakly) differentiable. This allows the advection terms in~\eqref{eqn:agg_diff} to be considered as lower-order terms with respect to the second-order linear diffusion terms since, for each \(i,\,j\), it follows that \(\nabla (K_{ij} \ast u_j) = (\nabla K_{ij}) \ast u_j\). As a consequence, local in time solutions can be obtained using semigroup theory, which are then continued to global solutions via \emph{a priori} bounds on the solution.

	To account for kernels with jump discontinuities, hence not weakly differentiable, a more delicate analysis is required. Motivated by the top hat kernel,~\cite{carrillo2024well} prove global well-posedness of~\eqref{eqn:agg_diff} for integrable and even kernels of bounded variation using a subtle application of entropy methods and gradient flow. In order to extend the application of these tools from the homogeneous model to the heterogeneous model, the analysis requires the \emph{detailed balance} condition on the kernels, i.e., there exists some constants \(\pi_1,\dots,\pi_N > 0\) such that \(\pi_i K_{ij} = \pi_j K_{ji}\) for each choice of \(i\) and \(j\), which is necessary and sufficient for maintaining a gradient flow structure. For related gradient flow arguments over the torus, see also\cite{jungel2022nonlocal, carrillo2026long}. 
	
	A motivation for this work is studying the global well-posedness of~\eqref{eqn:agg_diff} with minimal structural constraints between the shapes and regularities of the kernels, hence allowing for greater heterogeneity between the behaviours of the sub-populations. For instance, the analysis in this work allows for scenarios in which the so-called \emph{sensing radius} of species \(i\) observing \(j\) is larger than the sensing radius of the reverse observation. In particular, if species \(i\) is attracted to species \(j\) and species \(j\) is repulsed by species \(i\), then the model returns a range of chase and run scenarios, whose numerical solutions have been studied extensively in~\cite{painter2024variations, giunta2025phylogeny}. 

	To allow for minimal assumptions between the kernels, all perception kernels are assumed to have independent regularities, which has not been previously considered for~\eqref{eqn:agg_diff} to the authors' knowledge. That is, we assume that \(\nabla K_{ij} \in L^{q_{ij}}(\R^d),\, q_{ij} \in (1,\infty)\), for each choice of \(i,\, j \in \{1,\dots,N\}\). A key insight is that a global bound for solutions of~\eqref{eqn:agg_diff} follows in the case that the geometric mean of \(q_{ij}\) values over every \emph{interaction cycle}, as introduced in Definition~\ref{defn:int_cycle}, is bounded below by the number of spatial dimensions \(d\). This allows for some of the cross-perception kernels, \(K_{ij},\, i \neq j\), to be more singular than kernels in the homogeneous setting. For example,\cite{karch2011blow} showed that for the Bessel potential kernel, which has the regularity \(\nabla K \in L^{q}(\R^d)\) for any \(q \in [1,\frac{d}{d-1})\), the single-species analog of~\eqref{eqn:agg_diff} has a finite-time blow-up for dimension \(d \geq 2\). However, if a cross-perception kernel is the Bessel potential, then a global bound remains when the other cross-perception kernels are sufficiently regular to offset the singularity of the Bessel potential.

	Another insight that naturally arises from the analysis in this work concerns how an auxiliary concentration functional of each species evolves over time. For the \(N = 1\) system, this framework can be used to directly establish that, if \(u_0 \in L^1 \cap L^p(\R^d)\), \(u\) has a uniform-in-time \(L^p\) bound if one of the three conditions are satisfied: \(\nabla K \in L^q(\R^d)\) for some \(q \in (d,\infty)\), \(\nabla K \in L^d(\R^d)\) and \(u_0\) has sufficiently small mass, or \(\nabla K \in L^q(\R^d)\) for some \(q \in (1,d)\) and \(\norm{u_0}_{L^p}/\norm{u_0}_{L^1}\) is sufficiently small for some \(p \geq \max\{2, q'\}\). The concentration functional framework can also be used to establish uniform-in-time solutions even when the kernels do not satisfy the geometric mean condition for every interaction cycle, given that the concentrations of each species is sufficiently small.

    \section{Notation}

    Below is an overview of the different function spaces considered throughout the following work. Note that, for any \(T > 0\), \(\Omega_T\) denotes the space \((0,T] \times \R^d\) throughout.
    
	\emph{Lebesgue Spaces.} A measurable function \(\varphi : \R^d \to \R\) is in \(L^p(\R^d)\) for some \(p \in [1,\infty]\) if its norm \(\norm{\varphi}_{L^p}\) is finite, defined as
	\begin{equation}
		\norm{\varphi}_{L^p}:= 
		\begin{dcases}
			{\left(\int_{\R^d} \abs{\varphi}^p \rmd x\right)^{1/p}}, & 1 \leq p < \infty,\\
			{\mathrm{ess\, sup}_{x \in \R^d}\abs{\varphi(x)}}, & p = \infty.
		\end{dcases}
	\end{equation}
	We say that two exponents \(p,\, p' \in [1,\infty]\) are H\"older conjugates of each other if \(\frac{1}{p} + \frac{1}{p'} = 1\), using the convention \(1/\infty = 0\). For \(1 < p < \infty\), we write \(p' := p/(p-1)\).
	
	\emph{Sobolev spaces.} Let \(\alpha\in \N_0^d\) denote a multi-index with size \(\abs{\alpha} = \alpha_1 + \cdots + \alpha_d\) from which we use the notation \(D^\alpha \varphi := \partial_{x_1}^{\alpha_1} \cdots \partial_{x_d}^{\alpha_d}\, \varphi\).
	A locally integrable function \(\varphi: \R^d \to \R\) is in \(W^{k,p}(\R^d)\) for some \(p \in [1,\infty],\, k \in \N_0\), if \(\varphi\) is \(k\)-times (weakly) differentiable and its norm \(\norm{\varphi}_{W^{k,p}}\) is finite, defined as
	\begin{equation}
		\norm{\varphi}_{W^{k,p}} :=  \sum_{\abs{\alpha}\leq k} \norm{D^\alpha \varphi}_{L^p}.
	\end{equation}   
    Furthermore, \(\nabla \varphi = (\partial_{x_1} \varphi,\dots,\partial_{x_d}\varphi) \in L^p(\R^d)\) when \(\norm{\nabla \varphi}_{L^p} := \sum_{i=1}^d \norm{\partial_{x_i} \varphi}_{L^p}\) is finite. Similarly \(\nabla^2 \varphi = (\partial_{x_i} \partial_{x_j} \varphi)_{i,j=1}^d \in L^p(\R^d)\) when \(\norm{\nabla^2 \varphi}_{L^p} := \sum_{i,\, j=1}^d \norm{\partial_{x_i} \partial_{x_j} \varphi}_{L^p}\) is finite.

	\emph{Bessel potential spaces.} For some \(s \in [0,\infty)\), consider the Bessel potential operator \(\mathcal{J}_s\) defined such that, for any measurable \(\psi\),
	\begin{equation}
		{\mathcal{J}_{s}\psi} := (\mathrm{Id} - \Delta)^{-s/2}\psi = J_s \ast \psi,	
	\end{equation}
	where \(\widehat{J_s} := (1 + \abs{x}^2)^{-s/2}\) is the Fourier transform of \(J_s\). With this in mind, we say that a locally integrable function \(\varphi\) is in \(H^{s,p}(\R^d)\) for some \( s \in [0,\infty)\) and \(p \in (1,\infty)\) if there exists some \(\psi \in L^p(\R^d)\) such that \(\mathcal{J}_s \psi = \varphi\) and we write its norm as 
	\begin{equation}\label{def:BPS_norm}
		\norm{\varphi}_{H^{s,p}} := \norm{\psi}_{L^p}.
	\end{equation}	
	Note that \(H^{k,p}(\R^d)\) and \(W^{k,p}(\R^d)\) are equivalent spaces for any \(k \in \N_0,\, p \in (1,\infty)\), see~\cite[\(\S 6.2\)]{grafakos2009modern} for details.

	\emph{Bochner spaces.} For \(1 \leq p \leq \infty\), \(L^p([0,T];X)\) is the space of functions \(f:[0,T] \to X\), such that \(t \mapsto \norm{f(t)}_X\) is measurable, with the norm \(\norm{f}_{L^p_T X}\) defined as
	\begin{equation}\label{defn:L^infty(X)_norm}
		\norm{f}_{L_T^p X} := 
		\begin{dcases}
            {\left(\int_0^T \norm{f(t)}_{X}^p \rmd t\right)^{1/p}}, & \quad 1 \leq p < \infty, \\
		    {\mathrm{ess\, sup}_{0 \leq t \leq T} \norm{f(t)}_X}, & \quad p = \infty.
		\end{dcases} 
	\end{equation}
	Furthermore, we say that \(f \in C(I;X)\) for some interval \(I \subset [0,T]\) if, for every \(t \in I\),
    \begin{equation}
        \lim_{h \to 0} \norm{f(t+h)-f(t)}_X = 0
    \end{equation}
    and \(f \in C^1(I;X)\) if there exists some \(g \in C(I;X)\) such that, for every \(t \in I\),
    \begin{equation}
        \lim_{h \to 0} \left\|\frac{f(t+h)-f(t)}{h} - g(t) \right\|_X = 0.
    \end{equation}
    In such cases we will denote \(g\) as \(\frac{\rmd}{\rmd t} f\).

    \emph{Mixed Lebesgue spaces.} If \(X = L^q(\R^d)\) for some \(q \in [1,\infty]\), then we will denote \(L^{p,q}_{t,x}(\Omega_T) := L^p([0,T];L^q(\R^d))\). If \(p=q\) then we will simply write \(L^p_{t,x}(\Omega_T)\).

    \emph{Classical spaces.} \(C^k(\R^d)\) denotes the space of \(k\)-times continuously differentiable functions and \(C^\infty(\R^d) := \bigcap_{k\geq 1} C^k(\R^d)\) denotes the space of \emph{smooth} functions. \(C^k_u(\R^d)\) denotes the space of \(k\)-times \emph{uniformly} continuous functions.
    Furthermore, consider the space \(\Omega_T := (0,T] \times \R^d\). \(C^{l,k}_{t,x}(\Omega_T)\) denotes the space of functions \(\varphi : \Omega_T \to \R\) such that \(\varphi(\cdot,x) \in C^l((0,T])\) for each \(x \in \R^d\) and \(\varphi(t,\cdot) \in C^k(\R^d)\) for each \(t \in [0,T]\).
    
	\emph{Intersection spaces.} Now consider \(X,\, Y\) to be a pair of Banach spaces with non-empty intersection. Then \(X \cap Y\) is an \emph{intersection} space with norm 
	\begin{equation}\label{defn:l1_cap_lp}
		\norm{\cdot}_{X \cap Y} := \norm{\cdot}_X + \norm{\cdot}_Y.
	\end{equation}
	\emph{Product spaces.} For any \(N \in \N\), \(P = (p_i)_{i=1}^N \in [1,\infty]^N\) and \(Q = (q_{ij})_{ij} \in [1,\infty]^{N \times N}\), we define the following product spaces as
	\begin{equation}\label{def:prod_leb_space}
			L^P(\R^d) := \bigoplus_{i=1}^N L^{p_i}(\R^d), \quad L^Q(\R^d) := \bigoplus_{i,j=1}^N L^{q_{ij}}(\R^d)
	\end{equation}
	with associated norms \(\norm{\varphi}_{L^P} := \sum_{i=1}^N \{\norm{\varphi_i}_{L^{p_i}}\}\) and \(\norm{\varphi}_{L^Q} := \sum_{i,j=1}^N \{\norm{\varphi_{ij}}_{L^{q_{ij}}}\}\).
    Furthermore, for \(R,\, P \in (1,\infty)^N\), such that \(r_i \leq p_i\) for each \(i \in \{1,\dots,N\}\), define the following spaces
    \begin{equation}
        L^1 \cap L^P(\R^d) := \bigoplus_{i=1}^N L^1 \cap L^{p_i}(\R^d), \quad H_{R,P}^s(\R^d) := \bigoplus_{i=1}^N H^{s,r_i} \cap H^{s,p_i}(\R^d) 
    \end{equation}
    with associated norms \(\norm{\varphi}_{L^1 \cap L^P} := \sum_{i=1}^N \{\norm{\varphi_i}_{L^1 \cap L^{p_i}}\}\) and \(\norm{\varphi}_{H_{R,P}^s} := \sum_{i=1}^N \{\norm{\varphi_i}_{{H^{s,r_i} \cap H^{s,p_i}}}\}\). When \(s = n \in \N\), we will denote \(H_{R,P}^n(\R^d) = W_{R,P}^n(\R^d)\).

    \emph{Sum Lebesgue space.} For some \(Q_1 = (q_{ij1})_{ij},\, Q_2 = (q_{ij2})_{ij} \in (1,\infty]^{N \times N}\) we define
	\begin{equation}
		L^{Q_1}(\R^d) + L^{Q_2}(\R^d) := \{G_1 + G_2 : G_1 \in L^{Q_1}(\R^d),\, G_2 \in L^{Q_2}(\R^d)\},
	\end{equation} 
	with the associated norm \(\norm{H}_{L^{Q_1}+L^{Q_2}} := \inf\{\norm{G_1}_{L^{Q_1}} + \norm{G_2}_{L^{Q_2}}: H = G_1 + G_2\}\).

    \section{Main results}

    We first demonstrate the local in time well-posedness of solutions to~\eqref{eqn:agg_diff} in some sense for rough initial data. To achieve this, consider \emph{mild} solutions of~\eqref{eqn:agg_diff}.

	 \begin{definition}[Mild solution]\label{def:mild}
		Let \(u_0 \in L^1 \cap L^P(\R^d)\) for some \(P \in [1,\infty]^N\) and \(\nabla K \in L^Q(\R^d)\) for some \(Q \in [1,\infty]^{N \times N}\). We say that  \(u = (u_1,\dots,u_N) \in C([0,T];L^1 \cap L^P(\R^d))\) is a \emph{mild} solution of~\eqref{eqn:agg_diff} up to some time \(T > 0\) with kernels \(K\) and initial data \(u_0\) if it is a fixed point of \(\Psi[\cdot;u_0] = (\Psi_1[\cdot;u_0],\dots,\Psi_N[\cdot;u_0])\), defined for each \(t \in [0,T]\) and \(i \in \{1,\dots,N\}\) as
		\begin{equation}\label{eqn:mild_solution}
			\Psi_i[u;u_0](t) := S_i(t)u_{0i} - B_i[u,u](t),
		\end{equation}
		where \(S_i(t)\) denotes the heat semigroup for diffusion constant \(D_i\) and \(B_i[u,u](t)\) is defined for each \(t \in [0,T]\) and \(i \in \{1,\dots,N\}\) to satisfy
		\begin{equation}\label{defn:B}
			B_i[u,v](t) := \int_0^t \nabla S_i(t-\tau) \cdot (u_i(\tau) \nabla F_i[v(\tau)]) \rmd \tau,\quad F_i[v(\tau)] = \sum_{j=1}^N K_{ij} \ast v_j(\tau). 
		\end{equation} 
	\end{definition}

	In Section~\ref{section:LWP} existence of a mild solution is established via a contraction mapping argument, from which continuous dependence on the choice of kernels and initial data follows from a Gr\"onwall-type inequality. These results can be summarised as the following theorem.
    
	\begin{theorem}[Local well-posedness]\label{Thm:local}
		Suppose that \(\nabla K \in L^Q(\R^d)\) for some \(Q \in (1,\infty]^{N \times N}\) and \(u_{0} \in L^1 \cap L^{P}(\R^d)\), for some \(P \in [1,\infty)^N\) such that \(q_{ij}' \leq p_j\) for every \(i,\, j \in \{1,\dots,N\}\). Then, there exists some \(T > 0\) and a unique \(u \in C([0,T]; L^1 \cap L^{P}(\R^d))\) which is a mild solution to~\eqref{eqn:agg_diff} up to time \(T\). Furthermore, the solution in the \(L^1 \cap L^P(\R^d)\) sense continuously depends on the initial data in the \(L^1 \cap L^P(\R^d)\) sense and the gradients of the kernels in the \(L^Q(\R^d)\) sense. 
	\end{theorem}

	\begin{example}
		Typical examples of kernels satisfying \(\nabla K_{ij} \in L^{q_{ij}}(\R^d)\) for each \(i,\, j \in \{1,\dots,N\}\) have the following behaviours in some neighbourhood around the origin (up to some additive constant) and be smooth with sufficiently fast decay as \(|x| \to \infty\) outside the neighbourhood:
		\begin{itemize}
			\item For \(\nabla K_{ij} \in L^{q_{ij}}(\R^d)\), \(q_{ij} > d\), choose \(K_{ij}\) to be of the form \(K_{ij}(x) \propto \abs{x}^\gamma\), \(
            0 \leq \gamma < 1 - d/q_{ij}\), in some neighbourhood of \(0\).
			\item For \(\nabla K_{ij} \in L^{d}(\R^d)\), \(d \geq 2\), choose \(K_{ij}\) to be of the form \(K_{ij}(x) \propto \log(\log(1+\abs{x}^{-1}))\) in some neighbourhood of \(0\).
			\item For \(\nabla K_{ij} \in L^{q_{ij}}(\R^d)\), \(q_{ij} < d\), \(d \geq 2\), choose \(K_{ij}\) to be of the form \(K_{ij}(x) \propto \abs{x}^{-\gamma}\), \(0 < \gamma < -1 + d/q_{ij}\), or \(K_{ij}(x) \propto \log(\abs{x})\) in some neighbourhood of \(0\).
		\end{itemize}
	\end{example}

	By then asserting the technical assumption that \(q_{ij} < \infty\) for each choice of \(i,\, j \in \{1,\dots,N\}\), the instantaneous smoothing of mild solutions for positive time is then proven in Section~\ref{section:regularity} via a bootstrapping argument on Bessel potential spaces of the type \(H^s_{R,P}(\R^d)\). 

	\begin{theorem}[Regularity]\label{thm:regularity}
		Consider \(Q \in (1,\infty)^{N \times N}\) and some \(R,\, P \in (1,\infty)^N\) such that, for every \(i,\,j \in \{1,\dots,N\}\), \(r_j \leq q_{ij}' \leq p_j\). Suppose that \(u\) is a mild solution to~\eqref{eqn:agg_diff} up to time \(T > 0\) with kernels \(\nabla K \in L^Q(\R^d)\) and initial data \(u_{0} \in L^1 \cap L^{P}(\R^d)\). Then, \(u\) has the regularity 
		\begin{equation} 
			u \in C([0,T];L^1 \cap L^P(\R^d)) \cap C((0,T];W_{R,P}^{n}(\R^d)) \cap C^1((0,T];W_{R,P}^{n-2}(\R^d)),
		\end{equation}
		for any \(n \geq 2\).
	\end{theorem}

	Using a Sobolev embedding argument, it can then be concluded that mild solutions of~\eqref{eqn:agg_diff} are classical solutions in the following sense.

	\begin{definition}[Classical solution]\label{defn:classical-sol}
		Let \(u_0 \in L^1 \cap L^P(\R^d)\) for some \(P \in [1,\infty)^N\) and \(\nabla K \in L^Q(\R^d)\) for some \(Q \in (1,\infty]^{N \times N}\). We say that  \(u = (u_1,\dots,u_N)\) is a \emph{classical} solution of~\eqref{eqn:agg_diff} up to some time \(T > 0\) with kernels \(K\) and initial data \(u_0\) if it satisfies the following for each \(i \in \{1,\dots,N\}\):
		\begin{itemize}
			\item $u_i \in C([0,T];L^1 \cap L^{p_i}(\R^d)) \cap C((0,T]; C_u^{2}(\R^d))$,
			\item $\partial_t u_i,\, \nabla u_i,\, \nabla \cdot (u_i \nabla F_i[u]) \in C((0,T];C_u(\R^d) \cap L^{p_i}(\R^d))$,
			\item $u(t,x)$ solves~\eqref{eqn:agg_diff} pointwise for every \((t,x) \in (0,T] \times \R^d\),
			\item \(u(t) \to u_0\) in the \(L^1 \cap L^P\) sense as \(t \to 0^+\).
		\end{itemize}
	\end{definition}

	With classical derivatives of the solution, non-negativity of solutions can then be established in Section~\ref{section:preserved}, which in turn allows us to demonstrate conservation of mass.

	\begin{theorem}[Non-negativity and conservation of mass]\label{thm:properties}
		Suppose that \(u\) is a solution to~\eqref{eqn:agg_diff} up to time \(T > 0\) with kernels \(\nabla K \in L^Q(\R^d)\) for some \(Q \in (1,\infty)^{N \times N}\) and initial data \(u_{0} \in L^1 \cap L^{P}(\R^d)\) for some \(P \in (1,\infty)^N\) such that \(q_{ij}' \leq p_j\) for every \(i,\, j \in \{1,\dots,N\}\). Further suppose that \(u_{0i} \geq 0\) almost everywhere for all \(i\). Then, for each \(i \in \{1,\dots, N\}\) and \(t > 0\), \(u_i(t) \geq 0\) and \(\norm{u_i(t)}_{L^1} = \norm{u_{0i}}_{L^1}\).
	\end{theorem}
    
	In the Section~\ref{section:global}, the focus then shifts to finding sufficient conditions for global in time solutions. This is achieved by balancing the regularities of the perception kernels. A key concept utilised for this purpose is the \emph{perception graph} of the system described by~\eqref{eqn:agg_diff}. 
	
	\begin{definition}[Perception graph]
		The \emph{perception graph} corresponding to the system~\eqref{eqn:agg_diff} is a directed graph with \(N\) nodes, constructed as follows. First enumerate each node from \(1\) up to \(N\). Then, for every pair of nodes \(i,\, j \in \{1,\dots,N\}\), draw an edge from node \(i\) to node \(j\) if there exists some \(\varphi \in C^1(\R^d)\) such that \(\nabla (K_{ij} \ast \varphi) \neq 0\).
	\end{definition} 

	\begin{remark}
		For kernels of the form \(\nabla K_{ij} \in L^{q_{ij}}(\R^d)\) for some \(q_{ij} \in [1,\infty]\), the condition for drawing an edge on the perception graph is equivalent to the assertion that \(\nabla K_{ij}\) is \emph{non-trivial}, i.e.\ \(\norm{\nabla K_{ij}}_{L^{q_{ij}}} > 0\). 
	\end{remark}

	To illustrate the definition, consider the following example which will be referred to throughout the work.

	\begin{example}[Perception graph of a 4-system]\label{ex:perp-graph}
		Suppose the family of kernels \(K\) is given as
		\begin{equation}\label{def:4-sys-kernels}
			K = \begin{pmatrix}
				0 & K_{12} & 0 & 0 \\
				K_{21} & 0 & K_{23} & K_{24} \\
				K_{31} & 0 & 0 & 0 \\
				0 & 0 & 0 & K_{44}
			\end{pmatrix}
		\end{equation}
		where, for each given \(K_{ij}\), \(\nabla K_{ij}\) is non-trivial. Then, the corresponding perception graph of the system~\eqref{eqn:agg_diff} with respect to \(K\) is given in Figure~\ref{fig:eg-perp-graph}.

		\begin{figure}[ht]
			\centering
			\begin{tikzpicture}[state/.append style={minimum size=6mm, inner sep=1pt, thick}]
				\node[state] (1) at (-2, 1) {$1$};
				\node[state] (2) at (0, 0) {$2$};
				\node[state] (3) at (-2, -1) {$3$};
				\node[state] (4) at (1.5, 0) {$4$};
				
				\path[->, >=angle 90]
					(1) edge [thick, bend left=30] (2)
					(2) edge [thick, bend left=30] (1)
					(2) edge [thick, bend left=30] (3)
					(2) edge [thick] (4)
					(3) edge [thick, bend left=30] (1)
					(4) edge [thick, out = -45, in = 45, looseness=12] (4);
			\end{tikzpicture}			
			\caption{A perception graph for the system corresponding to perception kernels given by~\eqref{def:4-sys-kernels}.} 
			\label{fig:eg-perp-graph}
		\end{figure}
	\end{example}
	
	By considering the perception graph for a system, it becomes clear how each species influences their own movement. In Figure~\ref{fig:eg-perp-graph}, observe that species $4$ influences itself directly via the perception kernel $K_{44}$. On the other hand, by following the reverse directions of the perceptions, notice that species $1$ also influences its own movement indirectly via two \emph{interaction cycles}, defined as follows. 

	\begin{definition}[Interaction cycle]\label{defn:int_cycle}
        Suppose \(\nabla K \in L^Q(\R^d)\) for some \(Q \in (1,\infty)^{N \times N}\). For \(n \in \N\), define an $n$-tuple of perception kernels \((K_{i_1 i_2},\dots, K_{i_n i_1})\) corresponding to the index cycle \(i_1 \rightarrow \cdots \rightarrow i_n \rightarrow i_1\) as an \emph{interaction cycle of \(K\)} if \(i_1,\dots,i_n\) are all distinct and \(\nabla K_{i_{k}i_{k+1}}\) is non-trivial for every \(k \in \{1,\dots, n\}\), where \(i_{n+1} = i_1\).
	\end{definition}

	Referring back to species 1 in Figure~\ref{fig:eg-perp-graph}, we see that it influences its own movement via the interaction cycles \((K_{12},K_{21})\) and \((K_{12}, K_{23}, K_{31})\).

	To control the growth of the \(L^P\) norm of the solution, and consequently continue the solution globally in time, it is sufficient to bound an energy functional using a Nash-type inequality, as seen in Section~\ref{section:global}. To conclude the argument, a regularity condition on each interaction cycle is imposed. Said condition can be characterised by a geometric mean of the exponents \(q_{ij}\) along each of the cycles, as seen in the following theorem.

	\begin{theorem}[Global existence]\label{thm:global-1}
		Suppose that \(u\) is a solution to~\eqref{eqn:agg_diff} for kernels \(\nabla K \in L^Q(\R^d)\) for some \(Q \in (1,\infty)^{N \times N}\) and initial data \(u_0 \in L^1 \cap L^P(\R^d)\) for some \(P \in [2,\infty)^N\) such that \(q_{ij}' \leq p_j\) for every \(i,\, j \in \{1,\dots,N\}\). Further suppose that, for every interaction cycle of \(K\), the corresponding exponents \(q_{ij}\) satisfy
		\begin{equation}\label{cond:reg_cycle-intro}
            \left(q_{i_1 i_2} \cdots q_{i_n i_1} \right)^{1/n} \geq d.
        \end{equation}
        Then, after refining the choice of \(P\) to satisfy
        \begin{equation}
			\frac{d(p_i-1) + 2}{d(p_j-1) + 2} \leq \frac{q_{ij}}{d}
		\end{equation}
		for all \(i,\, j \in \{1,\dots, N\}\) such that \(\nabla K_{ij}\) is non-trivial, the solution to~\eqref{eqn:agg_diff} can be continued in time up to any time \(T > 0\).
	\end{theorem}

	\begin{example}\label{ex:good-cycles}
		Consider the kernels \(K\) as given in Example~\ref{ex:perp-graph}. From the corresponding perception graph it can be seen that there are three interaction cycles: \((K_{44})\), \((K_{12}, K_{21})\), and \((K_{12}, K_{23}, K_{31})\). Suppose \(\nabla K \in L^Q(\R^d)\) for \(d = 4\) and 
		\begin{equation}
			Q = \begin{pmatrix}
				- & 12 & - & - \\
				3 & - & 2 & 2 \\
				9 & - & - & - \\
				- & - & - & 5
			\end{pmatrix}
		\end{equation}
		where the `$-$' entries correspond to the trivial kernels and hence can be arbitrarily chosen between \(1\) and \(\infty\). By verifying that 
		\begin{align*}
			q_{44} &= 5 \geq d, \\
			(q_{12} q_{21})^{1/2} &= 6 \geq d, \\
			(q_{12} q_{23} q_{31})^{1/3} &= 6 \geq d,
		\end{align*}
		we can conclude via Theorem~\ref{thm:global-1} that the solutions exist globally in time. Note that since the kernel \(K_{24}\) is not contained within an interaction cycle, its regularity has no influence on the continuation of solutions.
	\end{example}

	Note that while Theorem~\ref{thm:global-1} establishes a condition on the continuation of the solution, it is insensitive to the size and shape of the initial data. Therefore, similar to establishing a small mass estimate for chemotaxis models, in Section~\ref{section:dynamics} we proceed to establish a uniform in time bound for the \(L^P\) norms of the solution, depending on the smallness condition on the initial data. To achieve this, first observe the following dynamical property of a concentration functional for each species.

	\begin{theorem}[Dynamical property]\label{thm:conc-dynam}
		Suppose \(\nabla K \in L^Q(\R^d)\) for some \(Q \in (1,\infty)^{N \times N}\) and \(u_0 \in L^1 \cap L^P(\R^d)\) such that \(u_0 \geq 0\) almost everywhere, for some \(P \in [2,\infty)^N\) such that \(q_{ij}' \leq p_j\) for every \(i,\, j \in \{1,\dots,N\}\). For each species \(u_i\) with corresponding exponent \(p_i\) consider the concentration functional
		\begin{equation}\label{defn:rho_(intro)}
			\rho_i(t) := \left(\norm{u_i(t)}_{L^{p_i}}/M_i\right)^{p'_i/d}, \quad M_{i} := \norm{u_{0i}}_{L^1}
		\end{equation}
        and the operator \(\Phi : [0,\infty)^N \to [0,\infty)^N\), defined component-wise as
		\begin{equation}\label{defn:phi_intro}
			\Phi_i(\rho) :=  \frac{A_i p_i}{D_i} \sum_{j=1}^N M_j \norm{\nabla K_{ij}}_{L^{q_{ij}}} \rho_j^{d/q_{ij}},
		\end{equation}
		where \(A_i > 0\) depends on \(d\). If \(\rho_i(t)\) satisfies \(\rho_i(t) \geq \Phi_i(\rho(t))\) for some time \(t > 0\), then
		\begin{equation}\label{ineq:rho-prime-bound}
			\rho_i'(t) \leq - \frac{B_i D_i}{p_i} \rho_i(t)^2 (\rho_i(t) - \Phi_i(\rho(t)))
		\end{equation}
		where \(B_i > 0\) also depends on \(d\).
	\end{theorem}
	
	For \(N = 1\), we use direct methods to classify choices of initial data and kernels for which Theorem~\ref{thm:conc-dynam} implies a uniform-in-time \(L^p\) bound on the solutions. 

	For general \(N \in \N\), we apply a cone contraction--expansion fixed point theorem to find positive fixed points of \(\Phi\) and conclude the following smallness condition on the concentration functional \(\rho(t)\) of the solution.

	\begin{theorem}[Small concentration condition]\label{thm:small-conc}
		Suppose \(\nabla K \in L^Q(\R^d)\) for some \(Q \in (1,\infty)^{N \times N}\) and \(u_0 \in L^1 \cap L^P(\R^d)\) such that \(u_0 \geq 0\) almost everywhere, for some \(P \in [2,\infty)^N\) such that \(q_{ij}' \leq p_j\) for every \(i,\, j \in \{1,\dots,N\}\).
		Further suppose that \(Q \in (1,d)^{N \times N} \cup (d,\infty)^{N \times N}\) and the perception graph of~\eqref{eqn:agg_diff} is \emph{strongly connected}, that is, for any pair of nodes \((i,j)\) there exists a path from node \(i\) to node \(j\). Then, there exists some \(\rho^* = (\rho^*_1,\dots,\rho^*_N) \in (0,\infty)^N\) such that \(\rho(t_0) \leq \rho^*\) for some \(t_0 \geq 0\) implies that \(\rho(t) \leq \rho^*\) for all \(t \geq t_0\), where \(\rho(t)\) is defined in~\eqref{defn:rho_(intro)}.
	\end{theorem}

	Section~\ref{section:dynamics} is then concluded by considering examples of systems for \(N = 2\) in which the corresponding \(\Phi\) has a positive fixed point despite not satisfying the assumptions in Theorem~\ref{thm:small-conc}, which motivates extensions of the result.
    
	Section~\ref{section:numerics} applies the theory developed in this work to an example family of kernels, supplemented by numerical approximations of solutions to~\eqref{eqn:agg_diff}. For a single-species system we illustrate how Theorem~\ref{thm:conc-dynam} implies an upper bound for the peak concentration of the solution that depends on the regularity of the kernel. For a two-species system, we present example pairs of kernels $K_{12}$, $K_{21}$, where $K_{12}$ may be unbounded at the origin, for which the corresponding solutions to~\eqref{eqn:agg_diff} have a bounded rate of change of the energy. Motivated by these examples, we then consider how the regularity conditions on the kernels may be generalised.
 
\section{Local well-posedness}\label{section:LWP}

    By considering mild solutions of~\eqref{eqn:agg_diff}, as defined in Definition~\ref{def:mild}, we first establish existence of solutions in \(C([0,T];L^1 \cap L^P(\R^d))\) via a contraction argument on the operator \(\Psi^K[\cdot;u_0]\). Similar approaches can be found in~\cite{karch2011blow,cozzi2025global,giuntaLocalGlobalExistence2022}.

	Consider \(v \in L^\infty([0,T];L^1 \cap L^P(\R^d))\) for some \(T > 0\) and \(u_0 \in L^1 \cap L^P(\R^d)\) for some \(P \in [1,\infty)^N\). For brevity, denote \(\Psi^K[v;u_0](t) = S(t)u_0 - B^K[v,v](t)\), where
	\begin{equation}
		S(t)u_0 := (S_1(t)u_{10}, \dots, S_N(t)u_{N0})
	\end{equation} 
	and \(B^K[v,v] = (B^K_1[v,v],\dots,B^K_N[v,v])\) is defined component-wise to be 
	\begin{equation}\label{defn:B_shorthand}
		B^K_i[v,v](t) = \int_0^t \nabla S_i(t-\tau) \cdot \A_i^K [v(\tau),v(\tau)] \rmd \tau,
	\end{equation}
	where, for \(\varphi,\, \psi \in L^1 \cap L^P(\R^d)\), \(\A^K[\varphi,\psi] = (\A^K_1[\varphi,\psi],\dots,\A^K_N[\varphi,\psi])\) is defined component-wise as
	\begin{equation}\label{defn:advec_op}
		\A_i^K[\varphi,\psi] := \varphi_i \nabla F_i[\psi] = \sum_{j=1}^{N} \varphi_i (\nabla K_{ij} \ast \psi_j),
	\end{equation}
	where \(\nabla K \in L^Q(\R^d)\) for some \(Q \in (1,\infty]^{N \times N}\).
    To identify \(P \in [1,\infty)^N\) for which the operator \(\Psi^K[\cdot,u_0]\) is a well-defined function over the space \(L^\infty([0,T];L^1 \cap L^P(\R^d))\) given kernels \(K\), it is sufficient to demonstrate that \(\A\) is a bounded bilinear operator from \((L^1 \cap L^P(\R^d))^2\) to \(L^1 \cap L^P(\R^d)\) when, for each \(j \in \{1,\dots,N\}\), 
	\begin{equation}\label{cond:P}
			\max_{1 \leq i \leq N} \{q'_{ij}\} \leq p_j.
	\end{equation} 
    
	\begin{lemma}\label{lem:advec_op_bbd}
		Suppose \(\nabla K \in L^Q(\R^d)\) for some \(Q \in (1,\infty]^{N \times N}\). Then \(\A^K\) as defined in~\eqref{defn:advec_op} is a bounded bilinear operator from \((L^1 \cap L^P(\R^d))^2\) to \(L^1 \cap L^P(\R^d)\) for any \(P \in [1,\infty)\) that satisfies~\eqref{cond:P}. That is, for any \(\varphi,\, \psi \in L^1 \cap L^P(\R^d)\),
		\begin{equation}\label{ineq:advec_op_bound}
			\norm{\A^K[\varphi,\psi]}_{L^1 \cap L^P} \leq \norm{\nabla K}_{L^Q} \norm{\varphi}_{L^1 \cap L^P} \norm{\psi}_{L^1 \cap L^P}.
		\end{equation} 
	\end{lemma}

	\begin{proof}
		For each \(i \in \{1,\dots,N\}\) and \(r \in [1,p_i]\), apply H\"older's and Young's inequalities to observe that for any \(\varphi,\, \psi \in L^1 \cap L^P(\R^d)\),
		\begin{equation}\label{ineq:trip_prod}
			\begin{aligned}
				\norm{\A_i^K[\varphi,\psi]}_{L^r} 
				& \leq \norm{\varphi_i}_{L^r} \sum_{j=1}^N \norm{\nabla K_{ij} \ast \psi_j}_{L^\infty} \\
				& \leq \norm{\varphi_i}_{L^r} \sum_{j=1}^N \norm{\nabla K_{ij}}_{L^{q_{ij}}} \norm{\psi_j}_{L^{q_{ij}'}}.
			\end{aligned}
		\end{equation}
		From the assumption that \(p_j \geq q_{ij}'\) for every \(j\), it follows from~\eqref{ineq:lp_int} and Remark~\ref{rem:bessel_int} that
		\begin{equation}\label{ineq:psi_l^p_int}
			\norm{\psi_j}_{L^{q_{ij}'}} \leq \norm{\psi_j}_{L^1} + \norm{\psi_j}_{L^{p_j}} = \norm{\psi_j}_{L^1 \cap L^{p_i}} \leq \norm{\psi}_{L^1 \cap L^P}.
		\end{equation}
		By substituting~\eqref{ineq:psi_l^p_int} into~\eqref{ineq:trip_prod}, it can be concluded that
		\begin{align}
			\norm{\A^K[\varphi,\psi]}_{L^1 \cap L^P} & = \sum_{i=1}^N (\norm{\A_i^K[\varphi,\psi]}_{L^1} + \norm{\A_i^K[\varphi,\psi]}_{L^{p_i}}) \\
			& \leq \norm{\psi}_{L^1 \cap L^P} \sum_{i=1}^N (\norm{\varphi_i}_{L^1} + \norm{\varphi_i}_{L^{p_i}}) \sum_{j=1}^N \norm{\nabla K_{ij}}_{L^{q_{ij}}} \\
			& \leq \norm{\nabla K}_{L^Q} \norm{\varphi}_{L^1 \cap L^P} \norm{\psi}_{L^1 \cap L^P}. 
		\end{align} 
	\end{proof}

    \begin{remark}\label{rem:broader-K}
        Note that the assumption \(\nabla K \in L^Q(\R^d)\) can be relaxed to the condition \(\nabla K \in L^{Q_1}(\R^d) + L^{Q_2}(\R^d)\) for \(Q_1 = (q_{ij1})_{ij},\, Q_2 = (q_{ij2})_{ij} \in (1,\infty]^{N \times N}\). In this case the condition on \(P\) is instead
        \begin{equation}\label{cond:P-v2}
            p_j \geq \max \{q'_{ijk} : i \in \{1,\dots,N\}, k \in \{1,2\}\}
        \end{equation}
        for every \(j \in \{1,\dots,N\}\) and~\eqref{ineq:trip_prod} becomes
        \begin{equation}\label{ineq:trip_prod-v2}
            \norm{\A^K[\varphi,\psi]}_{L^1 \cap L^P} \leq \norm{\nabla K}_{L^{Q_1} + L^{Q_2}} \norm{\varphi}_{L^1 \cap L^P} \norm{\psi}_{L^1 \cap L^P}.
        \end{equation}
        Due to this, all the following results in Section~\ref{section:LWP} follow analogously for the family of kernels \(\nabla K \in L^{Q_1}(\R^d) + L^{Q_2}(\R^d)\) when \(P\) satisfies~\eqref{cond:P-v2}. This relaxed condition allows for power law kernels of the form \(K_{ij}(x) \propto \abs{x}^{\alpha}\), where \(\alpha \in (-d,1)\) and the logarithm kernel, \(K_{ij}(x) \propto \log(\abs{x})\). This is of particular interest since the Newtonian potential is a power law kernel with exponent \(\alpha = 2-d\) when \(d \geq 3\) and a logarithm kernel when \(d = 2\). More generally, this relaxed condition allows us to consider \(\nabla K = \nabla K_1 + \nabla K_2\), where \(\nabla K_1\) has a singularity at the origin and \(\nabla K_2\) has slow decay as \(\abs{x} \to \infty\). 
    \end{remark}
    
    \begin{lemma}\label{lem:duhamel_bound_lp}
		Suppose \(\nabla K \in L^Q(\R^d)\) for some \(Q \in (1,\infty]^{N \times N}\) and that \(u,\, v \in L^\infty([0,T];L^1 \cap L^P(\R^d))\) for some \(P \in [1,\infty)^N\) satisfying~\eqref{cond:P}. Then, for each \(t \in [0,T]\),
		\begin{equation}\label{ineq:L1_Lp_B_bound}
			\begin{aligned}
				\norm{B^K[u,v](t)}_{L^1 \cap L^P} &\leq \frac{C_1}{\sqrt{D_{\min}}} \norm{\nabla K}_{L^Q} \int_0^t (t-\tau)^{-1/2} \norm{u(\tau)}_{L^1 \cap L^P} \norm{v(\tau)}_{L^1 \cap L^P} \rmd \tau
			\end{aligned}
		\end{equation}
		where \(D_{\min} = \min_i \{D_i\}\) and \(C_1\) is given in~\eqref{ineq:semigrp_deriv_bund}. Moreover,
		\begin{equation}\label{ineq:L1_Lp_B_crude_bound}
			\norm{B^K[u,v]}_{L^\infty_T(L^1 \cap L^P)} \leq C_B \sqrt{T} \norm{u}_{L^\infty_T(L^1 \cap L^P)} \norm{v}_{L^\infty_T(L^1 \cap L^P)},
		\end{equation}
		where \(C_B = 2 C_1 \norm{\nabla K}_{L^Q} / \sqrt{D_{\min}}\).
	\end{lemma}

	\begin{proof}
		By properties of the Bochner integral and~\eqref{ineq:semigrp_deriv_bund}, it follows that
		\begin{equation}\label{ineq:B^K_i_first_bound}
			\begin{aligned}
				&\norm{B^K[u,v](t)}_{L^1 \cap L^P} = \sum_{i=1}^N (\norm{B^K_i[u,v](t)}_{L^1} + \norm{B^K_i[u,v](t)}_{L^{p_i}}) \\
				& \quad \leq \sum_{i=1}^N \int_0^t \norm{\nabla S_{i}(t-\tau) \cdot \A_i^K[u(\tau),v(\tau)]}_{L^1} + \norm{\nabla S_{i}(t-\tau) \cdot \A_i^K[u(\tau),v(\tau)]}_{L^{p_i}}\rmd \tau  \\
				& \quad \leq \sum_{i=1}^N \int_0^t C_1 (D_i (t-\tau))^{-{1/2}} (\norm{\A_i^K[u(\tau),v(\tau)]}_{L^1} + \norm{\A_i^K[u(\tau),v(\tau)]}_{L^{p_i}}) \rmd \tau \\
				& \quad \leq \frac{C_1}{\sqrt{D_{\min}}} \int_0^t (t-\tau)^{-{1/2}} \norm{\A^K[u(\tau),v(\tau)]}_{L^1 \cap L^P} \rmd \tau.
			\end{aligned}
		\end{equation}
		Inequality \eqref{ineq:L1_Lp_B_bound} then follows from~\eqref{ineq:advec_op_bound}. Lastly, the inequality~\eqref{ineq:L1_Lp_B_crude_bound} follows immediately from~\eqref{ineq:L1_Lp_B_bound}.
	\end{proof}

	\begin{lemma}\label{cor:cty_in_time}
        Suppose \(v \in C([0,T];L^1 \cap L^P(\R^d))\) for some \(T > 0\) and \(u_0 \in L^1 \cap L^P(\R^d)\). Then \(\Psi^K[v;u_0] \in C([0,T];L^1 \cap L^P(\R^d))\).
    \end{lemma}

    \begin{proof}
        To deduce the continuity of \(\Psi^K[v,u_0](t)\) in the \(L^1 \cap L^P\) sense, we establish that both the semigroup term \(S_i(t)u_0\) and the bilinear term \(B^K[v,v](t)\) are each continuous in time. By the continuity and semigroup property of \(S_i(t)\) from Lemma~\ref{lem:heat_semi_props} and the inequality~\eqref{ineq:semigrp_contraction}, it follows that \(\lim_{t \to t_0} \norm{(S_i(t) - S_i(t_0))u_0}_{L^1 \cap L^{p_i}} = 0.\) For the continuity of \(B^K_i[v,v](t)\) first observe that, for \(t_1 := \min \{t, t_0\} \), \(t_2 := \max\{t, t_0\}\) and some \(i \in \{1,\dots,N\}\),
		\begin{equation}\label{eqn:B-time-diff}
			\begin{aligned}
				&B^K_i[v,v](t_2) - B^K_i[v,v](t_1) \\
				& \quad = \int_0^{t_2} \nabla S_i(\tau) \cdot f_i(t_2 - \tau) \rmd \tau - \int_0^{t_1} \nabla S_i(\tau) \cdot f_i(t_1 - \tau) \rmd \tau  \\
				& \quad = \int_{t_1}^{t_2} \nabla S_i(\tau) \cdot f_i(t_2 - \tau)\rmd \tau + \int_{0}^{t_1} \nabla S_i(\tau) \cdot (f_i(t_2-\tau) - f_i(t_1 - \tau)) \rmd \tau, \\
			\end{aligned}
		\end{equation}
		where \(f_i(\tau) := \A_i^K[v(\tau),v(\tau)]\). It then follows that
		\begin{equation}\label{L1LP-B-time-diff}
			\begin{aligned}
				\norm{B^K(t_2) -& B^K(t_1)}_{L^1 \cap L^P} \leq \sum_{i=1}^N \Big(\Big\|\int_{t_1}^{t_2} \nabla S_i(\tau) \cdot f_i(t_2 - \tau)\rmd \tau \Big\|_{L^1 \cap L^{p_i}}\Big) \\
								&+ \sum_{i=1}^N \Big(\Big\| \int_{0}^{t_1} \nabla S_i(\tau) \cdot (f_i(t_2-\tau) - f_i(t_1 - \tau)) \rmd \tau \Big\|_{L^1 \cap L^{p_i}}\Big).
			\end{aligned}
		\end{equation}
		By inequality~\eqref{ineq:semigrp_deriv_bund} and Lemma~\ref{lem:advec_op_bbd} the first term on the right hand side of~\eqref{L1LP-B-time-diff} can be bounded as
		\begin{equation}
			\begin{aligned}
				\sum_{i=1}^N \Big\|\int_{t_1}^{t_2} \nabla S_i(\tau) \cdot f_i(t_2-\tau)\rmd \tau\Big\|_{L^1 \cap L^{p_i}} 
				& \leq \sum_{i=1}^N\int_{t_1}^{t_2} C_1 (D_i \tau)^{-1/2} \norm{f_i(t_2 - \tau)}_{L^1 \cap L^{p_i}} \rmd \tau \\
				& \leq C_1 \int_{t_1}^{t_2} (D_{\min} \tau)^{-1/2} \norm{f_i(t_2 - \tau)}_{L^1 \cap L^{P}} \rmd \tau \\
				& \leq 2 C_1 \norm{\nabla K}_{L^Q} \norm{v}_{L^\infty_T(L^1 \cap L^P)}^2 \sqrt{(t_2-t_1)/D_{\min}},
			\end{aligned}
		\end{equation}
		which vanishes as \(t_2 - t_1 \to 0\). By again applying~\eqref{ineq:semigrp_deriv_bund}, observe that second term on the right hand side of~\eqref{L1LP-B-time-diff} satisfies
		\begin{equation}
			\begin{aligned}
				\sum_{i=1}^N \Big\|\int_0^{t_1} \nabla &S_i(\tau) \cdot (f_i(t_2-\tau) - f_i(t_1 - \tau)) \rmd \tau\Big\|_{L^1 \cap L^{p_i}} \\
				& \leq \sum_{i=1}^N \int_0^{t_1} C_1 (D_i \tau)^{-1/2} \norm{f_i(t_2-\tau) - f_i(t_1 - \tau)}_{L^1 \cap L^{p_i}} \rmd \tau \\
				& \leq 2C_1 \sqrt{T/D_{\min}} \sup_{0 \leq \tau \leq t_1} \{\norm{f(t_2-\tau) - f(t_1 - \tau)}_{L^1 \cap L^P}\}.
			\end{aligned}
		\end{equation}
		Now recall that \(v\) is uniformly continuous over \([0,T]\) in the \(L^1 \cap L^P\) sense and that \(\varphi \mapsto \A^K[\varphi,\varphi]\) is a continuous mapping over \(L^1 \cap L^P(\R^d)\). Hence, \(f \in C([0,T];L^1 \cap L^P(\R^d))\) and we conclude that the second term will also vanish as \(t\) tends to \(t_0\). As both terms on the right hand side of~\eqref{L1LP-B-time-diff} vanishes as \(t_2 - t_1 \to 0\), equivalently \(t \to t_0\), then \(B^K[v,v] \in C([0,T];L^1 \cap L^P(\R^d))\), as required. 
    \end{proof}

	\begin{proposition}[Local existence]\label{thm:mild_soln}
		Suppose \(\nabla K \in L^Q(\R^d)\) for some \(Q \in (1,\infty]^{N \times N}\) and \(u_0 \in L^1 \cap L^P(\R^d)\) for some \(P \in [1,\infty)^N\) that satisfies~\eqref{cond:P}. Then, for any \(\lambda \in (1,2)\), there exists a unique fixed point \(u\) of \(\Psi^K[\cdot;u_0]\) in the set
		\begin{equation}
			\Z_\lambda := \{v \in C([0,T];L^1 \cap L^P(\R^d)) : \norm{v}_{L^\infty_T(L^1 \cap L^P)} \leq \lambda \norm{u_0}_{L^1 \cap L^{P}}\}
		\end{equation}
		for all \(T \in [0,T_\lambda]\), where \(T_\lambda > 0\) satisfies
		\begin{equation}\label{bound:T}
			C_B \norm{u_0}_{L^1 \cap L^P} \sqrt{T_\lambda} \leq (\lambda-1)/\lambda^2,
		\end{equation} 
		with \(C_B\) given in~\eqref{ineq:L1_Lp_B_crude_bound}. That is, \(u\) is a mild solution to~\eqref{eqn:agg_diff} up to time \(T_\lambda\) with kernels \(K\) and initial data \(u_0\) that is bounded in \(L^\infty([0,T];L^1 \cap L^P(\R^d))\) by \(\lambda \norm{u_0}_{L^1 \cap L^P}\).
	\end{proposition}
	
	\begin{proof}
		By the Banach fixed point theorem it is sufficient to show that \(\Psi^K[\cdot;u_0]\) is a contraction mapping from \(\Z_\lambda\) to itself when \(T\) satisfies the bound in~\eqref{bound:T} and then show continuity-in-time of the solution in the \(L^1 \cap L^P\) sense.
		
		To show that \(\Psi^K[\cdot;u_0]\) is a contraction mapping for such a small \(T\), we consider \(v,w \in \Z_\lambda\). As \(B\) is bilinear,
		\begin{equation}\label{equal:B_diff}
			\Psi^K[v;u_0]-\Psi^K[w;u_0] = B^K[w,w] - B^K[v,v] = B^K[w, w - v] + B^K[w-v, v].
		\end{equation}
		Hence, after applying~\eqref{ineq:L1_Lp_B_crude_bound} to both \(B^K[w, w - v]\) and \(B^K[w-v, v]\), it can be deduced from~\eqref{equal:B_diff} that
		\begin{equation}\label{ineq:contraction}
			\begin{aligned}
				&\norm{\Psi^K[v;u_0]-\Psi^K[w;u_0]}_{L^\infty_T(L^1 \cap L^P)}\\
                &\qquad \leq C_B \sqrt{T} (\norm{w}_{L^\infty_T(L^1 \cap L^P)} +\norm{v}_{L^\infty_T(L^1 \cap L^P)}) \norm{w-v}_{L^\infty_T(L^1 \cap L^P)} \\
				&\qquad \leq C_B \sqrt{T} (2 \lambda \norm{u_0}_{L^1 \cap L^P}) \norm{v-w}_{L^\infty_T(L^1 \cap L^P)}.
			\end{aligned}
		\end{equation}
		As \(T \leq T_\lambda\) satisfies~\eqref{bound:T} and \((\lambda - 1)/\lambda^2 < 1/2\lambda\) for all \(\lambda \in (1,2)\) it holds that
		\begin{equation}
			C_B \sqrt{T} \norm{u_0}_{L^1 \cap L^P} < 1/2\lambda
		\end{equation}
		which, via~\eqref{ineq:contraction}, implies that \(\Psi^K[\cdot;u_0]\) is a contraction mapping over \(\Z_\lambda\).
		
		As the continuity in time of \(\Psi^K[v;u_0]\) for any \(v \in \Z_\lambda\) is given by Lemma~\ref{cor:cty_in_time}, it follows that \(\Psi^K[\,\cdot\,;u_0]: \Z_\lambda \to \Z_\lambda\) so long as
		\begin{equation}
			\norm{\Psi^K[v;u_0]}_{L^\infty_T(L^1 \cap L^P)} \leq \norm{S(t)u_{0}}_{L^\infty_T(L^1 \cap L^P)} + \norm{B^K[v,v]}_{L^\infty_T(L^1 \cap L^P)} \leq \lambda \norm{u_0}_{L^1 \cap L^P}.
		\end{equation}
		From~\eqref{ineq:semigrp_contraction}, it follows that \(\norm{S(t)u_{0}}_{L^\infty_T(L^1 \cap L^P)} \leq \norm{u_0}_{L^1 \cap L^P}\). Via~\eqref{ineq:L1_Lp_B_crude_bound} and the choice of \(v \in \Z_\lambda\), observe that
		\begin{equation}\label{ineq:B_u_0}
			\norm{B^K[v,v]}_{L^\infty_T(L^1 \cap L^P)} \leq C_B \sqrt{T} (\lambda \norm{u_0}_{L^1 \cap L^P})^2.
		\end{equation}
		Again, since \(T \leq T_\lambda\), it follows that \(\norm{B^K[v,v]}_{L^\infty_T(L^1 \cap L^P)} \leq (\lambda - 1) \norm{u_0}_{L^1 \cap L^p}\). 
		Therefore, 
		\begin{equation}
			\norm{\Psi^K[v;u_0]}_{L^\infty_T(L^1 \cap L^P)} \leq (1 + (\lambda - 1)) \norm{u_0}_{L^1 \cap L^p} = \lambda \norm{u_0}_{L^1 \cap L^p}
		\end{equation}
		and so \(\Psi^K[\cdot;u_0]\) maps \(\Z_\lambda\) to itself for any choice of \(T \in [0, T_\lambda]\).
	\end{proof}

    \begin{remark}\label{rmk:unique-outside-Z}
        Note that while the contraction mapping argument employed in the proof of Proposition~\ref{thm:mild_soln} ensures that the mild solution is unique in \(\mathcal{Z}_\lambda\), it is still necessary to exclude the possibility of the existence of other solutions that violate the bound \(\norm{u}_{L^\infty_T(L^1 \cap L^P)} \leq \lambda \norm{u_0}_{L^1 \cap L^P}\) for some \(T \leq T_\lambda\).
    \end{remark}

	We now consider the continuous dependence of the solution with respect to both the initial data and the choice of kernels. To achieve this we establish a stability estimate for mild solutions of~\eqref{eqn:agg_diff}, which utilises the following version of the Gr\"onwall inequality.

	\begin{lemma}[\!{\cite[Corollaries 1,\! 2]{ye2007generalized}}]\label{lem:general_gronwall}
		Suppose \(a(t)\) is a non-negative, non-decreasing function, \(b \geq 0\) and \(\varphi(t)\) is a non-negative and locally integrable function over \([0,T)\) for some \(T > 0\) such that
		\begin{equation}\label{eqn:frac_gron_criteria}
			\varphi(t) \leq a(t) + b \int_0^t (t-\tau)^{-1/2} \varphi(\tau) \rmd \tau
		\end{equation}
		for all \(t \in [0,T)\). Then, for any \(t \in [0,T)\) it follows that
		\begin{equation}
			\varphi(t) \leq a(t) E_{1/2}\left(b \sqrt{\pi t}\right),
		\end{equation}
		where \(E_{1/2}\) is the Mittag-Leffler function of parameter \(1/2\) which for any \(z \geq 0\) can be denoted via~\cite{haubold2011mittag} as
		\begin{equation}\label{def:mittag}
			E_{1/2}(z) := \exp(z^2) \left(1 + 2 \pi^{-1/2} \int_0^z \exp(-x^2)\rmd x\right).
		\end{equation} 
	\end{lemma}

	\begin{proposition}[Continuous dependence]\label{lem:cont_dep}
		Suppose that \(u\) and \(v\) are mild solutions to~\eqref{eqn:agg_diff} up to some time \(T > 0\) with kernels \(K\), \(G\) and initial datum \(u_0\), \(v_0\) respectively. Suppose further that \(\nabla K,\, \nabla G \in L^Q(\R^d)^{N \times N}\) for some \(Q \in (1,\infty]^{N \times N}\) and \(u_0,\, v_0 \in L^1 \cap L^P(\R^d)\) for some \(P \in [1,\infty)^N\) that satisfies~\eqref{cond:P}. Then, for all \(t \in [0,T]\),
		\begin{equation}\label{eqn:cty_dep_bound}
			\norm{u(t)-v(t)}_{L^1 \cap L^P} \leq \left(\norm{u_0-v_0}_{L^1 \cap L^P} + A_1 \sqrt{t} \norm{\nabla {(K-G)}}_{L^Q}  \right) E_{1/2} (A_2 \sqrt{\pi t}),
		\end{equation}
		where
		\begin{equation}\label{eqn:cty_dep_constants}
			\begin{aligned}
				A_1 &:= 2 C_1 D_{\min}^{-1/2} \norm{u}_{L^\infty_T(L^1 \cap L^P)} \norm{v}_{L^\infty_T(L^1 \cap L^P)},\\
				A_2 &:= C_1 D_{\min}^{-1/2} (\norm{\nabla K}_{L^Q} \norm{u}_{L^\infty_T(L^1 \cap L^P)} + \norm{\nabla G}_{L^Q} \norm{v}_{L^\infty_T(L^1 \cap L^P)}) ,
			\end{aligned}
		\end{equation} 
		and \(E_{1/2}\) is defined as in~\eqref{def:mittag}.
	\end{proposition}

	\begin{proof}
		Given the couples \((K,\, u_0),\, (G,\, v_0)\) consider their respective solutions \(u\) and \(v\). Denote \(h := u - v\) and observe that, for all \(t \in [0,T]\),
		\begin{equation}\label{eqn:diff_agg}
			\begin{aligned}
				h(t) 
				&= \Psi^K[u;u_0](t) - \Psi^G[v;v_0](t) \\
				&= S(t)(u_0 - v_0) - \left(B^K[u,u](t) - B^G[v,v](t)\right) \\
				&= S(t)h_{0} - \left(B^K[h,u](t) + B^{K-G}[v,u](t) + B^G[v,h](t) \right),
			\end{aligned}
		\end{equation}
		where the final equality follows from the fact that, for each \(t \in [0,T]\), the mapping \((u,K,v) \mapsto B^K[u,v](t)\) is trilinear.
		From here take the \(L^1 \cap L^P\) norm of both sides of~\eqref{eqn:diff_agg}, then apply~\eqref{ineq:semigrp_contraction} and~\eqref{ineq:L1_Lp_B_bound} to obtain that, for all \(t \in [0,T]\),
		\begin{equation}\label{ineq:cty_dep_gronable}
			\begin{aligned}
				&\norm{h(t)}_{L^1 \cap L^p} \leq \norm{h_0}_{L^1 \cap L^P} + \frac{C_1}{\sqrt{D_{\min}}} \Big\{2 \sqrt{t} \norm{\nabla (K-G)}_{L^Q} \norm{u}_{L^\infty_T(L^1 \cap L^P)} \norm{v}_{L^\infty_T(L^1 \cap L^P)} \\
				&\qquad + \left(\norm{\nabla K}_{L^Q} \norm{u}_{L^\infty_T(L^1 \cap L^P)} + \norm{\nabla G}_{L^Q} \norm{v}_{L^\infty_T(L^1 \cap L^P)}\right)
				\int_0^t (t-\tau)^{-1/2} \norm{h(\tau)}_{L^1 \cap L^P} \rmd \tau \Big\}.
			\end{aligned}
		\end{equation}
		Now notice that~\eqref{ineq:cty_dep_gronable} is of the form~\eqref{eqn:frac_gron_criteria} for \(\varphi(t) = \norm{h(t)}_{L^1 \cap L^P}\). Therefore, apply Lemma~\ref{lem:general_gronwall}, along with the continuity of \(\varphi\) at time \(T\) via Lemma~\ref{cor:cty_in_time}, to conclude that~\eqref{eqn:cty_dep_bound} holds for all \(t \in [0,T]\).
	\end{proof}
	
    Using the stability estimate~\eqref{eqn:cty_dep_bound}, we can resolve the concern raised in Remark~\ref{rmk:unique-outside-Z} by strengthening the uniqueness of the mild solution constructed in Proposition~\ref{thm:mild_soln} to the whole solution space \(C([0,T];L^1 \cap L^P(\R^d))\) for any \(T > 0\) where a solution exists.

	\begin{corollary}[Uniqueness]\label{cor:uniq-sol}
		Suppose \(u\) is a mild solution to~\eqref{eqn:agg_diff} up to some time \(T > 0\) for kernels \(\nabla K \in L^Q(\R^d)\) and initial data \(u_0 \in L^1 \cap L^P(\R^d)\) for some \(P \in [1,\infty)^N\) satisfying~\eqref{cond:P}. Then \(u\) is unique in \(C([0,T];L^1 \cap L^P(\R^d))\).
	\end{corollary}
	
	\begin{proof}
		Suppose that \(v\) is another mild solution to~\eqref{eqn:agg_diff} with kernels \(G = K\) and initial data \(v_0 = u_0\). Since \(u,\,v \in C([0,T];L^1 \cap L^P(\R^d))\), \(A_1\) and \(A_2\), as defined in~\eqref{eqn:cty_dep_constants}, are finite. Thus, it follows from~\eqref{eqn:cty_dep_bound} that \(\norm{u(t) - v(t)}_{L^1 \cap L^P} = 0\) for all \(t \in [0,T]\).
	\end{proof}

	With local well-posedness obtained for mild solutions, we can turn our focus to the continuation of the solution in time. To accomplish this, we identify two useful properties of mild solutions of~\eqref{eqn:agg_diff}. The first property is the ability to truncate the time interval for which the solution is defined and the second is the ability to construct a new mild solution by chaining together solutions so long as the end-point of one is equal to the initial data of another. Indeed, both of these properties are consequences of the following result.

	\begin{proposition}\label{lem:unique_u_decomp}
		Suppose \(u_0 \in L^1 \cap L^P(\R^d)\) for some \(P \in [1,\infty)^N\) and \(K = (K_{ij})_{i,j=1}^N\) is a collection of locally integrable kernels such that \(v \mapsto \Psi^K[v;u_0]\) is a well-defined mapping over the space \(C([0,T];L^1 \cap L^P(\R^d))\). Then, \(u\) is a mild solution of~\eqref{eqn:agg_diff} up to time \(T > 0\) with kernels \(K\) and initial data \(u_0\) if and only if, for any \(T^* \in (0,T]\), \(u\) can be written as
		\begin{equation}\label{def:mild_soln_decomp}
			u(t) = \begin{cases}
				u_A(t), &\quad t \in [0,T^*), \\
				{u_B}(t-T^*), &\quad t \in [T^*,T],
			\end{cases}
		\end{equation}
		where \(u_A\) is a mild solution of~\eqref{eqn:agg_diff} up to time \(T^*\) with kernels \(K\) and initial data \(u_0\) and \({u_B}\) is a mild solution of~\eqref{eqn:agg_diff} up to time \(T - T^*\) with kernels \(K\) and initial data \(u(T^*)\).
	\end{proposition}

	\begin{proof}
		Suppose \(u_A \in C([0,T^*];L^1 \cap L^P(\R^d))\) and \(u_B \in C([0,T - T^*];L^1 \cap L^P(\R^d))\) for some \(0 < T^* \leq T\) such that \(u_A(T^*) = u_B(0)\). Then, \(u\) as defined in~\eqref{def:mild_soln_decomp} is in \(C([0,T];L^1 \cap L^P(\R^d))\) if \(u_A(T^*) = u_B(0)\) almost everywhere. Similarly, if \(u_A\) and \(u_B\) are defined via~\eqref{def:mild_soln_decomp} given some \(u \in C([0,T];L^1 \cap L^P(\R^d))\), then \(u_A \in C([0,T^*];L^1 \cap L^P(\R^d))\), \(u_B \in C([0,T - T^*];L^1 \cap L^P(\R^d))\) and \(u_A(T^*) = u_B(0)\) almost everywhere.
		
		With continuity in time established, it follows that \(u\) is a mild solution to~\eqref{eqn:agg_diff} up to time \(T\) with initial data \(u_0\) if and only if \(u_A\) is a mild solution up to time \(T^* > 0\) with initial data \(u_0\) and \(u_B\) is a mild solution up to time \(T - T^* > 0\) with initial data \(u(T^*)\), namely if \(u_A,\, u_B\) and \(u\) satisfy
        \begin{align}
			\Psi^K[u;u_0](t) = \Psi^K[u_A;u_0](t), &\quad t \in [0,T^*], \label{eqn:fixed_point_A} \\
			\Psi^K[u;u_0](t) = \Psi^K[u_B;u(T^*)](t-T^*), &\quad t \in [T^*,T]. \label{eqn:fixed_point_B}
        \end{align}
        As \(u(\tau) = u_A(\tau)\) for all \(\tau \in [0,T^*]\),~\eqref{eqn:fixed_point_A} follows directly from the definition of \(\Psi^K\) in~\eqref{eqn:mild_solution} that \(\Psi^K[u;u_0](t) = \Psi^K[u_A;u_0](t)\) for all \(t \in [0,T^*)\). To show~\eqref{eqn:fixed_point_B}, consider the following for each \(i \in \{1,\dots,N\}\), \(t \in [0,T-T^*]\),
		\begin{equation}\label{eqn:u^*_not_as_trivial}
			\begin{aligned}
				&\Psi^K_i[u,u_0](t+T^*) \\
				& \quad = S_i(t+T^*)u_{0i} -  \int_0^{t+{T^*}} \nabla S_i(t+{T^*}-\tau) \cdot (u_i(\tau) \nabla F_i[u(\tau)]) \rmd \tau  \\
				& \quad = S_i(t)S_i(T^*)u_{0i} -  \Bigl( S_i(t) \int_0^{T^*} \nabla S_i(T^*-\tau)  \cdot (u_i(\tau) \nabla F_i[u(\tau)]) \rmd \tau  \\
				& \qquad + \int_{T^*}^{t+{T^*}} \nabla S_i(t+T^*-\tau)  \cdot (u_i(\tau) \nabla F_i[u(\tau)]) \rmd \tau\Bigr)\\
				& \quad = S_i(t) \left(\Psi^K_i[u,u_0](T^*)\right) -  \int_0^t \nabla S_i(t-\tau) \cdot ({u_B}_i(\tau) \nabla F_i[u_B(\tau)]) \rmd \tau \\
                & \quad = \Psi^K_i[{u_B};u(T^*)](t).
			\end{aligned}
		\end{equation}
		The second equality uses the semigroup property of \(S_i(t)\), 
		\begin{equation}
			\nabla S_i(t+\tau) = \nabla S_i(t) S_i(\tau) = S_i(t) \nabla S_i(\tau)
		\end{equation} 
		and the third equality uses a change of variable \(\tau \mapsto \tau - T^*\) for the second integral.
	\end{proof}

	If we can \emph{a priori} control the growth of the \(L^1 \cap L^P\) norm of the solution to remain bounded on any finite interval \([0,T]\), then we can continue the (unique) solution indefinitely via repeated applications of Proposition~\ref{thm:mild_soln} and Proposition~\ref{lem:unique_u_decomp}. Otherwise the solution may blow-up in the \(L^1 \cap L^P\) norm in finite time. This dichotomy can be summarised as follows.

	\begin{corollary}[Blow-up or continuation of the solution]\label{prop:continuation}
        Suppose the system~\eqref{eqn:agg_diff} has kernels \(\nabla K \in L^Q(\R^d)\) for some \(Q \in (1,\infty]^{N \times N}\) and initial data \(u_0 \in L^1 \cap L^P(\R^d)\), for some \(P \in [1,\infty)^N\) satisfying~\eqref{cond:P}. Then there either exists a solution \(u \in C([0,\infty);L^1 \cap L^P(\R^d))\) or there exists a solution \(u \in C([0,T);L^1 \cap L^P(\R^d))\) for some earliest time \(T > 0\) such that
		\begin{equation}\label{cond:L1LP-blowup}
			\lim_{t \to T} \norm{u(t)}_{L^1 \cap L^P} = \infty.
		\end{equation}
    \end{corollary}

\section{Regularity of the solution}\label{section:regularity}

	Under the regularity assumptions on~\eqref{eqn:agg_diff}, the equations can be treated like second order parabolic partial differential equations with nonlocal \emph{semi-linear} terms. Consequently, it is to be expected that the regularity of its solutions smooth out for positive time, which is proven below by bootstrapping the regularity over the family of spaces \(H_{R,P}^s(\R^d)\) for \(s \geq 0\) and \(R,\, P \in (1,\infty)^N\) chosen such that, for some \(Q \in (1,\infty)^{N \times N}\) and any \(j \in \{1,\dots,N\}\),
	\begin{equation}\label{cond:R-P-Q}
        r_j \leq \min_{1 \leq i \leq N} \{q'_{ij}\} \leq \max_{1 \leq i \leq N} \{q'_{ij}\} \leq p_j.
    \end{equation} 

	The following methodology builds on the approach considered in Lemmas~\ref{lem:advec_op_bbd}--\ref{cor:cty_in_time}.

    \begin{lemma}\label{lem:bi-lin_bound}
		Suppose \(\nabla K \in L^Q(\R^d)\) for some \(Q \in (1,\infty)^{N \times N}\). Then \(\A^K\) as defined in~\eqref{defn:advec_op} is a bounded bilinear operator from \((H_{R,P}^s(\R^d))^2\) to \(H_{R,P}^s(\R^d)\) for any \(s \geq 0\) and \(R,\, P \in (1,\infty)^N\) that satisfies~\eqref{cond:R-P-Q}. That is, for any \(\varphi,\, \psi \in H_{R,P}^s(\R^d)\),
		\begin{equation}\label{ineq:advec_bi-lin_bound}
			\norm{\A^K[\varphi,\psi]}_{H_{R,P}^s} \leq C_H \norm{\nabla K}_{L^Q} \norm{\varphi}_{H_{R,P}^s} \norm{\psi}_{H_{R,P}^s},
		\end{equation} 
		where \(C_H > 0\) depends on \(s\) and \(d\), as defined in Lemma~\ref{lem:tri_lin_s_bound}.
    \end{lemma}

    \begin{proof}
        By first applying Lemma~\ref{lem:tri_lin_s_bound}, followed by the \(H^{s,p}\) interpolation from Remark~\ref{rem:bessel_int} as \(r_j \leq q_{ij}' \leq p_j\) for any choice of \(i,\,j\), it can be deduced that
		\begin{align}
			\norm{\A^K[\varphi,\psi]}_{H_{R,P}^s} & = \sum_{i=1}^N \left(\norm{\A_i^K[\varphi,\psi]}_{H^{s,r_i}} + \norm{\A_i^K[\varphi,\psi]}_{H^{s,p_i}}\right) \\
			& \leq \sum_{i,\,j=1}^N C_H (\norm{\varphi_i}_{H^{s,r_i}} + \norm{\varphi_i}_{H^{s,p_i}}) \norm{\nabla K_{ij}}_{L^{q_{ij}}} \norm{\psi_j}_{H^{s,q_{ij}'}} \\
			& \leq \sum_{i,\, j=1}^N C_H \norm{\varphi}_{H^{s,r_i} \cap H^{s,p_i}} \norm{\nabla K_{ij}}_{L^{q_{ij}}} \norm{\psi_j}_{H^{s,r_j} \cap H^{s,p_j}} \\
			& \leq C_H \norm{\nabla K}_{L^Q} \norm{\varphi}_{H_{R,P}^s} \norm{\psi}_{H_{R,P}^s}. 
		\end{align}
    \end{proof}

    \begin{remark}
        Similarly to Remark~\ref{rem:broader-K}, it can be shown that the corresponding results in Section~\ref{section:regularity} for kernels of the form \(\nabla K \in L^{Q_1}(\R^d) + L^{Q_2}(\R^d)\) for \(Q_1,\, Q_2 \in (1,\infty)^{N \times N}\) hold, so long as \(R,\, P \in (1,\infty)^N\) satisfy
        \begin{equation}
            r_j \leq \min \{q'_{ijk} : i \in \{1,\dots,N\}, k \in \{1,2\}\} \leq \max \{q'_{ijk} : i \in \{1,\dots,N\}, k \in \{1,2\}\} \leq p_j
        \end{equation}
        for every \(j \in \{1,\dots,N\}\).
    \end{remark}
    
	\begin{lemma}\label{lem:mild_sol_extra_reg}
        Suppose \(\nabla K \in L^Q(\R^d)\) for some \(Q \in (1,\infty)^{N \times N}\). Further suppose that \(u \in C([0,T];H_{R,P}^s(\R^d))\) for some \(s \in [0,\infty)\) and \(R,\, P \in (1,\infty)^N\) that satisfy~\eqref{cond:R-P-Q} is a mild solution of~\eqref{eqn:agg_diff} up to time \(T > 0\) with kernels \(K\) and initial data \(u_0 \in H_{R,P}^s(\R^d)\). Then for any choice of \(\sigma \in [0,1)\), it follows that \(u \in C((0,T]; H_{R,P}^{s+\sigma}(\R^d))\).
	\end{lemma}

	\begin{proof}
		From the definition of a mild solution of~\eqref{eqn:agg_diff} recall that, for all \(t \in [0,T]\),
		\begin{equation}
			u(t) = \Psi^K[u;u_0](t) := S(t)u_0 - B^K[u,u](t).
		\end{equation} 
		Hence, it is sufficient to show that \(S(t)u_0 \in C((0,T];H_{R,P}^{s+\sigma}(\R^d))\) given that \(u_0 \in H_{R,P}^s(\R^d)\), and \(B^K[v,v] \in C([0,T];H_{R,P}^{s+\sigma}(\R^d))\), given that \(v \in C([0,T];H_{R,P}^s(\R^d))\).

		For some arbitrary \(t_0 > 0\), fix some \(\tau \in (0,t_0)\). Then, for any \(t \in [\tau,T]\), denote \(t_1 := \min \{t, t_0\} \) and \(t_2 := \max\{t, t_0\}\). Observe that, via~\eqref{ineq:snu_1},
		\begin{equation}
			\begin{aligned}
				&\norm{S(t_2)u_{0} - S(t_1)u_{0}}_{H_{R,P}^{s+\sigma}} = \sum_{i=1}^N \norm{S_i(\tau) (S_i(t_2 - \tau)u_{0i} - S_i(t_1 - \tau)u_{0i})}_{H^{s+\sigma,r_i} \cap H^{s+\sigma,p_i}} \\
				& \quad \leq \sum_{i=1}^N C_\sigma \max\{(D_i \tau)^{-\sigma/2},\, 1\} \norm{S_i(t_1 - \tau) (S(t_2-t_1)u_{0i} - u_{0i})}_{H^{s,r_i} \cap H^{s,p_i}} \\
				& \quad \leq C_\sigma \max\{(D_{\min} \tau)^{-\sigma/2},\, 1\} \norm{S(t_2-t_1)u_{0} - u_{0}}_{H_{R,P}^{s}}.
			\end{aligned} 
		\end{equation}	
		Consequently, it can be deduced that for any \(t_0 \in (0,T]\), the term \(\norm{S(t)u_{0} - S(t_0)u_{0}}_{H_{R,P}^{s+\sigma}}\) vanishes as \(t \to t_0\). 
		
		Now consider some arbitrary \(v \in C([0,T]; H_{R,P}^s(\R^d))\) and denote \(f(\tau) := \A^K[v(\tau),v(\tau)]\). Due to the continuity in time of \(v\) and the continuity of the mapping \(\varphi \mapsto \A^K[\varphi,\varphi]\) in the \(H_{R,P}^s(\R^d)\) sense over \([0,T]\), it follows that \(f \in C([0,T];H_{R,P}^s(\R^d))\). Recall~\eqref{eqn:B-time-diff} which, for \(t_1 = \min \{t, t_0\} \) and \(t_2 = \max\{t, t_0\}\) and for each \(i \in \{1,\dots,N\}\), it follows that
		\begin{equation}\label{ineq:Hsp-B-diff-bound}
			\begin{aligned}
				\norm{B^K[v,v](t_2) -& B^K[v,v](t_1)}_{H_{R,P}^{s+\sigma}} \leq \sum_{i=1}^N \Big( \Big\|\int_{t_1}^{t_2} \nabla S_i(\tau) \cdot f_i(t_2 - \tau)\rmd \tau\Big\|_{H^{s+\sigma,r_i} \cap H^{s+\sigma,p_i}} \Big) \\
				&+ \sum_{i=1}^N \Big( \Big\|\int_0^{t_1} \nabla S_i(\tau) \cdot (f_i(t_2-\tau) - f_i(t_1 - \tau)) \rmd \tau\Big\|_{H^{s+\sigma,r_i} \cap H^{s+\sigma,p_i}} \Big).
			\end{aligned}
		\end{equation}
		To bound the first term on the right hand side of~\eqref{ineq:Hsp-B-diff-bound}, apply~\eqref{ineq:snu_2} followed by the inequality \(\max \{a, b\} \leq a + b\) for any \(a,\, b \geq 0\) to obtain that
		\begin{equation}\label{ineq:s-to-s+sig-1}
			\begin{aligned}
				&\sum_{i=1}^N \Big\| \int_{t_1}^{t_2} \nabla S_i(\tau) \cdot f_i(t_2-\tau) \rmd \tau \Big\|_{H^{s+\sigma,r_i} \cap H^{s+\sigma,p_i}}\\
				& \qquad \leq \sum_{i=1}^N C_{\sigma + 1} \int_{t_1}^{t_2} \left((D_i \tau)^{-1/2} + (D_i \tau)^{-(1+\sigma)/2}\right) \norm{f_i(t_2-\tau)}_{H^{s,r_i} \cap H^{s,p_i}} \rmd \tau \\
				& \qquad \leq C_{\sigma + 1} \int_{t_1}^{t_2} \left((D_{\min} \tau)^{-1/2} + (D_{\min} \tau)^{-(1+\sigma)/2}\right) \norm{f(t_2-\tau)}_{H_{R,P}^{s}} \rmd \tau \\
				& \qquad \leq C_{\sigma + 1} \norm{\nabla K}_{L^Q} \norm{v}^2_{L^\infty_T H_{R,P}^s} \left(2 D_{\min}^{-1/2} (t_2 - t_1)^{1/2} + \frac{2}{1-\sigma} D_{\min}^{-(1+\sigma)/2} (t_2 - t_1)^{(1-\sigma)/2}\right).
			\end{aligned}
		\end{equation}
		It then follows that the right hand side of~\eqref{ineq:s-to-s+sig-1} vanishes as \(t_2 - t_1 \to 0\), since both the \((t_2 - t_1)^{1/2}\) and \((t_2 - t_1)^{(1-\sigma)/2}\) terms vanish when \(\sigma \in [0,1)\). For the second term repeat the same arguments as in~\eqref{ineq:s-to-s+sig-1} to obtain that
		\begin{equation}\label{ineq:s-to-s+sig-2}
			\begin{aligned}
				&\sum_{i=1}^N \Big\|\int_0^{t_1} \nabla S_i(\tau) \cdot (f_i(t_2-\tau) - f_i(t_1 - \tau)) \rmd \tau \Big\|_{H^{s + \sigma ,r_i} \cap H^{s+\sigma,p_i}}\\
				& \qquad \leq C_{\sigma + 1} \int_{0}^{t_1} \left((D_{\min} \tau)^{-1/2} + (D_{\min} \tau)^{-(1+\sigma)/2}\right) \norm{f_i(t_2-\tau) - f_i(t_1 - \tau)}_{H_{R,P}^{s}} \rmd \tau \\
				& \qquad \leq C_{\sigma + 1} \left(2 D_{\min}^{-1/2} T^{1/2} + \frac{2}{1-\sigma} D_{\min}^{-(1+\sigma)/2} T^{(1-\sigma)/2}\right) \sup_{0 \leq \tau \leq t_1} \norm{f(t_2- \tau) - f(t_1 - \tau)}_{H_{R,P}^s}.
			\end{aligned}
		\end{equation}
		By the uniformly continuity of \(f\) over \([0,T]\) in the \(H_{R,P}^{s}\) sense, the second term must also vanish as \(t\) approaches \(t_0\). By substituting~\eqref{ineq:s-to-s+sig-1} and~\eqref{ineq:s-to-s+sig-2} back into~\eqref{ineq:Hsp-B-diff-bound}, we can conclude that \(B^K[v,v] \in C([0,T];H_{R,P}^{s+\sigma}(\R^d))\), as required.
	\end{proof}

	
	\begin{proposition}[Regularity of the solution]\label{thm:smoothing}
		Let \(u\) be a \emph{mild} solution to~\eqref{eqn:agg_diff} up to some time \(T > 0\) with \(\nabla K \in L^Q(\R^d)\) for some \(Q \in (1,\infty)^{N \times N}\) and initial data \(u_0 \in L^{1} \cap L^P(\R^d)\) for \(P \in (1,\infty)^N\) that satisfies~\eqref{cond:P}. Then, for any \(n \geq 2\) and \(R \in (1,\infty)^N\) that satisfies~\eqref{cond:R-P-Q},
        \begin{equation}\label{cond:u-smooth}
            u \in C([0,T];L^1 \cap L^P(\R^d)) \cap C((0,T];W_{R,P}^{n}(\R^d)) \cap C^1((0,T];W_{R,P}^{n-2}(\R^d)).
        \end{equation}
	\end{proposition}

	\begin{proof}
		Fix some \(t^* \in (0,T)\) and consider some sequence \((t_n)_{n=0}^\infty\) such that \(t_0 = 0\) and \(t_n < t_{n+1} < t^*\) for all \(n \in \N\). Now for each \(n \in \N\) define \(\tau_n := t_{n+1} - t_n\) and the function \(u^{(n)}(t) := u(t+t_n)\) for \(t \in [0,T-t_n]\). As \(u\) is a mild solution to~\eqref{eqn:agg_diff} with initial data \(u_0 \in L^1 \cap L^P(\R^d)\), it follows from Proposition~\ref{lem:unique_u_decomp} that \(u^{(n)} \in C([0,T-t_n];L^1 \cap L^P(\R^d))\) is also a mild solution to~\eqref{eqn:agg_diff} up to time \(T-t_n\) with initial data \(u(t_n) \in L^1 \cap L^P(\R^d)\). By an induction argument we now want to prove that, for each \(n \in \N\),
        \begin{equation}\label{cond:u^n-reg}
            u^{(n)} \in C([0,T-t_n];H_{R,P}^{n/2}(\R^d)).
        \end{equation}
        The base case follows by observing that \(L^1 \cap L^P(\R^d) \subset H_{R,P}^0(\R^d)\) for any \(R\) that satisfies~\eqref{cond:R-P-Q}. For the induction step assume, for the sake of induction, that
        \begin{equation}
            u^{(n-1)} \in C([0,T-t_{n-1}];H_{R,P}^{(n-1)/2}(\R^d)).
        \end{equation}
        By applying Lemma~\ref{lem:mild_sol_extra_reg} for \(\sigma = 1/2\) then observing that \([\tau_{n-1},T-t_{n-1}] \subset (0,T-t_{n-1}]\), it follows that
        \begin{equation}\label{cond:u^n-1_reg}
            u^{(n-1)} \in C([\tau_{n-1},T-t_{n-1}];H_{R,P}^{n/2}(\R^d)).
        \end{equation}
        Due to the relationship \(u^{(n-1)}(t+ \tau_{n-1}) = u^{n}(t)\) for \(t \in [0,T-t_n]\),~\eqref{cond:u^n-1_reg} is equivalent to~\eqref{cond:u^n-reg}, as required. Since \(t^* > t_n\) for all \(n \in \N\) and that \(t^*\) was chosen arbitrarily, deduce that \(u \in C((0,T];W_{R,P}^n(\R^d))\) for any \(n \in \N\).
        
        Now fix some \(n \geq 2\). In order to study the regularity of \(\frac{\rmd}{\rmd t} u(t)\) for any \(t > 0\), consider
        \begin{equation}\label{def:u-tilde-reg}
            \tilde{u}(t) := u(t+t_0) \in C([0,T-t_0];W_{R,P}^n(\R^d))
        \end{equation}
        for some arbitrary \(t_0 \in (0,T]\). From the differentiability property of \(S_i(t)\) in Lemma~\ref{lem:heat_semi_props} it then follows that
        \begin{equation}\label{eqn:deriv-S_i}
            \frac{\rmd}{\rmd t}S_i(t)\tilde{u}(\tau) = D_i \Delta S_i(t)\tilde{u}(\tau).
        \end{equation}
        To find \(\frac{\rmd}{\rmd t} B[\tilde{u},\tilde{u}](t)\), first consider \(f(\tau) = \A^K[\tilde{u}(\tau),\tilde{u}(\tau)]\) for each \(\tau \in [0,T-t_0]\). Due to~\eqref{def:u-tilde-reg} and that \(\varphi \mapsto \A^K[\varphi,\varphi]\) is a continuous mapping over \(H_{R,P}^{s}(\R^d)\), it follows that \(f \in C([0,T-t_0]; W_{R,P}^n(\R^d))\). Moreover, notice \(g(\tau; t)\), defined component-wise as \(g_i(\tau; t) := \Delta S_i(t-\tau) \cdot \nabla f_i(\tau)\), is in the Bochner space \(L^1([0,t];W_{R,P}^{n-2}(\R^d))\) for any \(t \in [0,T-t_0]\). In particular, observe via~\eqref{ineq:semigrp_deriv_bund} and \(\Delta S_i(t) \cdot \nabla f_i = \nabla S_i(t) \cdot \Delta f_i\) that
		\begin{equation}
           \begin{aligned}
                \int_0^{t} \norm{g(\tau; t)}_{W_{R,P}^{n-2}}\, \rmd \tau &= \sum_{i=1}^N \int_0^{t} \norm{\nabla S_i(t-{\tau}) \cdot \Delta f_i({\tau})}_{W^{n-2,r_i} \cap W^{n-2,p_i}} \rmd \tau \\
                &\leq \sum_{i=1}^N \int_0^t C_1 (D_i (t-\tilde{\tau}))^{-1/2} \norm{\Delta f_i({\tau})}_{W^{n-2,r_i} \cap W^{n-2,p_i}} \rmd \tau \\
                &\leq 2 C_1 \sqrt{t / D_{\min}} \norm{f}_{L^\infty_{t} W_{R,P}^{n}}.
           \end{aligned}
        \end{equation}
		Therefore, by differentiating under the integral, it follows that for any \(t \in (0,T-t_0)\), \noeqref{eqn:dt_nonlin_bound}
        \begin{align}\label{eqn:dt_nonlin_bound}
			\begin{aligned}
				\frac{\rmd}{\rmd t} B_i[\tilde{u},\tilde{u}](t) &=  S_i(0) \nabla \cdot f_i(t) + \int_{0}^{t} D_i \Delta S_i(t-{\tau}) \nabla \cdot f_i(\tau) \rmd {\tau}.
			\end{aligned}
		\end{align}
		After applying Fubini's theorem and integration by parts, it holds for any compactly supported test function \(\varphi \in C^\infty(\R^d)\) that
        \begin{equation}\label{eqn:deriv-B_i}
            \int_{\R^d} \varphi(x) \int_0^t \left[\Delta S_i(t-{\tau}) \cdot \nabla f_i({\tau})\right](x)\, \rmd {\tau} \rmd x = \int_{\R^d} \varphi(x) [\Delta B_i[\tilde{u},\tilde{u}](t)](x) \rmd x.
        \end{equation}
		Therefore, by~\eqref{eqn:deriv-S_i}--\eqref{eqn:deriv-B_i} it holds that, for all \(t \in (0,T]\) and \(i \in \{1,\dots,N\}\),
        \begin{equation}\label{eqn:t-deriv-u_i}
            \begin{aligned}
                \frac{\rmd}{\rmd t} u_i(t) 
                & = \frac{\rmd}{\rmd t} S_i(t) u_{0i} - \frac{\rmd}{\rmd t} B_i[u,u](t) \\
                &= D_i \Delta \left(S_i(t) u_{0i} - B_i[u,u](t)\right) - \nabla \cdot (u_i(t) \nabla F_i[u(t)]) \\
                &= D_i \Delta u_i - \nabla \cdot \A^K_i[u(t),u(t)].
            \end{aligned}
        \end{equation}
        Since \(u \in C((0,T];W_{R,P}^n(\R^d))\) it follows that \(\Delta u \in C((0,T];W_{R,P}^{n-2} (\R^d))\) and \(\nabla \cdot \A^K_i[u,u] \in C((0,T];W_{R,P}^{n-1}(\R^d))\).
        Hence, using~\eqref{eqn:t-deriv-u_i} we conclude that \(u \in C^1((0,T];W_{R,P}^{n-2}(\R^d))\) for any \(n \geq 2\).
	\end{proof}


	With the newly established regularity acquired for the spatial derivatives of \(u\) and \(\frac{\rmd}{\rmd t} u\), we can conclude that mild solutions of~\eqref{eqn:agg_diff} are in fact classical in the sense of Definition~\ref{defn:classical-sol}.

	\begin{corollary}\label{cor:class-soln}
		Suppose \(u\) is a \emph{mild} solution to~\eqref{eqn:agg_diff} up to some time \(T > 0\) which satisfies the assumptions in Proposition~\ref{thm:smoothing}. Then, up to being redefined on a set of measure zero for each \(t \in [0,T]\), \(u\) is also a \emph{classical} solution of~\eqref{eqn:agg_diff}.
	\end{corollary}

	\begin{proof}
		Due to Proposition~\ref{thm:smoothing}, \(u\) has the regularity given in~\eqref{cond:u-smooth}. Now, for each \(i \in \{1,\dots,N\}\), choose some \(n > 2\) large enough such that \((n-2)p_i > d\). By Sobolev embedding~\cite[pp.~261,\,279--280]{evans2010} it then follows that, up to being redefined on a set of measure zero for each \(t > 0\), \(u_i(t) \in W^{n,p_i}(\R^d) \subset C_u(\R^d) \cap L^{p_i}(\R^d)\), \(\nabla u_i(t) \in W^{n-1,p_i}(\R^d) \subset C_u(\R^d) \cap L^{p_i}(\R^d)\), and \(\frac{\rmd}{\rmd t} u_i(t),\, \nabla^2 u_i(t) \in W^{n-2,\,p_i} \subset C_u(\R^d) \cap L^{p_i}(\R^d)\).
		Now, to demonstrate that \(\nabla \cdot (u_i \nabla F_i[u]) \in C_u(\R^d) \cap L^{p_i}(\R^d)\) for each \(i\), observe that
		\begin{equation}
			\nabla \cdot (u_i \nabla F_i[u]) = \nabla u_i \cdot \nabla F_i[u] + u_i \Delta F_i[u].
		\end{equation}
		Since \(u_i(t),\, \nabla u_i(t) \in C_u(\R^d) \cap L^{p_i}(\R^d)\) for each \(i \in \{1,\dots,N\}\) and \(t > 0\), it is sufficient to show that \(\nabla F_i[u],\, \Delta F_i[u] \in C_u(\R^d) \cap L^\infty(\R^d)\). From~\eqref{cond:u-smooth} we know that \(u(t) \in L^1 \cap L^P(\R^d)\) and \(\nabla u \in L^R \cap L^P(\R^d)\) for some \(R,\, P \in (1,\infty)^N\) such that \(r_j \leq q_{ij}' \leq p_j\) for every choice of \(i,\, j \in \{1,\dots,N\}\), in particular, \(u_j,\, \nabla u_j \in L^{q'_{ij}}(\R^d)\) for every \(i,\, j\). Hence, using that \(\varphi \ast \psi \in C_u(\R^d) \cap L^\infty(\R^d)\) for any \(\varphi \in L^p(\R^d),\, \psi \in L^{p'}(\R^d)\), \(p \in [1,\infty]\), we can conclude that 
		\begin{equation}\label{cond:F[u]}
			\begin{aligned}
				\nabla F_i[u] = \sum_{j=1}^N \nabla K_{ij} \ast u_j &\in C_u(\R^d) \cap L^\infty(\R^d), \\
				\Delta F_i[u] = \sum_{j=1}^N \nabla K_{ij} \ast \nabla u_j &\in C_u(\R^d) \cap L^\infty(\R^d),
			\end{aligned}
		\end{equation}
		where \(\nabla f \ast \nabla g\) denotes \(\sum_{k=1}^d \partial_k f \ast \partial_k g\). Finally, since \(\frac{\rmd}{\rmd t} u_i(t),\, \Delta u_i\), and \(\nabla \cdot (u_i \nabla F_i[u])\) are all (uniformly) continuous functions for every \(t > 0\), it follows from~\eqref{eqn:t-deriv-u_i} that \(u_i\) must solve~\eqref{eqn:agg_diff} pointwise.
	\end{proof}

	\begin{remark}\label{rem:vanish-at-infty}
		Note that \(\varphi \in C_u(\R^d) \cap L^{p}(\R^d)\) for any \(p \in [1,\infty)\) implies that \(\varphi \in L^\infty(\R^d)\) and that \(\varphi(x) \to 0\) as \(\abs{x} \to \infty\). Therefore, for any \(t > 0\), \(u(t),\, \frac{\rmd}{\rmd t} u(t),\, \nabla u(t)\), and \(\Delta u(t)\) all vanish as \(\abs{x} \to \infty\).
	\end{remark}

\section{Non-negativity and conservation of mass}\label{section:preserved}

	Preservation of non-negativity and mass of solutions are proven here by utilising similar arguments to those presented in~\cite{li2010wellposedness}. Due to Corollary~\ref{cor:class-soln}, solutions to~\eqref{eqn:agg_diff} will always be assumed to be classical in the sense of Definition~\ref{defn:classical-sol} throughout the rest of this work.

	\begin{lemma}[\!\!{\cite[Lemma 2.5]{li2010wellposedness}}]\label{lem:AD_positive}
		Let \(\Omega_T := (0,T] \times \R^d\) for some fixed \(T > 0\). Suppose \(\rho \in C^{1,2}_{t,x}({\Omega}_T) \cap C_{t,x}(\bar{\Omega}_T) \cap L^p_{t,x}(\Omega_T)\) for some \(p \in [1,\infty)\), \(g \in C(\R^d)\), and \(f \in C^{0,1}_{t,x}(\Omega_T)\). Suppose further that the following conditions hold:
		\begin{enumerate}
			\item \(\rho\) satisfies the following advection-diffusion equation up to time \(T\)
			\begin{equation}\label{eqn:linear_advec_diff} \tag{LAD}
				\partial_t \rho + \nabla \cdot (f\rho) = \nu \Delta \rho,
			\end{equation}
			for some \(\nu \geq 0\) with the initial condition \(\rho(0,\cdot) = g\).
			\item \(\mathrm{sup}_{\bar{\Omega}_T} \abs{\rho} + \mathrm{sup}_{\Omega_T} (\abs{\partial_t \rho} + \abs{\nabla \rho} + \abs{\Delta \rho}) \) is finite.
			\item \(g \geq 0\) and there exists some \(L > 0\) such that \(\mathrm{sup}_{\Omega_T} \abs{\nabla \cdot (f(t,x))} \leq L\).
		\end{enumerate}
		Then \(\rho \geq 0\) for all \((t,x) \in \bar{\Omega}_T\).
	\end{lemma}
		

	\begin{proposition}[Non-negativity]\label{prop:nonnegativity}
		Suppose \(u\) is a solution to~\eqref{eqn:agg_diff} up to time \(T > 0\) with \(\nabla K \in L^Q(\R^d)\) for some \(Q \in (1,\infty]^{N \times N}\) and initial data \(u_0 \in L^1 \cap L^P(\R^d)\) from some \(P \in (1,\infty)^N\) that satisfies~\eqref{cond:P}. If \(u_{0} \geq 0\) almost everywhere, then \(u_i(t) \geq 0\) for each \(i\) and for all \(t \in (0,T]\).
	\end{proposition}
    \begin{proof}
		The proof falls into two parts: first it is demonstrated that if \(u(t_0) \geq 0\) for some \(t_0 > 0\), then \(u(t) \geq 0\) for all \(t \in [t_0,T]\), after which it is established that if \(u_{0i} \geq 0\) almost everywhere, then there exists some \(t_0 > 0\) for which \(u(t) \geq 0\) for all \(t \in (0,t_0]\). 
		
		Suppose that \(u_i(t_0) \geq 0\) almost everywhere for each \(i \in\{1,\dots,N\}\) for some given \(t_0 > 0\). By Proposition~\ref{lem:unique_u_decomp} and Corollary~\ref{cor:class-soln}, it follows that \(\tilde{u}(t) := u(t+t_0)\) is a solution of~\eqref{eqn:agg_diff} up to time \(T-t_0\) with initial data \(\tilde{u}_0 = u(t_0)\). For each \(i \in \{1,\dots,N\}\), define
		\begin{equation}
			\rho_i := \tilde{u}_i,\quad f_i := \nabla F_i[\tilde{u}], \quad g_i := \tilde{u}_i(t_0).
		\end{equation}
		From Corollary~\ref{cor:class-soln},~\eqref{cond:F[u]} and Remark~\ref{rem:vanish-at-infty} it follows that \(\rho_i,\, f_i\) and \(g_i\) satisfies the regularity conditions in Lemma~\ref{lem:AD_positive} and solves~\eqref{eqn:linear_advec_diff} for \(\nu = D_i\) and for all \(t > 0\). Thus, from Lemma~\ref{lem:AD_positive}, conclude that \({u}_i(t) = \rho_i(t-t_0) \geq 0\) for all \(t \in [t_0,T]\) and \(i \in \{1,\dots,N\}\).

        Now suppose that \(u_{0i} \geq 0\) almost everywhere and consider a family of mollified initial data \(\{u_0^\vareps\}_{\vareps > 0}\) such that, for each \(\vareps > 0\), \(i \in \{1,\dots,N\}\), \(u_{0i}^\vareps := \mu_\vareps \ast u_{0i}\), where \(\mu_\vareps(x) := \vareps^{-d} \mu(x/\vareps)\) for some non-negative, compactly supported, smooth bump function \(\mu \in C^\infty(\R^d),\, \int_{\R^d} \mu = 1\). Since \(\norm{u_0^\vareps}_{L^1 \cap L^P} \leq \norm{u_0}_{L^1 \cap L^P}\) for all \(\vareps > 0\), it follows from Proposition~\ref{thm:mild_soln} that a (mild) solution to~\eqref{eqn:agg_diff} \(u^\vareps\) exists up to some time \(t_0(\norm{u}_{L^1 \cap L^P}) > 0\) independent of \(\vareps > 0\) with initial data \(u_{0i}^\vareps \in C(\R^d)\) for each \(i \in \{1,\dots,N\}\). By then applying Lemma~\ref{lem:AD_positive} using similar arguments used in the first part of the proof, it then follows that \(u_i^\vareps(t) \geq 0\) for all \(t \in (0,t_0]\), \(i \in \{1,\dots,N\}\). It remains to show that the non-negativity is preserved as \(\vareps \to 0\). From the continuous dependence result, Proposition~\ref{lem:cont_dep}, there exists some constant \(C\) such that, for all \(t \in [0,t_0]\),
        \begin{equation}\label{ineq:non-neg-eps-lim}
            \lim_{\vareps \to 0} \norm{u(t) - u^\vareps(t)}_{L^1 \cap L^P} \leq C \lim_{\vareps \to 0} \norm{u_0 - u^\vareps_0}_{L^1 \cap L^P} = 0.
        \end{equation}
        Therefore, as \(u(t)\) is a continuous function for any \(t > 0\), there cannot exist some \(\delta > 0\) and location \(x \in \R^d\) such that \(u(t,x) < -\delta\). If such a point did exist, then there would exist some neighbourhood around \(x\) of non-zero measure \(\mathcal{N}_x\) such that \(u(t,y) < -\delta\) for all \(y \in \mathcal{N}_x\). However, since \(u^\vareps(t) \geq 0\), that would imply that, for some \(\eta > 0\), \(\norm{u(t)-u^\vareps(t)}_{L^1} > \eta\) for all \(\vareps > 0\), a contradiction to~\eqref{ineq:non-neg-eps-lim}.
    \end{proof}
	
	\begin{proposition}[Conservation of mass]\label{prop:cons_mass}
		Suppose \(u\) is the solution to~\eqref{eqn:agg_diff} up to time \(T > 0\) with \(\nabla K \in L^Q(\R^d)\) for some \(Q \in (1,\infty)^{N \times N}\) and \(u_0 \in L^1 \cap L^P(\R^d),\, u_{0} \geq 0\) almost everywhere, from some \(P \in (1,\infty)^N\) that satisfies~\eqref{cond:P}. Then, for each \(i = \{1,\dots,N\}\), \(t \in [0,T]\),
		\begin{equation}
			\int_{\R^d} u_i(t) \rmd x = \int_{\R^d} u_{0i} \rmd x.
		\end{equation}
	\end{proposition}
	
	\begin{proof}
		As \(u(t) \geq 0\) for all \(t > 0\), due to Proposition~\ref{prop:nonnegativity}, and \(u \in C([0,T];L^1(\R^d))^N\), it follows that 
		\begin{equation}
			\int_{\R^d} u_{0i} \rmd x = \norm{u_{0i}}_{L^1} = \lim_{t \to 0} \norm{u_i(t)}_{L^1} = \lim_{t \to 0} \int_{\R^d} u_i(t) \rmd x.
		\end{equation}	
		Hence, to conclude conservation of mass, it is sufficient to demonstrate that
        \begin{equation}\label{eqn:mass-zero-deriv}
            \frac{\rmd}{\rmd t} \int_{\R^d} u_i(t) = 0
        \end{equation}
        for each \(i \in \{1,\dots,N\}\) and \(t > 0\). Consider a smooth compactly supported cut-off function \(\phi \in C^\infty(\R^d)\) such that \(\phi(x) = 1\) for all \(\abs{x} \leq 1\) and \(\phi(x) = 0\) for all \(\abs{x} \geq 2\). As \(u(t)\) is smooth for all \(t > 0\) via Proposition~\ref{thm:smoothing}, test the \(i\)-th component of~\eqref{eqn:agg_diff} with \(\phi_R := \phi(\cdot/R)\) for some arbitrary \(R > 0\) in order to obtain
		\begin{equation}
			\frac{\rmd}{\rmd t} \int_{\R^d} u_i \phi_R \rmd x =  \int_{\R^d} (D_i \Delta u_i) \phi_R \rmd x - \int_{\R^d} (\nabla \cdot (u_i \nabla F_i[u])) \phi_R \rmd x.
		\end{equation}
		We then can apply integration by parts to observe that
		\begin{align}
			\abs{\frac{\rmd}{\rmd t} \int_{\R^d} u_i \phi_R \rmd x} & = \abs{D_i \int_{\R^d} u_i \Delta \phi_R \rmd x +  \int_{\R^d} (u_i \nabla F_i[u]) \cdot \nabla \phi_R \rmd x} \nonumber \\
			& \leq D_i \norm{\Delta \phi_R}_{L^\infty} \norm{u_i}_{L^1} + \sum_{j=1}^N \norm{\nabla \phi_R}_{L^\infty} \norm{u_i \nabla K_{ij} \ast u_j}_{L^1}\nonumber \\
			& \leq \frac{D_i}{R^2} \norm{\Delta \phi}_{L^\infty} \norm{u_i}_{L^1} + \frac{N}{R} \norm{u_i}_{L^{1}} \norm{\nabla \phi}_{L^\infty} \max_j \{\norm{\nabla K_{ij}}_{L^{q_{ij}}} \norm{u_j}_{L^{q'_{ij}}}\}.
		\end{align}
		Since \(u_j \in C([0,T];L^1 \cap L^{p_j}) \subset C([0,T];L^{q_{ij}'})\) for \(p_j \geq \max_{i} \{q'_{ij}\}\), both the terms above can be bounded by
		\begin{equation}
			\abs{\frac{\rmd}{\rmd t} \int_{\R^d} u_i \phi_R \rmd x} \leq C(T) \left(\frac{1}{R} + \frac{1}{R^2}\right),
		\end{equation}
		for some \(C(T) > 0\) independent of \(t\). Letting \(R \to \infty\) then implies~\eqref{eqn:mass-zero-deriv} for each \(i\), as required.
	\end{proof}

\section{Global existence}\label{section:global}
	
    The objective of this section is to show that, given sufficient regularity of the solution and some appropriate \(P \in [1,\infty)^N\), the following energy functional for the system 
    \begin{equation}\label{defn:E_P}
            \mathcal{E}_P(t) := \sum_{i=1}^N \frac{1}{2(p_i-1)} \int_{\R^d} u_i(t)^{p_i} \rmd x
    \end{equation}
    is bounded over any time interval \([0,T]\). With such a bound, along with conservation of mass, it follows that \(\norm{u(t)}_{L^1 \cap L^P}\) is also bounded over \([0,T]\), from which we can continue the solution indefinitely in time via Corollary~\ref{prop:continuation}. In order to demonstrate an \emph{a priori} bound for \(\mathcal{E}_P\), we first consider the following lemmas.
    
    \begin{lemma}\label{lem:energy-deriv}
        Suppose that, for each \(i \in \{1,\dots,N\}\), the non-negative function \(u_i : [0,T] \to L^1 \cap L^{p_i}(\R^d)\) for some \(p_i \in [2,\infty)\) satisfies the regularity condition 
        \begin{equation}\label{cond:u-reg}
            \begin{aligned}
                u_i \in C([0,T];L^1 \cap L^{p_i}(\R^d)) \cap C^1((0,T);L^{p_i}(\R^d)).
            \end{aligned}
        \end{equation}
        Then it follows that \(u_i^{p_i} \in C^1((0,T);L^1(\R^d))\) for each \(i \in \{1,\dots,N\}\) and \(\mathcal{E}_{P} \in C([0,T]) \cap C^1((0,T))\). In particular, the derivative of \(\mathcal{E}_P\) for each \(t \in (0,T)\) is given as 
        \begin{equation}
            \frac{\rmd}{\rmd t} \mathcal{E}_P(t) = \sum_{i=1}^N \frac{p_i}{2 (p_i-1)} \int_{R^d} u_i^{p_i-1}(t) \partial_t u_i(t) \rmd x.
        \end{equation}
    \end{lemma}

    \begin{proof}
        By the chain rule it holds that, for any \(t \in (0,T)\), \(\partial_t u_i(t)^{p_i} = p_i u_i(t)^{p_i-1} \partial_t u_i(t)\). Since \(u_i^{p_i-1} \in C([0,T];L^{p_i'}(\R^d))\), \(\partial_t u_i \in C((0,T);L^{p_i}(\R^d))\) it follows from H\"older's inequality that \(u_i^{p_i} \in C^1((0,T);L^1(\R^d))\). By recalling the definition of \(C^1((0,T);L^1(\R^d))\), deduce for any \(i \in \{1,\dots,N\}\) and  \(t \in (0,T)\) that
        \begin{equation}
            \begin{aligned}
                &\left|\frac{\rmd}{\rmd t} \int_{\R^d} u_i(t)^{p_i} \rmd x - \int_{\R^d} \partial_t u_i(t)^{p_i} \rmd x\right| \\
                & \quad = \lim_{h \to 0} \left| \int_{\R^d} \frac{u_i(t+h)^{p_i} - u_i(t)^{p_i}}{h} - \partial_t u_i(t)^{p_i} \rmd x\right| \\ 
                & \quad \leq \lim_{h \to 0} \left\|\frac{u_i(t+h)^{p_i} - u_i(t)^{p_i}}{h} - \partial_t u_i(t)^{p_i}\right\|_{L^1} = 0.
            \end{aligned}
        \end{equation}
        That is, the following relationship is well-defined for all \(t \in (0,T)\)
        \begin{equation}\label{eqn:drag_out_diff}
            \frac{\rmd}{\rmd t} \mathcal{E}_P(t) = \sum_{i=1}^N \frac{1}{2(p_i-1)} \frac{\rmd}{\rmd t} \int_{\R^d} u_i(t)^{p_i} \rmd x = \sum_{i=1}^N \frac{p_i}{2 (p_i-1)} \int_{R^d} u_i(t)^{p_i-1} \partial_t u_i(t) \rmd x. 
        \end{equation}
    \end{proof}

	\begin{lemma}\label{lem:p-nash}
		Suppose \(p \in [2, \infty)\), \(\varphi \in L^1(\R^d)\) is non-negative and \(\nabla \varphi^{p/2} \in L^2(\R^d)\). Then
		\begin{equation}\label{ineq:p-nash}
			\norm{\varphi}_{L^p}^{p/2 + p'/d} \leq C_{\mathcal{N}} \norm{\varphi}_{L^1}^{p'/d} \norm{\nabla \varphi^{p/2}}_{L^2},
		\end{equation}
		where \(C_{\mathcal{N}} > 0\) is the constant from Nash's inequality,~\cite{nash1958continuity}.
	\end{lemma}


	\begin{proof}
		Since~\eqref{ineq:p-nash} holds trivially for \(\varphi = 0\) almost everywhere, assume that \(\norm{\varphi}_{L^p} > 0\).

		When \(p = 2\),~\eqref{ineq:p-nash} reduces to the standard Nash's inequality: for \(f \in L^1 \cap W^{1,2}(\R^d)\),
		\begin{equation}\label{ineq:nash's}
			\norm{f}_{L^2}^{1+2/d} \leq C_\mathcal{N} \norm{f}_{L^1}^{2/d} \norm{\nabla f}_{L^2},
		\end{equation}
		where \(C_\mathcal{N} > 0\) depends on \(d\).
		Now suppose that \(p \in (2,\infty)\). As \(\norm{\varphi}_{L^{p}} = \norm{\varphi^{p/2}}_{L^{2}}^{2/p}\), we can apply~\eqref{ineq:nash's}, for \(f = \varphi^{p/2}\) to deduce that
		\begin{equation}\label{ineq:p_nash_1}
			\norm{\varphi}_{L^p}^{p/2 + p/d} = \norm{\varphi^{p/2}}_{L^2}^{1+2/d} \leq C_\mathcal{N} \norm{\varphi^{p/2}}_{L^1}^{2/d} \norm{\nabla \varphi^{p/2}}_{L^2}.
		\end{equation}
		Now observe that, via \(L^p\) interpolation in Lemma~\ref{lem:lp_int}, it follows for any \(p \in [2,\infty)\) that
        \begin{equation}\label{ineq:p_nash_2}
            \norm{\varphi^{p/2}}_{L^1}^{2/d} = \norm{\varphi}_{L^{p/2}}^{p/d} \leq \norm{\varphi}_{L^1}^{p/d(p-1)} \norm{\varphi}_{L^p}^{p(p-2)/d(p-1)}.
        \end{equation} 
        Substituting~\eqref{ineq:p_nash_2} into~\eqref{ineq:p_nash_1} yields
		\begin{equation}\label{ineq:p_nash_3}
			\norm{\varphi}_{L^p}^{p/2 + p/d} \leq C_\mathcal{N} \norm{\varphi}_{L^1}^{p/d(p-1)} \norm{\varphi}_{L^p}^{p(p-2)/d(p-1)} \norm{\nabla \varphi^{p/2}}_{L^2}.
		\end{equation}
		Hence,~\eqref{ineq:p-nash} can be concluded after dividing both sides of~\eqref{ineq:p_nash_3} by \(\norm{\varphi}_{L^p}^{p(p-2)/d(p-1)}\).
	\end{proof}	

	\begin{lemma}
		Suppose \(x,\, y,\, z \geq 0\) and \(\eta > 0\). Then, for any \(\alpha,\beta \geq 0\) such that \(\alpha + \beta < 1\),
		\begin{equation}\label{ineq:young_prod}
			xy^{\alpha}z^\beta \leq \eta^{1-1/\delta} x^{1/\delta} + \eta (y + z),
		\end{equation}
        where \(\delta = 1 - (\alpha+\beta)\).
	\end{lemma}

	\begin{proof}
		Recall Young's product inequality, which states that, for any \(x,\,y,\,z > 0\) and \(\lambda_1,\, \lambda_2,\, \lambda_3 \geq 0\) such that \(\lambda_1 + \lambda_2 + \lambda_3 = 1\), then
        \begin{equation}\label{ineq:3-young}
            x^{\lambda_1} y^{\lambda_2} z^{\lambda_3} \leq \lambda_1 x + \lambda_2 y + \lambda_3 z.
        \end{equation}
        By choosing \(\delta = 1 - (\alpha + \beta)\), it then follows from~\eqref{ineq:3-young} that
		\begin{equation}
			xy^{\alpha} z^\beta = \eta ((x/\eta)^{1/\delta})^\delta y^{\alpha}z^\beta \leq \eta^{1-1/\delta} \delta x^{1/\delta} + \eta (\alpha y + \beta z),
		\end{equation} 
		which yields~\eqref{ineq:young_prod} since \(\delta,\, \alpha,\, \beta \leq 1\).
	\end{proof}

    With these lemmas in mind, we now establish an \emph{a priori} bound on \(\mathcal{E}_P\), given a certain relationship between \(P\) and \(Q\). 

	\begin{proposition}[A priori energy bound]\label{prop:L^p_est}
		Suppose \(u\) is a solution to~\eqref{eqn:agg_diff} up to time \(T > 0\) with kernels \(\nabla K \in L^Q(\R^d)\) for some \(Q \in (1,\infty)^{N \times N}\) and initial data \(u_0 \in L^1 \cap L^P(\R^d)\), \(P \in [2,\infty)^N\), satisfying \(u_0 \geq 0\) almost everywhere and
		\begin{equation}\label{defn:system_mass}
			(\norm{u_{01}}_{L^1},\dots,\norm{u_{0N}}_{L^1}) = M
		\end{equation}
		for some \(M \in (0,\infty)^N\). 
		If \(P\) and \(Q\) satisfy the relations~\eqref{cond:P} and
		\begin{equation}\label{ineq:r_ij}
			\frac{d(p_i-1) + 2}{d(p_j-1) + 2} \leq \frac{q_{ij}}{d}
		\end{equation}
		for all \(i,\, j \in \{1,\dots, N\}\) such that \(\nabla K_{ij}\) is non-trivial, then \(\mathcal{E}_P(t)\) as defined in~\eqref{defn:E_P} satisfies the bound
		\begin{equation}\label{eqn:global_bound_p} 
			\mathcal{E}_P(t) \leq \mathcal{C} t + \mathcal{E}_P(0)
		\end{equation}
		for all \(t > 0\), for some \(\mathcal{C} > 0\) which depends on \(P,\, Q,\, d,\, N,\,M,\, D_{\min}\) and \(\nabla K\).
	\end{proposition}
    
	\begin{proof}
		Recall from Proposition~\ref{thm:smoothing} and Corollary~\ref{cor:class-soln} that \(u\) satisfies the regularity condition~\eqref{cond:u-reg} in Lemma~\ref{lem:energy-deriv}. Therefore, to deduce~\eqref{eqn:global_bound_p}, it is sufficient to show that \(\frac{\rmd}{\rmd t}\mathcal{E}_P(t) \leq \mathcal{C}\) for all \(t > 0\).
        
		By testing each \(i\)-th component equation~\eqref{eqn:agg_diff} with \(\frac{p_i}{2(p_i - 1)} u_i^{{p_i}-1}\), applying integration by parts and the identity \(\nabla \varphi^{p/2} = \frac{p}{2} \varphi^{(p-2)/2} \nabla \varphi\), observe that 
        \begin{equation}\label{eqn:lp_rearrange}
            \begin{aligned}
                &\frac{p_i}{2(p_i-1)} \int_{\R^d} u_i^{p_i-1} {\partial_t} u_i \rmd x \\
                & \qquad = \frac{p_i}{2(p_i-1)} \left( D_i \int_{\R^d} u_i^{{p_i}-1} \Delta u_i \rmd x - \int_{\R^d} u_i^{{p_i}-1} \nabla \cdot (u_i \nabla F_i[u]) \rmd x \right) \\
    			& \qquad = - \frac{p_i}{2} \left( D_i \int_{\R^d} u_i^{{p_i}-2} \abs{\nabla u_i}^2 \rmd x -  \int_{\R^d} (u_i^{{p_i}-2} \nabla u_i) \cdot (u_i \nabla F_i[u]) \rmd x \right) \\
    			& \qquad = - \frac{2 D_i}{p_i}  \int_{\R^d} \abs{\nabla u_i^{{p_i}/2}}^2 \rmd x + \int_{\R^d} u_i^{{p_i}/2} \nabla u_i^{{p_i}/2} \cdot \nabla F_i[u] \rmd x.
            \end{aligned}
        \end{equation}
        By summing~\eqref{eqn:lp_rearrange} over \(i\), recalling~\eqref{eqn:drag_out_diff} and expanding \(F_i[u] := \sum_{j=1}^N K_{ij} \ast u_j\), we obtain
		\begin{equation}\label{ineq:lp_tested}
			\frac{\rmd}{\rmd t}\mathcal{E}_P(t) = \sum_{i,\,j=1}^N \int_{\R^d}  w_i \nabla w_i \cdot \left(\nabla K_{ij} \ast u_j\right) \rmd x - \sum_{i=1}^N \frac{2D_i}{{p_i}} \norm{\nabla w_i}_{L^2}^2,
		\end{equation}
		where \(w_i := u_i^{{p_i}/2}\). After two applications of Hölder's inequality it follows that
        \noeqref{ineq:RHS_1}
		\begin{equation}\label{ineq:RHS_1}
			\begin{aligned}
			\sum_{i,\,j=1}^N \int_{\R^d} \nabla w_i \cdot w_i \left(\nabla K_{ij} \ast u_j\right) \rmd x \leq \sum_{i,\, j=1}^N \norm{\nabla w_i}_{L^2} \norm{w_i}_{L^{2}} \norm{\nabla K_{ij} \ast u_j}_{L^\infty}.
			\end{aligned}
		\end{equation}

		For each choice of \(i,\, j \in \{1,\dots,N\}\), it is then sufficient to establish that
		\begin{equation}\label{ineq:applied_young}
			\norm{w_i}_{L^{2}} \norm{\nabla w_i}_{L^2} \norm{\nabla K_{ij} \ast u_j}_{L^\infty} \leq \mathcal{C}_{ij} + \frac{D_{\min}}{N p_{\max}}  \left(\norm{\nabla w_i}_{L^2}^2 + \norm{\nabla w_j}_{L^2}^2 \right),
		\end{equation}
        where \(\mathcal{C}_{ij} > 0\) depends on \(p_i,\, p_j,\, p_{\max},\, q_{ij},\, d,\, N,\, D_{\min},\, \nabla K_{ij},\, M_i\) and \(M_j\). Indeed, from~\eqref{ineq:lp_tested}--\eqref{ineq:applied_young} it can then be concluded that
		\begin{equation}\label{ineq:i_dependent}
			\begin{aligned}
			    \frac{\rmd}{\rmd t}\mathcal{E}_P(t) & \leq \sum_{i,\, j=1}^N \mathcal{C}_{ij} + \sum_{i,\, j=1}^N \frac{D_{\min}}{N p_{\max}}  \left(\norm{\nabla w_i}_{L^2}^2 + \norm{\nabla w_j}_{L^2}^2 \right) - \sum_{i=1}^N \frac{2D_i}{{p_i}} \norm{\nabla w_i}_{L^2}^2. \\
                & = \mathcal{C} +  \sum_{i=1}^N \left(\frac{2 D_{\min}}{p_{\max}} - \frac{2D_i}{{p_i}} \right) \norm{\nabla w_i}_{L^2}^2 \\
                & \leq \mathcal{C},
			\end{aligned}
		\end{equation}
        where \(\mathcal{C} := \sum_{i,\, j=1}^N \mathcal{C}_{ij}\).

		If \(\nabla K_{ij}\) is trivial, then~\eqref{ineq:applied_young} follows immediately. To demonstrate~\eqref{ineq:applied_young} in the case that \(\nabla K_{ij}\) is non-trivial, begin by recalling that \(\norm{u_{i}(t)}_{L^1} = M_i\) for each \(i \in \{1,\dots, N\}\) due to the nonnegativity and conservation of mass of solutions, as proven in Section~\ref{section:preserved}. Also note that \(\norm{u_i}_{L^{p_i}}^{p_i} = \norm{w_i}_{L^2}^2\). Consequently, for \(c_i := d(p_i - 1) + 2\), Lemma~\ref{lem:p-nash} yields
		\begin{equation}\label{ineq:nash-w}
			\norm{w_i}_{L^2} \leq C_\Nash^{1 - 2/c_i} M_i^{p_i/c_i} \norm{\nabla w_i}_{L^2}^{1 - 2/c_i}.
		\end{equation}
		
		To appropriately bound the term \(\norm{\nabla K_{ij} \ast u_j}_{L^\infty}\) such that~\eqref{ineq:applied_young} holds, consider the cases \(c_i/c_j < q_{ij}/d\) and \(c_i/c_j = q_{ij}/d\) separately.

		\emph{Case 1:} \(c_i/c_j < q_{ij}/d\). Since \(q_{ij}' \leq p_j\) from~\eqref{cond:P}, apply Young's convolution inequality and \(L^p\) interpolation to observe that
		\begin{equation}\label{ineq:conv_bound}
			\norm{\nabla K_{ij} \ast u_j}_{L^\infty} \leq \norm{\nabla K_{ij}}_{L^{q_{ij}}} \norm{u_j}_{L^{q'_{ij}}} \leq \norm{\nabla K_{ij}}_{L^{q_{ij}}} M_j^{1-p_j'/q_{ij}} \norm{u_j}_{L^{p_j}}^{p_j'/q_{ij}}.
		\end{equation}
		From the observation that \(1/(p_j-1) = d/(c_j-2)\) and~\eqref{ineq:nash-w}, it then follows that
		\begin{equation}\label{ineq:apply-nash-wj}
			\norm{u_j}_{L^{p_j}}^{p_j'/q_{ij}} = \norm{w_j}_{L^2}^{2d/q_{ij}(c_j-2)} \leq C_\Nash^{2d/c_j q_{ij}} M_j^{2p_j'/c_j q_{ij}} \norm{\nabla w_j}_{L^2}^{2d/c_j q_{ij}}.
		\end{equation}
		Hence,~\eqref{ineq:nash-w},~\eqref{ineq:conv_bound} and~\eqref{ineq:apply-nash-wj} yields
		\begin{equation}\label{ineq:first-nash}
			\norm{w_i}_{L^2} \norm{\nabla w_i}_{L^2} \norm{\nabla K_{ij} \ast u_j}_{L^\infty} \leq A_{ij} \norm{\nabla K_{ij}}_{L^{q_{ij}}} \norm{\nabla w_i}_{L^2}^{2(1-1/c_i)} \norm{\nabla w_j}_{L^2}^{2d/c_j q_{ij}},
		\end{equation} 
		where \(A_{ij}\) is given as
        \begin{equation}
            A_{ij} := C_\Nash^{1 + 2(d/c_jq_{ij} - 1/c_i)}  M_i^{p_i/c_i} M_j^{1 - (p_j'/q_{ij}) (1 - 2/c_j)}.
        \end{equation}
        Since the assumption that \(c_i/c_j < q_{ij}/d\) is equivalent to \((1 - 1/c_i) + d/c_jq_{ij} < 1\), apply via~\eqref{ineq:young_prod} to the right hand side of~\eqref{ineq:first-nash} for \(x = A_{ij} \norm{\nabla K_{ij}}_{L^{q_{ij}}}\), \(y = \norm{\nabla w_i}_{L^2}^2\), \(z = \norm{\nabla w_j}_{L^2}^2\), and \(\eta = D_{\min}/N p_{\max}\), to conclude~\eqref{ineq:applied_young}.

		\emph{Case 2:} \(c_i/c_j = q_{ij}/d\). Let \(\vareps_{ij} = D_{\min}/(2 A_{ij} N p_{\max} + 1)\). By the dominated convergence theorem, there exists some \(L_{ij}(\vareps_{ij}, \nabla K_{ij}) > 0\) such that
		\begin{equation}
			\norm{\chi_{\{\abs{\nabla K_{ij}} \leq L_{ij}\}} \nabla K_{ij}}_{L^{\infty}} \leq L_{ij},\quad \text{and}\quad \norm{\chi_{\{\abs{\nabla K_{ij}} > L_{ij}\}} \nabla K_{ij}}_{L^{q_{ij}}} \leq \vareps_{ij}.
		\end{equation}
		From the linearity of convolutions and Young's convolution inequality it then follows that
		\begin{equation}\label{ineq:K_splitting}
            \begin{aligned}
                \norm{w_i}_{L^2} \norm{\nabla w_i}_{L^2} \norm{\nabla K_{ij} \ast u_j}_{L^{\infty}} &\leq L_{ij} M_j \norm{w_i}_{L^2} \norm{\nabla w_i}_{L^2} + \vareps_{ij} \norm{u_j}_{L^{q'_{ij}}} \norm{w_i}_{L^2} \norm{\nabla w_i}_{L^2}.
            \end{aligned}
		\end{equation}
		Applying~\eqref{ineq:nash-w} followed by \(a x^{1-1/c_i} \leq \eta^{1-c_i} a^{c_i} + \eta x\) for \(a = C_\Nash^{1-2/c_i} M_i^{p_i/c_i} M_j L_{ij}\), \(x = \norm{\nabla w_i}_{L^2}^2\), and \(\eta = D_{\min}/2Np_{\max}\), we obtain that 
		\begin{equation}\label{ineq:ay-case-2.1}
			L_{ij} M_j \norm{w_i}_{L^2} \norm{\nabla w_i}_{L^2} \leq \mathcal{C}_{ij} + \frac{D_{\min}}{2 N p_{\max}} \norm{\nabla w_i}_{L^2}^2
		\end{equation}
		for some \(\mathcal{C}_{ij} > 0\) that depends on \(p_i,\, p_j,\, p_{\max},\, q_{ij},\, d,\, N,\, D_{\min},\, \nabla K_{ij},\, M_i\) and \(M_j\).

		From~\eqref{ineq:nash-w} and~\eqref{ineq:apply-nash-wj} it similarly follows that
		\begin{equation}\label{ineq:case-2.2}
			\vareps_{ij} \norm{u_j}_{L^{q'_{ij}}} \norm{w_i}_{L^2} \norm{\nabla w_i}_{L^2} \leq \vareps_{ij} A_{ij} \norm{\nabla w_i}_{L^2}^{2(1-1/c_i)} \norm{\nabla w_j}_{L^2}^{2d/c_j q_{ij}}.
		\end{equation}
		Since \(c_i/c_j = q_{ij}/d\) is equivalent to \((1-1/c_i) + d/c_jq_{ij} = 1\), deduce that
        \begin{equation}
            \norm{\nabla w_i}_{L^2}^{2(1-1/c_i)} \norm{\nabla w_j}_{L^2}^{2d/c_j q_{ij}} \leq \norm{\nabla w_i}_{L^2}^2 + \norm{\nabla w_j}_{L^2}^2.
        \end{equation}
        Hence, along with the choice of \(\varepsilon_{ij}\),~\eqref{ineq:case-2.2} implies that
		\begin{equation}\label{ineq:ay-case-2.2}
			\vareps_{ij} \norm{u_j}_{L^{q'_{ij}}} \norm{w_i}_{L^2} \norm{\nabla w_i}_{L^2} \leq \frac{D_{\min}}{2 N p_{\max}} (\norm{\nabla w_i}_{L^2}^2 + \norm{\nabla w_j}_{L^2}^2).
		\end{equation}

		After substituting~\eqref{ineq:ay-case-2.1} and~\eqref{ineq:ay-case-2.2} back into~\eqref{ineq:K_splitting}, conclude~\eqref{ineq:applied_young}.		
	\end{proof}

     \begin{remark}
        If the kernels are instead assumed to be of the form \(\nabla K \in L^{Q_1}(\R^d) + L^{Q_2}(\R^d)\), then~\eqref{ineq:r_ij} becomes 
        \begin{equation}\label{ineq:r_ij-v2}
			\frac{d(p_i-1) + 2}{d(p_j-1) + 2} \leq \frac{\min\{q_{ij1}, q_{ij2}\}}{d},
		\end{equation}
        since then \(c_i/c_j \leq q_{ijk}/d\) holds for both \(k \in \{1,2\}\). That is, the existence of the uniform-in-time bound for \(\frac{\rmd}{\rmd t}\mathcal{E}_P\) can be proven, so long as \(P\), \(Q_1\) and \(Q_2\) satisfy~\eqref{ineq:r_ij-v2} for any \(i,\, j\) for which \(\nabla K_{ij}\) is non-trivial.
    \end{remark}
    
    Theorem~\ref{thm:global-1} is then a consequence of Proposition~\ref{prop:L^p_est} together with the following lemma.
	
	\begin{lemma}\label{lem:int-cycle-rij}
		Suppose \(\nabla K \in L^Q(\R^d)\), where \(Q\) satisfies
		\begin{equation}\label{ineq:reg-cycle}
			\left(q_{i_1 i_2} \cdots q_{i_n i_1}\right)^{1/n} \geq d
		\end{equation}
		for every interaction cycle, then there exists some \(P \in [2,\infty)^N\) such that \(P\) and \(Q\) satisfy~\eqref{cond:P} and 
		\begin{equation}\label{ineq:r_ij-again}
				\frac{d(p_i-1) + 2}{d(p_j-1) + 2} \leq \frac{q_{ij}}{d}.
		\end{equation}
		for every \(i,\, j \in \{1,\dots,N\}\) for which \(\nabla K_{ij}\) is non-trivial. 
	\end{lemma}
	
	\begin{proof}
        Consider the transformations \(x_i = -\log(d(p_i-1)+2)\) for each \(i \in \{1,\dots,N\}\) and \(w_{ij} = \log q_{ij} - \log d\) for every \(i,\,j \in \{1,\dots,N\}\) for which \(\nabla K_{ij}\) is non-trivial. It then follows that~\eqref{ineq:r_ij-again} is equivalent to 
		\begin{equation}\label{ineq:x-w}
			x_j - x_i \leq w_{ij},
		\end{equation}
		for each \(i,\, j\) for which \(\nabla K_{ij}\) is non-trivial. The existence of such a \(x = (x_1,\dots,x_N) \in \R^N\) that satisfies~\eqref{ineq:x-w}, given that \(\sum_{k=1}^n w_{i_k i_{k+1}} \geq 0\) for every cycle, is then given by a standard constructive argument, as seen in~\cite[Theorem 24.9]{cormen2009}. 
		
		The construction is given as follows. First, consider a constraint graph with respect to the \(w_{ij}\) terms in which each edge from node \(i\) to node \(j\) on the perception graph is weighted by \(w_{ij}\). After that, add source node, along with edges of zero weight from the source node to nodes \(1, \dots, N\). By then applying the Bellman--Ford algorithm\cite[p.~651]{cormen2009} to find the shortest path from the source node to nodes \(1, \dots, N\), each \(x_i\) is then given as the weight of the shortest path from the source node to node \(i\). 
        
        To verify that our candidate choice of \(x\) satisfies~\eqref{ineq:x-w} for any \(i,\, j\) for which \(\nabla K_{ij}\) is non-trivial, consider the path given by appending the edge from node \(i\) to node \(j\) to the shortest path from the source node to node \(i\). Since this is a path from the source node to node \(j\) with a weight of \(x_i + w_{ij}\), we know by definition of \(x_j\) that \(x_j \leq x_i + w_{ij}\), as required.
		
		Now given \(x\) from the above construction, note that \(y_i = x_i - a\) also satisfies~\eqref{ineq:x-w} for any \(a \in \R\). Consequently, any \(P = (p_1,\dots,p_N)\) of the form
		\begin{equation}\label{eqn:p-from-x}
			p_i = \frac{\exp(a-x_i) - 2}{d} + 1,
		\end{equation}
		satisfies~\eqref{ineq:r_ij-again} for any \(a \in \R\) and \(i,\, j\) for which \(\nabla K_{ij}\) is non-trivial. Since \(p_i \to \infty\) as \(a \to \infty\) for every \(i \in \{1,\dots,N\}\), \(a\) can be chosen large enough so that \(\max\{q'_{ij},\, 2\} \leq p_j\) for every \(i,\, j \in \{1,\dots,N\}\).
	\end{proof}

    \begin{remark}
        Note that~\eqref{ineq:reg-cycle} must be a necessary condition of \(Q\) since the condition~\eqref{ineq:r_ij-again} implies that
		\begin{equation}
			\prod_{k=1}^{n} \frac{q_{i_k i_{k+1}}}{d} \geq \prod_{k=1}^{n} \frac{d(p_{i_k}-1)+2}{d(p_{i_{k+1}}-1)+2} = 1, \quad i_{n+1} = i_1,
		\end{equation} 
		holds for every interaction cycle.
    \end{remark}

	\begin{example}
		To illustrate how a \(P\) is chosen given some \(Q\) satisfying the geometric cycle condition~\eqref{ineq:reg-cycle}, recall the example of a system of~\eqref{eqn:agg_diff} given in Examples~\ref{ex:perp-graph} and~\ref{ex:good-cycles}. The corresponding constraint graph with source node, labelled \(\mathcal{O}\) is given by Figure~\ref{fig:constraint-graph}.

		\begin{figure}[ht]
			\centering
			\begin{tikzpicture}[state/.append style={minimum size=6mm, inner sep=1pt, thick}]
				\node[state] (0) at (-4, 0) {$\mathcal{O}$};
				\node[state] (1) at (-2, 1) {$1$};
				\node[state] (2) at (0, 0) {$2$};
				\node[state] (3) at (-2, -1) {$3$};
				\node[state] (4) at (1.75, 0) {$4$};
				
				\path[->, >=angle 90]
					(1) edge [thick, bend left=30] node [anchor = south] {$w_{12}$} (2)
					(2) edge [thick, bend left=30] node [anchor = north] {$w_{21}$} (1)
					(2) edge [thick, bend left=30] node [anchor = north west] {$w_{23}$} (3)
					(2) edge [thick] node [anchor = south] {$w_{24}$} (4)
					(3) edge [thick, bend left=30] node [anchor = east] {$w_{31}$} (1)
					(4) edge [thick, out = -45, in = 45, looseness=12] node [anchor = west] {$w_{44}$} (4)
					(0) edge [dashed, bend left=30] node [anchor = north] {$0$} (1)
					(0) edge [dashed, out = 75, in = 90, looseness = 1.4] node [anchor = south] {$0$} (2)
					(0) edge [dashed, bend right=30] node [anchor = south] {$0$} (3)
					(0) edge [dashed, out = -75, in = -90, looseness = 0.9] node [anchor = north] {$0$} (4);
			\end{tikzpicture}			
			\caption{The constraint graph corresponding to the system given by~\eqref{def:4-sys-kernels} with weights $w_{ij} = \log q_{ij} - \log d$.} 
			\label{fig:constraint-graph}
		\end{figure}

		Recall that we suppose \(d = 4\) and the regularities of each kernel are given by 
		\begin{equation}\label{eqn:Q}
			Q = \begin{pmatrix}
				- & 12 & - & - \\
				3 & - & 2 & 2 \\
				9 & - & - & - \\
				- & - & - & 5
			\end{pmatrix}.
		\end{equation}
		By then applying the Bellman--Ford algorithm, since \(Q\) satisfies~\eqref{ineq:reg-cycle} for every interaction cycle, or manually checking every path, it follows that each \(x_i\) is given as \(x_1 = \log(3) - \log(4),\, x_2 = 0,\) and \(x_3 = x_4 = \log(2) - \log(4)\). Thus, via~\eqref{eqn:p-from-x}, choose \(P\) of the form
		\begin{equation}
			p_1 = \frac{2b + 3}{6}, \quad p_2 = \frac{b + 2}{4}, \quad p_3 = p_4 = \frac{b+1}{2}.
		\end{equation}
		It can then be verified that, for any \(b \geq 6\), \(P\) and \(Q\) satisfy the conditions in Proposition~\ref{prop:L^p_est}.
	\end{example}

\section{Dynamical properties}\label{section:dynamics}
    
    For families of kernels \(\nabla K \in L^Q(\R^d)\) that do not satisfy the geometric mean condition~\eqref{ineq:reg-cycle} for every interaction cycle, a global \emph{a priori} bound can still be attained if there is a sufficient smallness condition on the ratios between \(\norm{u_{0i}}_{L^{p_i}}\) and \(\norm{u_{0i}}_{L^1}\) for each \(i \in \{1,\dots,N\}\). Throughout this section, we assume that~\eqref{eqn:agg_diff} cannot be decoupled into two distinct systems, since in that case the two systems can be studied separately.

    To begin, we establish a state-dependent differential inequality for solutions to~\eqref{eqn:agg_diff}.

    \begin{proposition}\label{prop:global-bound}
        Suppose \(u\) is a solution of~\eqref{eqn:agg_diff} up to some time \(T > 0\) with kernels \(\nabla K \in L^Q(\R^d)\) for some \(Q \in (1,\infty)^{N \times N}\) and initial data \(u_0 \in L^1 \cap L^P(\R^d)\), such that \(u_0 \geq 0\) almost everywhere, for \(P \in [2,\infty)^N\) that satisfies~\eqref{cond:P}. Now consider \(\rho \in C^1([0,T];\R^N)\) such that, for \(i \in \{1,\dots,N\}\),
        \begin{equation}\label{defn:rho}
            \rho_i(t) := \left(\norm{u_i(t)}_{L^{p_i}}/M_i\right)^{p'_i/d},
        \end{equation}
        where \(M_i := \norm{u_{0i}}_{L^1} > 0\). Then, \(\rho_i\) is non-increasing when 
        \begin{equation}\label{cond:rho_i}
            \rho_i(t) \geq \frac{p_i C_\Nash}{2 D_i} \sum_{j=1}^N M_j \norm{\nabla K_{ij}}_{L^{q_{ij}}} \rho_j(t)^{d/q_{ij}} =: \Phi_i(\rho(t)).
        \end{equation}
        Furthermore, for \(t \in (0,T]\) such that \(\rho_i(t)\) satisfies~\eqref{cond:rho_i}, then the following holds
        \begin{equation}\label{ineq:rho_deriv}
            \frac{\rmd}{\rmd t} \rho_i(t) \leq - C_i \rho_i(t)^2 \big( \rho_i(t) - \Phi_i(\rho(t)) \big),
        \end{equation}
        where \(C_i := 4 D_i/(d p_i C_\Nash^2)\).
    \end{proposition}

    \begin{proof}
        By repeating the arguments presented in the proof for Proposition~\ref{prop:L^p_est} up to~\eqref{eqn:lp_rearrange}, it follows that, for \(w_i := u_i^{p_i/2}\),
        \begin{equation}\label{eqn:msgb-tested}
            \frac{1}{2(p_i-1)} \frac{\rmd}{\rmd t} \norm{w_i}_{L^2}^2 = - \frac{2 D_i}{p_i} \norm{\nabla w_i}_{L^2}^2 + \sum_{j=1}^N \int_{\R^d} \nabla w_i \cdot w_i \left( \nabla K_{ij} \ast u_j \right) \rmd x.
        \end{equation}
        From here it is sufficient to bound both the right hand terms. By re-arranging the Nash-type inequality~\eqref{ineq:p-nash} it follows that
        \begin{equation}\label{ineq:msgb-diff-bound}
            \norm{\nabla w_i}_{L^2} \geq C_\Nash^{-1} \norm{u_i}_{L^1}^{-p_i'/d} \norm{u_i}_{L^{p_i}}^{p_i/2 + p_i'/d} = C_\Nash^{-1} \norm{w_i}_{L^2} \rho_i. 
        \end{equation}
        By H\"older's inequality, Young's convolution inequality and \(L^p\) interpolation it then follows that, for each \(j \in \{1,\dots,N\}\),
        \begin{equation}\label{ineq:msgb-agg-bound}
            \begin{aligned}
                \int_{\R^d} \nabla w_i \cdot w_i \left( \nabla K_{ij} \ast u_j \right) \rmd x 
                & \leq \norm{\nabla w_i}_{L^2} \norm{w_i}_{L^2} \norm{\nabla K_{ij}}_{L^{q_{ij}}} \norm{u_j}_{L^1}^{1-p_j'/q_{ij}} \norm{u_j}_{L^{p_j}}^{p_j'/q_{ij}} \\&= \norm{\nabla w_i}_{L^2} \norm{w_i}_{L^2} \norm{\nabla K_{ij}}_{L^{q_{ij}}} M_j \rho_j^{d/q_{ij}} ,
            \end{aligned}
        \end{equation}
        where, from nonnegativity and conservation of mass, \(\norm{u_j}_{L^1} = M_j\) for each \(j \in \{1,\dots,N\}\). Substituting the bounds~\eqref{ineq:msgb-diff-bound} and~\eqref{ineq:msgb-agg-bound} into~\eqref{eqn:msgb-tested} then yields
        \begin{equation}\label{ineq:msgb-deriv-first}
            \frac{1}{2(p_i-1)} \frac{\rmd}{\rmd t} \norm{w_i}_{L^2}^2 \leq - \norm{w_i}_{L^2} \norm{\nabla w_i}_{L^2} \left(\frac{2D_i}{p_i C_\Nash} \rho_i - \sum_{j=1}^N \norm{\nabla K_{ij}}_{L^{q_{ij}}} M_j \rho_j^{d/q_{ij}} \right).
        \end{equation}
        As \(\norm{u_i}_{L^{p_i}}\) is non-increasing if and only if \(\rho_i\) is non-increasing, we can conclude that \(\rho_i\) is non-increasing when \(\rho_i \geq \Phi_i(\rho)\) holds, where \(\Phi_i(\rho)\) is defined in~\eqref{cond:rho_i}. If \(\rho\) satisfies \(\rho_i \geq \Phi_i(\rho)\) for some \(i \in \{1,\dots,N\}\), then apply~\eqref{ineq:msgb-diff-bound} to~\eqref{ineq:msgb-deriv-first} and deduce that 
        \begin{equation}\label{ineq:msgb-nash-again}
            \frac{1}{2(p_i-1)} \frac{\rmd}{\rmd t} \norm{w_i}_{L^2}^2 \leq - \frac{2 D_i}{p_i C_\Nash^2} \norm{w_i}_{L^{2}}^{2} \rho_i \left(\rho_i - \Phi_i(\rho) \right).
        \end{equation}
        Now divide both sides of~\eqref{ineq:msgb-nash-again} by \(\norm{w_i}_{L^2}^2\). From the identities \(f'/f = (\log f)'\) when \(f > 0\) is differentiable and \(\norm{w_i}_{L^2}^2 = M_i^{p_i} \rho_i^{d(p_i-1)}\), observe that
        \begin{equation}\label{eqn:log-deriv}
            \begin{aligned}
                \frac{1}{2(p_i-1)} \frac{1}{\norm{w_i}_{L^2}^2} \frac{\rmd}{\rmd t} \norm{w_i}_{L^2}^2 
                = \frac{1}{2(p_i-1)} \frac{\rmd}{\rmd t} \big(\log \big(\rho_i^{d(p_i-1)}\big) + \log \left(M_i^{p_i}\right) \big) 
                = \frac{d}{2} \rho_i^{-1} \frac{\rmd}{\rmd t} \rho_i.
            \end{aligned}
        \end{equation}
        By combining~\eqref{ineq:msgb-nash-again} with~\eqref{eqn:log-deriv}, we can conclude~\eqref{ineq:rho_deriv} holds for each \(i \in \{1,\dots,N\}\) such that \(\rho_i \geq \Phi_i(\rho)\).
    \end{proof}  

    This result can be understood to be that \(\rho_i \geq \Phi_i(\rho)\) is a sufficient condition for the diffusive term to dominate the aggregative terms for species \(i\). This idea can be illustrated by the following trichotomy for the single species version of~\eqref{eqn:agg_diff} in the following sense: depending on the exponent \(q \in (1,\infty)\) in the regularity condition \(\nabla K \in L^q(\R^d)\) there exists either a uniform-in-time bound of the \(L^p\) norm of the solution, a smallness condition on the \(L^p\) norm, or a small mass condition.
    
    \begin{corollary}\label{cor:sing-species-asym}
        \label{cor:mild-strong}
        Suppose \(u(t)\) satisfies the conditions in Proposition~\ref{prop:global-bound} for \(N = 1\). Denote \(\lambda := d/q\), \(M := \norm{u_0}_{L^1} > 0\) and, for \(\lambda \neq 1\), 
        \begin{equation}\label{defn:rho-star}
            \rho^* := \left(p C_\Nash M \norm{\nabla K}_{L^q}/2D \right)^{1/(1-\lambda)}.
        \end{equation}
        Then the following statements hold:
        \begin{enumerate}
            \item \(\lambda < 1\) implies that \(\limsup_{t \to \infty}(\norm{u(t)}_{L^p}/M)^{p'/d} \leq \rho^*\).
            \item \(\lambda > 1\), \(d \geq 2\) and \((\norm{u(t_0)}_{L^p}/M)^{p'/d} < \rho^*\) for some \(t_0 \geq 0\) implies that \(\lim_{t \to \infty} \norm{u(t)}_{L^p} = 0\).
            \item \(\lambda = 1\), \(d \geq 2\) and 
            \begin{equation}\label{ineq:1S-mass-bd}
                M < \frac{2D}{p C_\Nash \norm{\nabla K}_{L^d}},
            \end{equation}
            implies \(\lim_{t \to \infty} \norm{u(t)}_{L^p} = 0\).
        \end{enumerate}
    \end{corollary}

    \begin{proof}
        Denote \(\rho(t) = (\norm{u}_{L^p}/M)^{p'/d}\) and \(a := p C_\Nash M \norm{\nabla K}_{L^q} / 2 D\) and consider the operator \(\Phi(\rho) := a \rho^\lambda\). For \(\lambda \neq 1\) observe that \(\rho^*\) as given by~\eqref{defn:rho-star} is the positive fixed point of \(\Phi\). 
        
        For \(\lambda < 1\), begin by noting that \(\rho(t) \geq \Phi(\rho(t))\) if and only if \(\rho(t) \geq \rho^*\). As \(\Phi\) satisfies~\eqref{cond:rho_i} for all \(t \geq 0\) such that \(\rho(t) \geq \rho^*\), it follows via Proposition~\ref{prop:global-bound} that
        \begin{equation}\label{ineq:rho-deriv-lam-big}
            \frac{\rmd}{\rmd t}\rho(t) \leq -C \rho(t)^2 (\rho(t) - a \rho(t)^\lambda) \leq - C (\rho^*)^2 (\rho(t) - a \rho(t)^\lambda),
        \end{equation}
        where \(C = 4D/dpC_\Nash^2\). Therefore, for any \(\rho(t_0) \geq \rho^*\), it follows from~\eqref{ineq:rho-deriv-lam-big} that either \(\rho(T) = \rho^*\) for some finite time \(T \geq t_0\), or \(\rho(t) \to \rho^*\) as \(t \to \infty\). Additionally, since \(\rho(t)\) must be non-increasing when \(\rho(t) = \rho^*\), it can be deduced that \(\rho(t_0) \leq \rho^*\) implies that \(\rho(t) \leq \rho^*\) for all \(t \geq t_0\). Combining the scenarios for \(\rho(t_0) \geq \rho^*\) and \(\rho(t_0) \leq \rho^*\) allows us to conclude \(\limsup_{t \to \infty}\rho(t) \leq \rho^*\).
        
        For the case \(\lambda > 1\) it instead follows that \(\rho(t) > \Phi(\rho(t))\) exactly when \(\rho(t) < \rho^*\). Again, from~\eqref{ineq:rho_deriv} it follows that
        \begin{equation}
            \frac{\rmd}{\rmd t}\rho(t) \leq -C \rho(t)^2 (\rho(t) - a \rho(t)^\lambda),
        \end{equation}
        from which it can be concluded that \(\rho(t)\) vanishes as \(t \to \infty\), equivalently \(\lim_{t \to \infty} \norm{u(t)}_{L^p} = 0\), so long as \(\rho(t_0) < \rho^*\) for some \(t_0 \geq 0\).

        Finally for \(\lambda = 1\), notice that \(a < 1\) implies \(\rho(t) > \Phi(\rho(t))\) for any value of \(\rho(t) > 0\). Hence~\eqref{ineq:rho_deriv} yields
        \begin{equation}
            \frac{\rmd}{\rmd t}\rho(t) \leq -C (1 - a) \rho(t)^3, 
        \end{equation}
        implying that \(\lim_{t \to \infty} \rho(t) = 0\) when \(a < 1\), which follows when~\eqref{ineq:1S-mass-bd} holds.
    \end{proof}

    

    While establishing the generalisation of Corollary~\ref{cor:sing-species-asym} for an \(N\) species system is beyond the scope of this work, a uniform \(L^P\) bound of the solution can be shown, given a small data condition. In particular, the result holds for systems that do not satisfy the regularity condition in Theorem~\ref{thm:global-1}. 
	
	Consider \(\Phi\) from~\eqref{cond:rho_i}, which can be written component-wise as
    \begin{equation}\label{defn:phi}
        \Phi_i(\rho) := \sum_{j=1}^N a_{ij} \rho_j^{\lambda_{ij}},
    \end{equation}
    where, for each choice of \(i,\, j \in \{1,\dots,N\}\), \(a_{ij}\) and \(\lambda_{ij}\) are given as
    \begin{equation}\label{defn:lam-a}
        a_{ij} := \frac{C_\Nash p_i \norm{\nabla K_{ij}}_{L^{q_{ij}}} M_j}{2 D_i}, \quad \lambda_{ij} := \frac{d}{q_{ij}}.
    \end{equation}
    Then, supposing that \(\Phi\) has a positive fixed point \(\rho^* \in (0,\infty)^N\), it follows from Proposition~\ref{prop:global-bound} that \(\rho^*\) determines an invariant set of \(\rho(t)\).

    \begin{proposition}\label{prop:invar-rect}
        Suppose \(u(t)\) satisfies the conditions in Proposition~\ref{prop:global-bound} and consider \(\rho(t)\) given by~\eqref{defn:rho} and \(\Phi\) given in~\eqref{defn:phi}--\eqref{defn:lam-a}. If \(\rho^* \in (0,\infty)^N\) is a positive fixed point of \(\Phi\), then the rectangle corresponding to \(\rho^*\), defined as
        \begin{equation}\label{eqn:rect*}
            R^* := \{\rho \in \R^N: \rho_i \in [0,\rho_i^*],\ 1 \leq i \leq N\}
        \end{equation}
        is an \emph{invariant set} of the mapping \(t \mapsto \rho(t)\). That is, if \(\rho(t_0) \in R^*\) for some \(t_0 \geq 0\), then \(\rho(t) \in R^*\) for all \(t \geq t_0\).
    \end{proposition}

    \begin{proof}
        First notice that the boundary of \(R^*\) can be written as the union of its faces,
        \begin{equation}
            \begin{aligned}
                \partial R^* = \bigcup_{i=1}^N &\left(\partial R^0_i \cup \partial R^*_i\right),\\
                \partial R^0_i := \{\rho \in R^* : \rho_i = 0\},&\quad  \partial R^*_i := \{\rho \in R^* : \rho_i = \rho_i^*\}.
            \end{aligned}
        \end{equation}
        By definition, \(\rho_i(t) \geq 0\) for every \(i \in \{1,\dots,N\}\), hence \(\rho(t)\) cannot pass through \(\partial R^0_i\) for any \(i\). To show that \(\rho(t)\) cannot pass through \(\partial R^*_i\) from within \(R^*\), it is sufficient to demonstrate that \(\rho_i(t) \in \partial R^*_i\) implies that \(\frac{\rmd}{\rmd t} \rho_i(t) \leq 0\). As \(a_{ij},\, \lambda_{ij} \geq 0\) for each \(i,\,j\), it follows from the observation \(x \mapsto x^\lambda\) is an increasing function for any \(\lambda > 0\) that \(\rho_i \leq \bar{\rho}_i\) for all \(i\) implies that \(\Phi_i(\rho) \leq \Phi_i(\bar{\rho})\) for all \(i\). In particular, for any \(\rho \in \partial R^*_i\), it follows that \(\rho_j \leq \rho^*_j\) for every \(j\) and so \(\Phi_i(\rho) \leq \Phi_i(\rho^*) = \rho_i\). Therefore, by the dynamic constraint~\eqref{ineq:rho_deriv}, we can conclude that \(\frac{\rmd}{\rmd t}\rho_i(t) \leq 0\) for all \(\rho(t) \in \partial R^*_i\), as required.
    \end{proof}

	In terms of solutions of~\eqref{eqn:agg_diff}, Proposition~\ref{prop:invar-rect} can be re-stated as the following result.
    
    \begin{corollary}\label{cor:small-lp}
        Suppose \(u(t)\) satisfies the conditions in Proposition~\ref{prop:global-bound} and consider \(\rho(t)\) given in~\eqref{defn:rho} and \(\Phi\) given in~\eqref{defn:phi}--\eqref{defn:lam-a}. If \(\Phi\) has a positive fixed point \(\rho^* \in (0,\infty)^N\), then \(u(t)\) can be extended to a global in time solution so long as \(u_0\) satisfies \(\norm{u_{0i}}_{L^{p_i}} \leq \norm{u_{0i}}_{L^1} (\rho_i^*)^{d/p_i'}\) for every \(i \in \{1,\dots,N\}\). In particular, for each \(i\), such a solution satisfies the bound \(\norm{u_{i}(t)}_{L^{p_i}} \leq \norm{u_{0i}}_{L^1} (\rho_i^*)^{d/p_i'}\) for all \(t \geq 0\).
    \end{corollary}

    With Corollary~\ref{cor:small-lp} in mind, it remains to establish a class of operators of the form in~\eqref{defn:phi} that produce a positive fixed point of \(\Phi\). To achieve this, we utilise the Krasnolsel'skii cone compression--expansion fixed point theorem for the Banach space \(\R^N\), which follows from~\cite[Theorem 2.3.4]{guo1988nonlinear}.

    \begin{lemma}\label{lem:cone-exp-comp}
        Consider the Banach space \((\R^N, \norm{\cdot})\), where \(\norm{\cdot}\) denotes an arbitrary norm over \(\R^N\). Let \(\Omega_1,\, \Omega_2 \subset \R^N\) be bounded open sets such that \(0 \in \Omega_1\), \(\overline{\Omega}_1 \subset \Omega_2\) and let \(C \subset \R^N\) be a cone, i.e.\ \(x \in C\) implies that \(\alpha x \in C\) for any \(\alpha \geq 0\), and \(x,\, -x \in C\) implies that \(x = 0\). If a continuous operator \(\Phi: C \cap (\overline{\Omega}_2 \setminus \Omega_1) \to C\) satisfies either of the following conditions:
        \begin{align}
            \left. \begin{aligned}
                \norm{\Phi(x)} \leq \norm{x}, \quad x \in C \cap \partial \Omega_1 \\
                \norm{\Phi(x)} \geq \norm{x}, \quad x \in C \cap \partial \Omega_2 
            \end{aligned} \right\},  \label{hyp:cone-1}
            \\[8pt] 
            \left. \begin{aligned}
                \norm{\Phi(x)} \geq \norm{x}, \quad x \in C \cap \partial \Omega_1 \\
                \norm{\Phi(x)} \leq \norm{x}, \quad x \in C \cap \partial \Omega_2
            \end{aligned} \right\}, \label{hyp:cone-2}
        \end{align}
        then \(\Phi\) has at least one fixed point in \(C \cap (\overline{\Omega}_2 \setminus \Omega_1)\).
    \end{lemma}

    \begin{proposition}\label{prop:phi-fixed}
        Suppose \(A = (a_{ij})\) and \(\Lambda = (\lambda_{ij})\) are non-negative \(N \times N\) matrices and define \(\Phi : [0,\infty)^N \to [0,\infty)^N\) as in~\eqref{defn:phi}. If, for any \(j \in \{1,\dots,N\}\), there is some \(i \in \{1,\dots,N\}\) such that \(a_{ij} > 0\) and for each \(i,\, j\) satisfying \(a_{ij} > 0\) it follows that \(\lambda_{ij} > 1\), then \(\Phi\) has a non-zero fixed point \(x^* \in [0,\infty)^N \setminus \{0\}\). \\
        Further suppose that \(A\) is \emph{irreducible}, i.e.\ for every pair \((i,\,j)\) there exists some path, \(i = i_0 \rightarrow \cdots \rightarrow i_l = j\), such that \(a_{i_{k} i_{k-1}} > 0\) for each \(k \in \{1,\dots,l\}\). Then, any non-zero fixed point of \(\Phi\) is in \((0,\infty)^N\).
    \end{proposition}

	\begin{remark}
		For \(\Phi\) corresponding to~\eqref{eqn:agg_diff}, it follows from~\eqref{defn:lam-a} that \(a_{ij} > 0\) if and only if \(\norm{\nabla K_{ij}}_{L^{q_{ij}}} > 0\). Consequently, the condition that \(A\) is irreducible is equivalent to the condition that the perception graph of~\eqref{eqn:agg_diff} is strongly connected.
	\end{remark}

    \begin{proof}
        Consider the max norm \(\norm{x}_\infty := \max_{1\leq i \leq N} \{\abs{x_i}\}\) and the cone \([0,\infty)^N\). In order to satisfy the assumptions in Lemma~\ref{lem:cone-exp-comp}, we verify that \(\Phi\) satisfies~\eqref{hyp:cone-1} with respect to the bounded open sets \(\Omega_1 = B(r_1)\), \(\Omega_2 = B(r_2)\) for some \(0 < r_1 < r_2 < \infty\) where, for any \(r > 0\), \(B(r) := \{x \in \R^N : \norm{x}_\infty < r\}\).

        To find some \(r_1\) satisfying \(\norm{\Phi(x)} \leq \norm{x}_\infty\) for every \(x \in [0,\infty)^N \cap \partial B(r_1)\) recall that \(\sum_{j} x_j/N \leq \norm{x}_\infty\). Thus, if \(a_{ij} x_j^{\lambda_{ij}} \leq x_j/N\) for every choice of \(i,\,j \in \{1,\dots,N\}\), then 
        \begin{equation}\label{ineq:r_1-bound-i}
            \Phi_i(x) = \sum_{j=1}^N a_{ij} x_j^{\lambda_{ij}} \leq \frac{1}{N} \sum_{j=1}^N x_j \leq \norm{x}_\infty
        \end{equation}
        for all \(i\). As the condition on \(x_j,\, a_{ij}\) and \(\lambda_{ij}\) is equivalent to \(x_j \leq (N a_{ij})^{-1/(\lambda_{ij}-1)}\) when \(a_{ij} > 0\), and trivially holds when \(a_{ij} = 0\), it can be concluded that \(\norm{\Phi(x)} \leq \norm{x}_\infty\) for all \(x \in [0,\infty)^N \cap \partial B(r_1)\) when
        \begin{equation}\label{cond:cone-r_1}
            r_1 \leq \min \left\{(N a_{ij})^{-1/(\lambda_{ij}-1)} : 1 \leq i,\, j \leq N,\, a_{ij} > 0\right\}.
        \end{equation}
        To obtain \(r_2\) for which \(\norm{\Phi(x)}_\infty \geq \norm{x}_\infty\) for all \(x \in [0,\infty)^N \cap \partial B(r_2)\), first assume that, for some \(x \in [0,\infty)^N\), its maximum value is achieved by the \(j\)-th term, \(x_j\). For our given \(j\), recall from our assumption on \(A\) that there is at least one choice of \(i\) such that \(a_{ij} > 0\). Therefore, so long as \(x_j\) is large enough to satisfy \(x_j \geq a_{ij}^{-1/(\lambda_{ij}-1)}\) for some \(i\), it follows that
        \begin{equation}
            \norm{\Phi(x)}_\infty \geq \Phi_i(x) \geq a_{ij} x_j^{\lambda_{ij}} \geq x_j =\norm{x}_\infty.
        \end{equation}  
        After accounting for the arbitrary choice of \(j\) such that \(x_j = \norm{x}_\infty\), it can be deduced that \(\norm{\Phi(x)}_\infty \geq \norm{x}_\infty\) for any \(x \in [0,\infty)^N \cap \partial B(r_2)\) whenever \(r_2\) is large enough to satisfy
        \begin{equation}\label{cond:cone-r_2}
            r_2 \geq \max_{1 \leq j \leq N} \left\{\min \left\{a_{ij}^{-1/(\lambda_{ij}-1)} : 1 \leq i \leq N,\, a_{ij} > 0 \right\} \right\}.
        \end{equation}
        After choosing some \(0 < r_1 < r_2\) satisfying~\eqref{cond:cone-r_1} and~\eqref{cond:cone-r_2} respectively, we can apply Lemma~\ref{lem:cone-exp-comp} to conclude that there exists a non-zero fixed point \(x^*\) satisfying \(r_1 \leq \norm{x^*}_\infty \leq r_2\).

        Now further assume that the matrix \(A\) is irreducible. As \(x^* \in [0,\infty)^N \setminus \{0\}\), there is some \(j_0\) such that \(x_{j_0}^* > 0\). Now choose some \(j_1 \in \{1,\dots,N\} \setminus \{j_0\}\). By the irreducible assumption on \(A\) there exists some path \(j_0 = i_0 \rightarrow \cdots \rightarrow i_{l} = j_1\) such that \(a_{i_{k} i_{k-1}} > 0\) for each \(k \leq l\). Hence, given that \(x^*_{i_{k-1}} > 0\) for some \(k \in \{1,\dots,l\}\) it follows that
        \begin{equation}\label{ineq:x_i_k}
            x^*_{i_k} = \Phi_{i_k}(x^*) \geq a_{i_{k} i_{k-1}} (x^*_{i_{k-1}})^{\lambda_{i_{k} i_{k-1}}} > 0.
        \end{equation}
        Since \(x^*_{j_0} > 0\) it can then be deduced, by iterating~\eqref{ineq:x_i_k} for each \(i_k\) along the path between \(j_0\) and \(j_1\), that \(x^*_{j_1} > 0\). For each \(n\) up to \(N-1\) repeat the same argument for \(j_n \in \{1,\dots,N\} \setminus \{j_0,\dots,j_{n-1}\}\) to conclude that \(x^*_j > 0\) for all \(j \in \{1,\dots,N\}\).
    \end{proof}	

    \begin{remark}\label{rem:phi-fixed-2}
        In the case that \(\lambda_{ij} < 1\) rather than \(\lambda_{ij} > 1\) for each \(i,\, j\) such that \(a_{ij} > 0\), a similar argument to the proof of Proposition~\ref{prop:phi-fixed} can be used to find some \(0 < r_1 < r_2 < \infty\) such that \(\Phi\) satisfies~\eqref{hyp:cone-2}.
    \end{remark}

    Although Proposition~\ref{prop:phi-fixed} and Remark~\ref{rem:phi-fixed-2} can be applied for a large class of kernels, they do not provide the full picture for when the operator \(\Phi\) has a positive fixed point. In the case of a two-species system, we demonstrate the existence of positive fixed points for two examples that do not fall into the aforementioned class of kernels.

    \begin{example}[No self-perception]\label{ex:no-self-perp}
        Suppose \(\nabla K_{11}\) and \(\nabla K_{22}\) are trivial and that \(\nabla K_{12} \in L^{q_{12}}(\R^d)\) and \(\nabla K_{21} \in L^{q_{21}}(\R^d)\) are both non-trivial for some \(q_{12},\, q_{21} \in (1,\infty)\) satisfying \((q_{12} q_{21})^{1/2} \neq d\). Then, the corresponding operator \(\Phi\) will be of the form
        \begin{equation}
            \Phi(x) := (a_{12} x_2^{\lambda_{12}},\, a_{21} x_1^{\lambda_{21}}),
        \end{equation}
        where \(a_{12},\, a_{21} > 0\) and \(\lambda := \lambda_{12} \lambda_{21} \neq 1\). By then explicitly solving the simultaneous equation \(\Phi(x) = x\) it follows that
        \begin{equation}\label{eqn:2-cycle-fixed}
            x^* = ((a_{12}a_{21}^{\lambda_{12}})^{1/(1-\lambda)},\, (a_{12}^{\lambda_{21}}a_{21})^{1/(1-\lambda)})
        \end{equation}
        is a positive fixed point, as desired.
    \end{example}

    \begin{example}[Unilateral cross-species perception]
        Suppose that \(K_{21}\) is equivalently zero, \(K_{12},\, K_{22}\) are not constant and \(\nabla K_{ij} \in L^{q_{ij}}(\R^d)\) for each choice of \(i,\,j \in \{1,2\}\), where \(q_{11},\, q_{22} < d\) and \(q_{12} \in (1,\infty)\). Then, the corresponding \(\Phi\) will be of the form
        \begin{equation}\label{defn:phi-scene-1}
            \Phi(x) := (a_{11} x_1^{\lambda_{11}} + a_{12} x_2^{\lambda_{12}},\, a_{22} x_2^{\lambda_{22}})
        \end{equation}
        where \(a_{11} \geq 0,\, a_{12},\, a_{22} > 0\), \(\lambda_{11},\, \lambda_{22} > 1\) and \(\lambda_{12} > 0\). Note that the condition \(a_{21} = 0\) implies that \(A\) is not irreducible. It follows immediately from~\eqref{defn:phi-scene-1} that a positive fixed point of \(\Phi\) satisfies \(x^*_2 = a_{22}^{-1/(\lambda_{22}-1)}\), which is then substituted back into \(\Phi\). If \(a_{11} = 0\), then \(x_1^* = a_{12} a_{22}^{-\lambda_{12}/\lambda_{22}-1}\). To determine if \(x^*_1\) exists when \(a_{11} > 0\), consider the function
        \begin{equation}
            f(y) := a_{11} y^{\lambda_{11}} + a_{12} a_{22}^{-\lambda_{12}/(\lambda_{22}-1)} - y.
        \end{equation}
        By the intermediate value theorem, it follows that there exists some \(y^* > 0\) such that \(f(y^*) = 0\) when \(f(y_{\min}) \leq 0\), where \(y_{\min} = (\lambda_{11} a_{11})^{-1/(\lambda_{11}-1)}\) is the minimiser of \(f\). After some re-arranging, observe that \(f(y_{\min}) \leq 0\) holds if and only if
        \begin{equation}\label{ineq:uni-perp-small-a}
            \left(\lambda_{11} a_{11}\right)^{1/(\lambda_{11} - 1)}  \leq \left(\frac{\lambda_{11} - 1}{\lambda_{11}}\right) a_{21}^{-1} a_{22}^{\lambda_{12}/(\lambda_{22}-1)},
        \end{equation}
        in which case \(x^*_1 = y^*\) satisfies \(\Phi_1(x) = x_1\). By recalling from the definition of \(a_{ij}\) in~\eqref{defn:lam-a},~\eqref{ineq:uni-perp-small-a} is equivalent to the following condition on the masses \(M_1,\, M_2\)
        \begin{equation}\label{ineq:mass-relationship}
            M_1 \leq c M_2^{\alpha},\quad \alpha = \frac{\lambda_{11} - 1}{\lambda_{22} - 1} (\lambda_{21} + 1 -\lambda_{22})
        \end{equation} for some \(c > 0\) depending on \(d,\, D,\, K,\, P\) and \(Q\). That is, if \(M_1,\, M_2\) satisfy~\eqref{ineq:mass-relationship} then there exists smallness estimates on \(\norm{u_{01}}_{L^{p_1}}\) and \(\norm{u_{02}}_{L^{p_2}}\) for which~\eqref{eqn:agg_diff} has a global-in-time bound.
    \end{example} 
    
    
\section{An example family of kernels}\label{section:numerics}

	The focus of this section is to illustrate some relationships between the regularities of the perception kernel and the behaviour of their corresponding solutions of~\eqref{eqn:agg_diff}. To do so, we consider an example family of perception kernels of the form \(K = \gamma W_s\), where \(\gamma > 0\) is the attractive strength of the kernel and, for each \(s \in (-d, \infty)\), \(W_s\) is a probability density function given as
	\begin{equation}\label{defn:gen-hats}
		W_s(x) :=  
		\begin{dcases}
			{\frac{s + d}{\omega_d s} (1 - \abs{x}^{s}) \chi_{B_1}(x)}, & s \neq 0, \\
			{-\frac{d}{\omega_d} \log(\abs{x}) \chi_{B_1}(x)}, & s = 0,
		\end{dcases}
	\end{equation} 
	where \(\chi_{B_1}\) is the indicator function for the unit ball \(B_1 := \{x: |x| \leq 1\}\) and \(\omega_d\) is the volume of \(B_1\) in \(d\) dimensions. Since all kernels in this section are chosen to be attractive, solutions of~\eqref{eqn:agg_diff} are specifically solutions to aggregation--diffusion systems.
 
	This choice of family of kernels allows for a wide range of regularities at the origin and at the boundary \(|x| = 1\). For \(s \leq 0\), \(W_s\) blows up at the origin and, as \(s\) tends to \(-d\) from above, an increasing proportion of the mass is concentrated about the origin and so \(W_s\) tends to the Dirac delta distribution (i.e.\ a point mass) in the sense of distributions. At the other extreme, as \(s\) tends to infinity, \(W_s(x)\) tends to the uniform distribution over \(B_1\) in the sense of distributions.

	Since the analysis in this work assumes that all the kernels have a \(q\)-integrable weak derivative for some \(q \in (1,\infty)\), the results obtained regarding global existence and uniform bounds of the solutions as covered in Sections~\ref{section:global} and~\ref{section:dynamics} only apply to \(W_s\) when \(s \in (-d+1,\infty)\). This follows by observing that the pointwise derivative, given as
    \begin{equation}\label{defn:gen-hats-deriv}
        \nabla W_s(x) = - \frac{s+d}{\omega_d} \frac{x}{\abs{x}} \abs{x}^{s-1} \chi_{B_1}(x)
    \end{equation}
    for all \(x \in B_1 \setminus \{0\}\) is \(q\)-integrable for some \(q > 1\). In particular, the \(L^q\) norm of \(\nabla W_s\),
	\begin{equation}
		\norm{\nabla W_s}_{L^q}^q = \int_{B_1} \abs{\nabla W_s(x)}^{q} \rmd x = \frac{d(s+d)^q}{\omega_d^{q-1}} \int_0^1 r^{q(s-1) + d-1} \rmd r = \frac{d(s+d)^q}{\omega_d^{q-1}(q(s-1) + d)},
	\end{equation}
	is well-defined and finite for any choice of \(s \in (-d+1,\infty)\) and \(q \in (1,\infty)\) satisfying \(q(s-1) > - d\). Thus, we can conclude that \(\nabla W_s \in L^q(\R^d)\) for any choice of \(q \in (1,\infty)\) when \(s \in [1,\infty)\) and for all \(q \in (1, d/(1-s))\) when \(s \in (-d+1,1)\). With this range in mind, define the critical integrability of \(\nabla W_s\) to be
	\begin{equation}
		\bar{q} := \sup\{q \in [1,\infty) : \nabla W_s \in L^q(\R^d)\} = 
		\begin{dcases}
			\frac{d}{1-s}, \, &s \in (-d+1,1), \\
			\infty, \, &s \in [1,\infty).
		\end{dcases}
	\end{equation}
	Similarly, by recalling the regularity parameter \(\lambda := d/q\) as used in Section~\ref{section:dynamics}, define the critical regularity of \(W_s\) to be
	\begin{equation}\label{defn:crit-reg}
		\bar{\lambda} := \inf\{\lambda \in (0,d] : \nabla W_s \in L^{d/\lambda}(\R^d)\} = \max\{1-s,\, 0\}.
	\end{equation}
    Throughout this section, we study how numerical approximations of solutions of~\eqref{eqn:agg_diff} behave with respect to \(s\). By doing so, we can illustrate the main results presented in this work and consider their potential generalisations.

	\begin{figure}
		\centering
        \resizebox{0.45\linewidth}{!}{
		\begin{subfigure}{0.45\linewidth}
			\begin{tikzpicture}[baseline=(current axis.north)]
			\begin{axis}[
				xlabel={$x$},
				ylabel={$W_s(x)$},
                xlabel style={font=\small\sffamily, yshift=0em},
				ylabel style={font=\small\sffamily, yshift=-0.5em},
                tick label style={font=\footnotesize},
				xtick={-1, 0, 1},
				ytick={0, 0.5, 1.0, 1.5, 2.0},
				grid=major,
				grid style={gray!30},
				width= 7.5cm,
				height=5.5cm,
				xmin=-1,
				xmax=1,
				ymin=0,
				ymax=2,
				legend pos=north east,
				legend style={font=\scriptsize},
				domain=-1:1,
				samples=50,
				smooth,
			]

			\addplot[
				color=black,
				line width=1,
				dashed,
				domain=-1:-0.1,
				samples=50,
				forget plot
			] {(6/10) * (1 - (abs(x))^5))};
			\addplot[
				color=black,
				line width=1,
				dashed,
				domain=0.1:1,
				samples=50,
				forget plot
			] {(6/10) * (1 - (abs(x))^5))};

			\addplot[
				color=black,
				line width=1,
				dashed,
				domain=-0.1:0.1,
				samples=100
			] {(6/10) * (1 - (abs(x))^5))};
            
			\addplot[
				color=black,
				line width=1,
				domain=-1:-0.1,
				samples=50,
				forget plot
			] {(1+1)/2 * (1 - (abs(x))^(1/1))};

			\addplot[
				color=black,
				line width=1,
				domain=0.1:1,
				samples=50,
				forget plot
			] {(1+1)/2 * (1 - (abs(x))^(1/1))};

			\addplot[
				color=black,
				line width=1,
				domain=-0.1:0.1,
				samples=50
			] {(1+1)/2 * (1 - (abs(x))^(1/1))};







			\addplot[
				color=black,
				line width=1,
				densely dotted,
				domain=-1:-0.1,
				samples=50
			] { 1/2 * (abs(x)^-1/2 - 1)};

            
			\addplot[
				color=black,
				line width=1,
				densely dotted,
				domain=0.1:1,
				samples=50,
				forget plot
			] { 1/2 * (abs(x)^-1/2 - 1)};

			\addplot[
				color=black,
				line width=1,
				densely dotted,
				domain=-0.1:0.1,
				samples=100,
				forget plot
			] { 1/2 * (abs(x)^-1/2 - 1)};

			\legend{$s = 5$, $s = 1$, $s = -0.5$}
			\end{axis}
		\end{tikzpicture}
		\end{subfigure}%
        }%
		\hspace{0.05\linewidth}%
        \resizebox{0.45\linewidth}{!}{
		\begin{subfigure}{0.45\linewidth}
			\begin{tikzpicture}[baseline=(current axis.north)]
			\begin{axis}[
				xlabel={$x$},
				ylabel={$W'_s(x)$},
                xlabel style={font=\small\sffamily, yshift=0em},
				ylabel style={font=\small\sffamily, yshift=-0.5em},
                tick label style={font=\footnotesize},
				xtick={-1, 0, 1},
				ytick={-4.0, -2.0, 0, 2.0, 4.0},
				grid=major,
				grid style={gray!30},
				width=7.5cm,
				height=5.5cm,
				xmin=-1,
				xmax=1,
				ymin=-4,
				ymax=4,
				legend pos=north east,
				legend style={font=\scriptsize},
				domain=-1:1,
				samples=50,
				smooth,
			]

			\addplot[
				color=black,
				line width=1,
				dashed,
				domain=-1:0,
				samples=100,
				forget plot
			] {3 * abs(x)^4};

			\addplot[
				color=black,
				line width=1,
				dashed,
				domain=0:1,
				samples=100,
				forget plot
			] {- 3 * abs(x)^4};

			\addplot[
				color=black,
				line width=1,
				dashed,
				domain=0.00001:0.1,
				samples=100
			] {- 3 * abs(x)^4};

			\addlegendentry{$s=5$}
            
			\addplot[
				color=black,
				line width=1,
				domain=-1:0
			] {1};

			\addplot[
				color=black,
				line width=1,
				domain=0:1,
				forget plot
			] {-1};

			\addplot[
				color=black,
				line width=1,
				domain=-1:1,
				forget plot
			] coordinates {(0,1) (0,-1)};





			\addlegendentry{$s=1$}






			\addplot[
				color=black,
				line width=1,
				densely dotted,
				domain=-1:-0.1,
				samples=50,
				forget plot
			] {1/4 * (abs(x))^(-3/2)};

			\addplot[
				color=black,
				line width=1,
				densely dotted,
				domain=0.1:1,
				samples=50,
				forget plot
			] {- 1/4 * (abs(x))^(-3/2)};

			\addplot[
				color=black,
				line width=1,
				densely dotted,
				domain=-0.1:-0.01,
				samples=50,
				forget plot
			] {1/4 * (abs(x))^(-3/2)};

			\addplot[
				color=black,
				line width=1,
				densely dotted,
				domain=0.01:0.1,
				samples=50
			] {- 1/4 * (abs(x))^(-3/2)};

			\addlegendentry{$s= -0.5$}
			\end{axis}
		\end{tikzpicture}
		\end{subfigure}
        }
		\caption{Graphs of the functions $W_s(x)$ and $W'_s(x)$ over the domain \(x \in [-1,1]\), as defined in~\eqref{defn:gen-hats} and~\eqref{defn:gen-hats-deriv} respectively, for $s \in \{-0.5,\, 1,\, 5\}$.} \label{plot:hats}
	\end{figure}

    In the following examples, all numerical approximations of solutions of~\eqref{eqn:agg_diff} are computed using the finite element method which is implemented via the FEniCSx package~\cite{BarattaEtal2023, ScroggsEtal2022, BasixJoss, AlnaesEtal2014}. 
    
    \begin{example}[Dependence of peak concentration on $s$, $\gamma$ and $D$]\label{eg:num-steady}

        In this example, we explore the relationship between \(s\), \(\gamma\), \(D\) and the evolution of a family of concentration functionals, defined as 
        \begin{equation}\label{def:rho_p}
            \rho_p(t) := \begin{dcases}
                \left(\int_{\R} \abs{u(t)}^p \rmd x\right)^{1/(p-1)}, \quad &p \in (1,\infty), \\
                \sup_{x \in \R} \abs{u(t)}, \quad &p = \infty,
            \end{dcases}
        \end{equation}
        where \(u\) is the solution of
        \begin{equation}\label{eqn:1d-agg-diff} \tag{Agg-Diff}
            \begin{dcases}
                \partial_t u = D\, \partial_{xx} u - \gamma\, \partial_x(u\, \partial_x(W_s \ast u)), &\quad (t, x) \in  (0,\infty) \times \R\\
                u(0,x) = \max\{1 - \abs{x},\, 0\}, & \quad x \in \R,
            \end{dcases}
        \end{equation}
        with \(D,\, \gamma,\, s > 0\). Notice that, for \(p \in (1,\infty)\),~\eqref{def:rho_p} coincides with the concentration defined in~\eqref{defn:rho_(intro)} for the special case \(M = d = 1\) and,  because \(\lim_{p \to \infty} p' = 1\), \(\rho_\infty\) is the limit of \(\rho_p\) as \(p \to \infty\). Since the initial data has the regularity \(L^1 \cap L^\infty(\R)\), it follows from Theorem~\ref{Thm:local} and Theorem~\ref{thm:global-1} that~\eqref{eqn:1d-agg-diff} is globally well-posed.

        To understand how \(\rho_p(t)\) evolves over time, we approximate \(\rho_p(t)\) for varying choices of \(L > 1\) by
        \begin{equation}\label{defn:approx-rho}
            \rho^L_p(t) := \begin{dcases}
                \left(\int_{-L}^L \abs{u^L(t)}^p \rmd x\right)^{1/(p-1)}, \quad &p \in (1,\infty), \\
                \sup_{x \in [-L,L]} \abs{u^L(t)}, \quad &p = \infty,
            \end{dcases}
        \end{equation}
        where \(u^L\) is the computed solution of the spatially truncated initial-boundary value problem
        \begin{equation}\label{eqn:1d-cauchy}
            \begin{dcases}
                \partial_t u = D\, \partial_{xx} u - \gamma\, \partial_x(u\, \partial_x(W_s \ast u)), &\quad (t, x) \in  (0,\infty) \times [-L, L]\\
                u(0,x) = \max\{1 - \abs{x},\, 0\}, & \quad x \in [-L,L],
            \end{dcases}
        \end{equation}
        with zero flux boundary conditions and spatial convolution defined by extending the solution by zero outside the domain. 

        Figure~\ref{fig:rho-evol} illustrates how the approximations of \(\rho_2\) evolve over time, depending on the choice of \(\gamma\) and \(L \in \{5,10,100\}\). In the short term the evolution of \(\rho^L_2(t)\) seems to be insensitive to the choice of \(L\),  however, in the long term the approximations do seem to depend on the choice of \(L\). Although it is beyond the scope of this work, the authors believe that there exists some critical domain length \(L^*(\gamma/D, s) \in [1,\infty)\) such that, for all \(L > L^*\), \(\rho^L_2(t)\) eventually decays towards \((2L)^{-1}\), the minimum possible value of \(\rho^L_2\) for a function of unit mass over the interval \([-L,L]\).

        \begin{figure}[!ht]
    	\centering
    	\begin{subfigure}[t]{0.32\linewidth}
            \centering
                \begin{tikzpicture}
                    \begin{axis}[
                        xlabel={$t$},
                        ylabel={$\rho^L_2(t)$},
                        xlabel style={font=\small\sffamily},
                        ylabel style={font=\small\sffamily},
                        tick label style={font=\footnotesize},
                        width=5.5cm, height=5cm,
                        xmin=0, xmax=3.00043,
                        xtick={0, 0.30103, 1.04139, 2.00432, 3.00043},
                        xticklabels={0,1,10,100,1000},
                        minor xtick={0.47712, 0.60206, 0.69897, 0.77815, 0.84510, 0.90309, 0.95424, 1, 1.32222, 1.49136, 1.61278, 1.70757, 1.78533, 1.85126, 1.90849, 1.95904, 2.30320, 2.47857, 2.60314, 2.69984, 2.77887, 2.84572, 2.90363, 2.95472},
                        ymin=0, ymax=0.75,
                        grid=both,
                        major grid style={solid, gray!60},
                        minor grid style={solid, gray!15},
                        minor tick length=0pt,
                        every axis plot/.append style={x filter/.expression={log10(x+1)}},
                        ]
    
                        \addplot[
                        color=black,
                        dotted,
                        line width = 1pt,
                        no marks
                        ] 
                        coordinates {
                        (0.0, 0.66667)
                        (0.1, 0.64041)
                        (0.2, 0.61771)
                        (0.3, 0.59719)
                        (0.4, 0.57842)
                        (0.5, 0.56118)
                        (0.6, 0.54527)
                        (0.7, 0.53053)
                        (0.8, 0.51685)
                        (0.9, 0.50409)
                        (1.0, 0.49218)
                        (2.0, 0.40599)
                        (3.0, 0.35184)
                        (4.0, 0.31404)
                        (5.0, 0.28584)
                        (6.0, 0.2638)
                        (7.0, 0.24599)
                        (8.0, 0.23124)
                        (9.0, 0.21876)
                        (10.0, 0.20803)
                        (20.0, 0.14918)
                        (30.0, 0.12409)
                        (40.0, 0.11197)
                        (50.0, 0.10598)
                        (60.0, 0.103)
                        (70.0, 0.10152)
                        (80.0, 0.10078)
                        (90.0, 0.1004)
                        (100.0, 0.10022)
                        (200.0, 0.10001)
                        (300.0, 0.10001)
                        (400.0, 0.10001)
                        (500.0, 0.10001)
                        (600.0, 0.10001)
                        (700.0, 0.10001)
                        (800.0, 0.10001)
                        (900.0, 0.10001)
                        (1000.0, 0.10001)
                        };

                        \addplot[
                        color=black,
                        line width = 1pt,
                        dashed,
                        no marks
                        ] 
                        coordinates {
                        (0.0, 0.66667)
                        (0.1, 0.64041)
                        (0.2, 0.61771)
                        (0.3, 0.59719)
                        (0.4, 0.57842)
                        (0.5, 0.56118)
                        (0.6, 0.54527)
                        (0.7, 0.53053)
                        (0.8, 0.51685)
                        (0.9, 0.50409)
                        (1.0, 0.49218)
                        (2.0, 0.40599)
                        (3.0, 0.35184)
                        (4.0, 0.31404)
                        (5.0, 0.28584)
                        (6.0, 0.2638)
                        (7.0, 0.24599)
                        (8.0, 0.23124)
                        (9.0, 0.21876)
                        (10.0, 0.20803)
                        (20.0, 0.14866)
                        (30.0, 0.12102)
                        (40.0, 0.10443)
                        (50.0, 0.09311)
                        (60.0, 0.0848)
                        (70.0, 0.0784)
                        (80.0, 0.07334)
                        (90.0, 0.06925)
                        (100.0, 0.06592)
                        (200.0, 0.05266)
                        (300.0, 0.05045)
                        (400.0, 0.05008)
                        (500.0, 0.05001)
                        (600.0, 0.05)
                        (700.0, 0.05)
                        (800.0, 0.05)
                        (900.0, 0.05)
                        (1000.0, 0.05)
                        };

                        \addplot[
                        color=black,
                        line width = 1pt,
                        no marks
                        ] 
                        coordinates {
                        (0.0, 0.66667)
                        (0.1, 0.64041)
                        (0.2, 0.61771)
                        (0.3, 0.59719)
                        (0.4, 0.57842)
                        (0.5, 0.56118)
                        (0.6, 0.54527)
                        (0.7, 0.53053)
                        (0.8, 0.51685)
                        (0.9, 0.50409)
                        (1.0, 0.49218)
                        (2.0, 0.40599)
                        (3.0, 0.35184)
                        (4.0, 0.31404)
                        (5.0, 0.28584)
                        (6.0, 0.2638)
                        (7.0, 0.24599)
                        (8.0, 0.23124)
                        (9.0, 0.21876)
                        (10.0, 0.20803)
                        (20.0, 0.14866)
                        (30.0, 0.12102)
                        (40.0, 0.10442)
                        (50.0, 0.0931)
                        (60.0, 0.08476)
                        (70.0, 0.0783)
                        (80.0, 0.0731)
                        (90.0, 0.0688)
                        (100.0, 0.06518)
                        (200.0, 0.04592)
                        (300.0, 0.03732)
                        (400.0, 0.03221)
                        (500.0, 0.02875)
                        (600.0, 0.0262)
                        (700.0, 0.02422)
                        (800.0, 0.02263)
                        (900.0, 0.02132)
                        (1000.0, 0.02021)
                        };

                    \end{axis}
                \end{tikzpicture}
    		\caption{\(\gamma = 0.1\)}
    	\end{subfigure}
    	\hfill
        \begin{subfigure}[t]{0.32\linewidth}
            \centering
    			\begin{tikzpicture}
    				\begin{axis}[
                        xlabel={$t$},
                        ylabel={$\rho^L_2(t)$},
                        xlabel style={font=\small\sffamily},
                        ylabel style={font=\small\sffamily},
                        tick label style={font=\footnotesize},
                        width=5.5cm, height=5cm,
                        xmin=0, xmax=3.00043,
                        xtick={0, 0.30103, 1.04139, 2.00432, 3.00043},
                        xticklabels={0,1,10,100,1000},
                        minor xtick={0.47712, 0.60206, 0.69897, 0.77815, 0.84510, 0.90309, 0.95424, 1,
                                    1.32222, 1.49136, 1.61278, 1.70757, 1.78533, 1.85126, 1.90849, 1.95904,
                                    2.30320, 2.47857, 2.60314, 2.69984, 2.77887, 2.84572, 2.90363, 2.95472},
                        ymin=0, ymax=1.2,
                        grid=both,
                        major grid style={solid, gray!60},
                        minor grid style={solid, gray!15},
                        minor tick length=0pt,
                        legend columns=3,
                        legend style={at={(0.5,1.03)}, anchor=south, draw=none,
                        font=\scriptsize\sffamily,
                        /tikz/every even column/.append style={column sep=1.2em}},
                        every axis plot/.append style={x filter/.expression={log10(x+1)}},
                        ]
    
    					\addplot[
    					color=black,
    					line width = 1pt,
                        dotted,
    					no marks
    					] coordinates {
                        (0.0, 0.66667)
                        (0.04, 0.67517)
                        (0.08, 0.68449)
                        (0.12, 0.69424)
                        (0.16, 0.70428)
                        (0.2, 0.71452)
                        (0.24, 0.72486)
                        (0.28, 0.73526)
                        (0.32, 0.74565)
                        (0.36, 0.756)
                        (0.4, 0.76628)
                        (0.44, 0.77643)
                        (0.48, 0.78645)
                        (0.52, 0.79631)
                        (0.56, 0.80598)
                        (0.6, 0.81546)
                        (0.64, 0.82472)
                        (0.68, 0.83376)
                        (0.72, 0.84257)
                        (0.76, 0.85114)
                        (0.8, 0.85948)
                        (0.84, 0.86758)
                        (0.88, 0.87543)
                        (0.92, 0.88304)
                        (0.96, 0.89042)
                        (1.0, 0.89755)
                        (1.4, 0.9505)
                        (1.8, 0.98918)
                        (2.2, 1.01663)
                        (2.6, 1.03581)
                        (3.0, 1.04907)
                        (3.4, 1.05817)
                        (3.8, 1.06435)
                        (4.2, 1.0685)
                        (4.6, 1.07122)
                        (5.0, 1.07293)
                        (5.4, 1.07393)
                        (5.8, 1.07441)
                        (6.2, 1.07452)
                        (6.6, 1.07435)
                        (7.0, 1.074)
                        (7.4, 1.07349)
                        (7.8, 1.07287)
                        (8.2, 1.07218)
                        (8.6, 1.07142)
                        (9.0, 1.07062)
                        (9.4, 1.06979)
                        (9.8, 1.06893)
                        (12.0, 1.06483)
                        (16.0, 1.05638)
                        (20.0, 1.04797)
                        (24.0, 1.03995)
                        (28.0, 1.03234)
                        (32.0, 1.02511)
                        (36.0, 1.01823)
                        (40.0, 1.01167)
                        (44.0, 1.0054)
                        (48.0, 0.99941)
                        (52.0, 0.99368)
                        (56.0, 0.9882)
                        (60.0, 0.98295)
                        (64.0, 0.97792)
                        (68.0, 0.9731)
                        (72.0, 0.96847)
                        (76.0, 0.96404)
                        (80.0, 0.95978)
                        (84.0, 0.9557)
                        (88.0, 0.95178)
                        (92.0, 0.94801)
                        (96.0, 0.94439)
                        (100.0, 0.94092)
                        (140.0, 0.91916)
                        (180.0, 0.9003)
                        (220.0, 0.88645)
                        (260.0, 0.87631)
                        (300.0, 0.86884)
                        (340.0, 0.8633)
                        (380.0, 0.85919)
                        (420.0, 0.85611)
                        (460.0, 0.85382)
                        (500.0, 0.85209)
                        (540.0, 0.8508)
                        (580.0, 0.84983)
                        (620.0, 0.8491)
                        (660.0, 0.84855)
                        (700.0, 0.84814)
                        (740.0, 0.84783)
                        (780.0, 0.8476)
                        (820.0, 0.84742)
                        (860.0, 0.84729)
                        (900.0, 0.84719)
                        (940.0, 0.84711)
                        (980.0, 0.84705)
                        (1000.0, 0.84703)
    					};

                        \addplot[
    					color=black,
    					line width = 1pt,
                        dashed,
    					no marks
    					] coordinates {
                        (0.0, 0.66667)
                        (0.04, 0.67517)
                        (0.08, 0.68449)
                        (0.12, 0.69424)
                        (0.16, 0.70428)
                        (0.2, 0.71452)
                        (0.24, 0.72486)
                        (0.28, 0.73526)
                        (0.32, 0.74565)
                        (0.36, 0.756)
                        (0.4, 0.76628)
                        (0.44, 0.77643)
                        (0.48, 0.78645)
                        (0.52, 0.79631)
                        (0.56, 0.80598)
                        (0.6, 0.81546)
                        (0.64, 0.82472)
                        (0.68, 0.83376)
                        (0.72, 0.84257)
                        (0.76, 0.85114)
                        (0.8, 0.85948)
                        (0.84, 0.86758)
                        (0.88, 0.87543)
                        (0.92, 0.88304)
                        (0.96, 0.89042)
                        (1.0, 0.89755)
                        (1.4, 0.9505)
                        (1.8, 0.98918)
                        (2.2, 1.01663)
                        (2.6, 1.03581)
                        (3.0, 1.04907)
                        (3.4, 1.05817)
                        (3.8, 1.06435)
                        (4.2, 1.0685)
                        (4.6, 1.07122)
                        (5.0, 1.07293)
                        (5.4, 1.07393)
                        (5.8, 1.07441)
                        (6.2, 1.07452)
                        (6.6, 1.07435)
                        (7.0, 1.074)
                        (7.4, 1.07349)
                        (7.8, 1.07287)
                        (8.2, 1.07218)
                        (8.6, 1.07142)
                        (9.0, 1.07062)
                        (9.4, 1.06979)
                        (9.8, 1.06893)
                        (12.0, 1.06483)
                        (16.0, 1.05638)
                        (20.0, 1.04797)
                        (24.0, 1.03993)
                        (28.0, 1.03228)
                        (32.0, 1.02497)
                        (36.0, 1.01795)
                        (40.0, 1.0112)
                        (44.0, 1.00467)
                        (48.0, 0.99833)
                        (52.0, 0.99216)
                        (56.0, 0.98615)
                        (60.0, 0.98027)
                        (64.0, 0.97451)
                        (68.0, 0.96886)
                        (72.0, 0.96331)
                        (76.0, 0.95784)
                        (80.0, 0.95246)
                        (84.0, 0.94716)
                        (88.0, 0.94192)
                        (92.0, 0.93674)
                        (96.0, 0.93162)
                        (100.0, 0.92655)
                        (140.0, 0.88838)
                        (180.0, 0.84554)
                        (220.0, 0.80345)
                        (260.0, 0.76196)
                        (300.0, 0.72064)
                        (340.0, 0.67904)
                        (380.0, 0.63665)
                        (420.0, 0.59294)
                        (460.0, 0.54725)
                        (500.0, 0.49881)
                        (540.0, 0.44672)
                        (580.0, 0.38997)
                        (620.0, 0.3277)
                        (660.0, 0.26003)
                        (700.0, 0.19038)
                        (740.0, 0.12889)
                        (780.0, 0.08857)
                        (820.0, 0.06932)
                        (860.0, 0.06072)
                        (900.0, 0.05635)
                        (940.0, 0.05387)
                        (980.0, 0.0524)
                        (1000.0, 0.05189)
    					};

                        \addplot[
    					color=black,
    					line width = 1pt,
    					no marks
    					] coordinates {
                        (0.0, 0.66667)
                        (0.04, 0.67517)
                        (0.08, 0.68449)
                        (0.12, 0.69424)
                        (0.16, 0.70428)
                        (0.2, 0.71452)
                        (0.24, 0.72486)
                        (0.28, 0.73526)
                        (0.32, 0.74565)
                        (0.36, 0.756)
                        (0.4, 0.76628)
                        (0.44, 0.77643)
                        (0.48, 0.78645)
                        (0.52, 0.79631)
                        (0.56, 0.80598)
                        (0.6, 0.81546)
                        (0.64, 0.82472)
                        (0.68, 0.83376)
                        (0.72, 0.84257)
                        (0.76, 0.85114)
                        (0.8, 0.85948)
                        (0.84, 0.86758)
                        (0.88, 0.87543)
                        (0.92, 0.88304)
                        (0.96, 0.89042)
                        (1.0, 0.89755)
                        (1.4, 0.9505)
                        (1.8, 0.98918)
                        (2.2, 1.01663)
                        (2.6, 1.03581)
                        (3.0, 1.04907)
                        (3.4, 1.05817)
                        (3.8, 1.06435)
                        (4.2, 1.0685)
                        (4.6, 1.07122)
                        (5.0, 1.07293)
                        (5.4, 1.07393)
                        (5.8, 1.07441)
                        (6.2, 1.07452)
                        (6.6, 1.07435)
                        (7.0, 1.074)
                        (7.4, 1.07349)
                        (7.8, 1.07288)
                        (8.2, 1.07218)
                        (8.6, 1.07142)
                        (9.0, 1.07062)
                        (9.4, 1.06979)
                        (9.8, 1.06893)
                        (12.0, 1.06483)
                        (16.0, 1.05638)
                        (20.0, 1.04797)
                        (24.0, 1.03993)
                        (28.0, 1.03228)
                        (32.0, 1.02497)
                        (36.0, 1.01795)
                        (40.0, 1.0112)
                        (44.0, 1.00467)
                        (48.0, 0.99833)
                        (52.0, 0.99216)
                        (56.0, 0.98615)
                        (60.0, 0.98027)
                        (64.0, 0.97451)
                        (68.0, 0.96886)
                        (72.0, 0.96331)
                        (76.0, 0.95784)
                        (80.0, 0.95246)
                        (84.0, 0.94715)
                        (88.0, 0.94191)
                        (92.0, 0.93673)
                        (96.0, 0.93161)
                        (100.0, 0.92654)
                        (140.0, 0.88826)
                        (180.0, 0.84501)
                        (220.0, 0.80195)
                        (260.0, 0.75872)
                        (300.0, 0.71471)
                        (340.0, 0.6693)
                        (380.0, 0.62184)
                        (420.0, 0.57161)
                        (460.0, 0.51778)
                        (500.0, 0.4594)
                        (540.0, 0.39544)
                        (580.0, 0.3252)
                        (620.0, 0.24948)
                        (660.0, 0.17374)
                        (700.0, 0.11142)
                        (740.0, 0.07469)
                        (780.0, 0.05779)
                        (820.0, 0.04928)
                        (860.0, 0.04401)
                        (900.0, 0.04027)
                        (940.0, 0.03741)
                        (980.0, 0.03512)
                        (1000.0, 0.03412)
    					};
                        \legend{$L=5$,$L=10$,$L=100$}
    				\end{axis}
    			\end{tikzpicture}
    		\caption{\(\gamma = 0.7\)}
            \label{subfig:gamma=0.7}
    	\end{subfigure}
    	\hfill
        \begin{subfigure}[t]{0.32\linewidth}
            \centering
            \begin{tikzpicture}
                \begin{axis}[
                    xlabel={$t$},
                    ylabel={$\rho^L_2(t)$},
                    xlabel style={font=\small\sffamily},
                    ylabel style={font=\small\sffamily},
                    tick label style={font=\footnotesize},
                    width=5.5cm, height=5cm,
                    xmin=0, xmax=4.00004,
                    xtick={0, 0.30103, 1.04139, 2.00432, 3.00043, 4.00004},
                    xticklabels={0,1,10,100,1000,10000},
                    minor xtick={0.47712, 0.60206, 0.69897, 0.77815, 0.84510, 0.90309, 0.95424, 1,
                                1.32222, 1.49136, 1.61278, 1.70757, 1.78533, 1.85126, 1.90849, 1.95904,
                                2.30320, 2.47857, 2.60314, 2.69984, 2.77887, 2.84572, 2.90363, 2.95472,
                                3.30125, 3.47727, 3.60217, 3.69906, 3.77822, 3.84516, 3.90314, 3.95429},
                    ymin=0.0, ymax=1.6,
                    grid=both,
                    major grid style={solid, gray!60},
                    minor grid style={solid, gray!15},
                    minor tick length=0pt,
                    legend style={font=\scriptsize\sffamily, inner sep=2pt, outer sep=0pt},
                    every axis plot/.append style={x filter/.expression={log10(x+1)}},
                    ]

                    \addplot[
                    color=black,
                    line width = 1pt,
                    dotted,
                    no marks
                    ] coordinates {
                    (0.0, 0.66667)
                    (0.05, 0.68423)
                    (0.1, 0.70377)
                    (0.15, 0.7247)
                    (0.2, 0.74671)
                    (0.25, 0.76959)
                    (0.3, 0.79312)
                    (0.35, 0.81711)
                    (0.4, 0.84136)
                    (0.45, 0.8657)
                    (0.5, 0.88997)
                    (0.55, 0.91401)
                    (0.6, 0.93769)
                    (0.65, 0.9609)
                    (0.7, 0.98353)
                    (0.75, 1.0055)
                    (0.8, 1.02675)
                    (0.85, 1.04722)
                    (0.9, 1.06687)
                    (0.95, 1.08567)
                    (1.0, 1.10362)
                    (1.5, 1.2359)
                    (2.0, 1.30601)
                    (2.5, 1.34085)
                    (3.0, 1.35817)
                    (4.0, 1.37192)
                    (5.0, 1.37633)
                    (6.0, 1.378)
                    (7.0, 1.37866)
                    (8.0, 1.37885)
                    (9.0, 1.37879)
                    (10.0, 1.37858)
                    (20.0, 1.37442)
                    (30.0, 1.36993)
                    (40.0, 1.3659)
                    (50.0, 1.36236)
                    (60.0, 1.35927)
                    (70.0, 1.3566)
                    (80.0, 1.35427)
                    (90.0, 1.35224)
                    (100.0, 1.35048)
                    (200.0, 1.34171)
                    (300.0, 1.33947)
                    (400.0, 1.3389)
                    (500.0, 1.33875)
                    (600.0, 1.33871)
                    (700.0, 1.3387)
                    (800.0, 1.33869)
                    (900.0, 1.33869)
                    (1000.0, 1.33869)
                    (2000.0, 1.33871)
                    (3000.0, 1.33871)
                    (4000.0, 1.33871)
                    (5000.0, 1.33871)
                    (6000.0, 1.33871)
                    (7000.0, 1.33871)
                    (8000.0, 1.33871)
                    (9000.0, 1.33871)
                    (10000.0, 1.33871)
                    };

                    \addplot[
                    color=black,
                    line width = 1pt,
                    dashed,
                    no marks
                    ] coordinates {
                    (0.0, 0.66667)
                    (0.05, 0.68423)
                    (0.1, 0.70377)
                    (0.15, 0.7247)
                    (0.2, 0.74671)
                    (0.25, 0.76959)
                    (0.3, 0.79312)
                    (0.35, 0.81711)
                    (0.4, 0.84136)
                    (0.45, 0.8657)
                    (0.5, 0.88997)
                    (0.55, 0.91401)
                    (0.6, 0.93769)
                    (0.65, 0.9609)
                    (0.7, 0.98353)
                    (0.75, 1.0055)
                    (0.8, 1.02675)
                    (0.85, 1.04722)
                    (0.9, 1.06687)
                    (0.95, 1.08567)
                    (1.0, 1.10362)
                    (1.5, 1.2359)
                    (2.0, 1.30601)
                    (2.5, 1.34085)
                    (3.0, 1.35817)
                    (4.0, 1.37192)
                    (5.0, 1.37633)
                    (6.0, 1.378)
                    (7.0, 1.37866)
                    (8.0, 1.37885)
                    (9.0, 1.37879)
                    (10.0, 1.37858)
                    (20.0, 1.37442)
                    (30.0, 1.36988)
                    (40.0, 1.36567)
                    (50.0, 1.36176)
                    (60.0, 1.35812)
                    (70.0, 1.35469)
                    (80.0, 1.35144)
                    (90.0, 1.34835)
                    (100.0, 1.34538)
                    (200.0, 1.32081)
                    (300.0, 1.30195)
                    (400.0, 1.28722)
                    (500.0, 1.27565)
                    (600.0, 1.26653)
                    (700.0, 1.25932)
                    (800.0, 1.25363)
                    (900.0, 1.24911)
                    (1000.0, 1.24553)
                    (2000.0, 1.2331)
                    (3000.0, 1.23171)
                    (4000.0, 1.23155)
                    (5000.0, 1.23153)
                    (6000.0, 1.23153)
                    (7000.0, 1.23153)
                    (8000.0, 1.23153)
                    (9000.0, 1.23153)
                    (10000.0, 1.23153)
                    };

                    \addplot[
                    color=black,
                    line width = 1pt,
                    no marks
                    ] coordinates {
                    (0.0, 0.66667)
                    (0.05, 0.68423)
                    (0.1, 0.70377)
                    (0.15, 0.7247)
                    (0.2, 0.74671)
                    (0.25, 0.76959)
                    (0.3, 0.79312)
                    (0.35, 0.81711)
                    (0.4, 0.84136)
                    (0.45, 0.8657)
                    (0.5, 0.88997)
                    (0.55, 0.91401)
                    (0.6, 0.93769)
                    (0.65, 0.9609)
                    (0.7, 0.98353)
                    (0.75, 1.0055)
                    (0.8, 1.02675)
                    (0.85, 1.04722)
                    (0.9, 1.06687)
                    (0.95, 1.08567)
                    (1.0, 1.10362)
                    (1.5, 1.2359)
                    (2.0, 1.30601)
                    (2.5, 1.34085)
                    (3.0, 1.35817)
                    (4.0, 1.37192)
                    (5.0, 1.37633)
                    (6.0, 1.378)
                    (7.0, 1.37866)
                    (8.0, 1.37885)
                    (9.0, 1.37879)
                    (10.0, 1.37858)
                    (20.0, 1.37442)
                    (30.0, 1.36988)
                    (40.0, 1.36567)
                    (50.0, 1.36176)
                    (60.0, 1.35812)
                    (70.0, 1.35469)
                    (80.0, 1.35144)
                    (90.0, 1.34834)
                    (100.0, 1.34538)
                    (200.0, 1.32049)
                    (300.0, 1.30005)
                    (400.0, 1.28201)
                    (500.0, 1.26547)
                    (600.0, 1.24997)
                    (700.0, 1.23523)
                    (800.0, 1.22105)
                    (900.0, 1.20731)
                    (1000.0, 1.19391)
                    (1500.0, 1.13105)
                    (2000.0, 1.06874)
                    (2500.0, 1.00466)
                    (3000.0, 0.93605)
                    (3500.0, 0.85945)
                    (4000.0, 0.76921)
                    (4500.0, 0.6536)
                    (5000.0, 0.47772)
                    (5500.0, 0.06786)
                    (6000.0, 0.02222)
                    (6500.0, 0.01714)
                    (7000.0, 0.01458)
                    (7500.0, 0.01294)
                    (8000.0, 0.01176)
                    (8500.0, 0.01087)
                    (9000.0, 0.01015)
                    (9500.0, 0.00956)
                    (9990.0, 0.00908)
                    };
                    
                \end{axis}
            \end{tikzpicture}
    		\caption{\(\gamma = 0.85\)}
            \label{subfig:gamma=1.0}
    	\end{subfigure}
		\caption{Plots of the long-time behaviour of the concentration functionals $\rho^L_2$ for numerical solutions of~\eqref{eqn:1d-cauchy} for $s=1$ with diffusive strength $D = 0.1$, attractive strengths $\gamma \in \{0.1,\,0.7,\,0.85\}$ and domain lengths $L \in \{5,\,10,\,100\}$. These simulations were run on a uniform grid with widths of $5 \times 10^{-4}$ and time steps of size $10^{-3}$ when $t \in [0,1]$ and of size $10^{n-2}$ when $t \in [10^n,10^{n+1}]$.}
        \label{fig:rho-evol}
	    \end{figure}   

        For sufficiently large values of \(\gamma\), Figures~\ref{subfig:gamma=0.7} and~\ref{subfig:gamma=1.0} suggest that \(\rho_2(t)\) corresponding to~\eqref{eqn:1d-agg-diff} initially increases to some transient peak concentration, before very slowly decaying to zero. With this behaviour in mind, we are motivated to determine an upper bound for the peak concentration
        \begin{equation}
            \hat{\rho}_p := \sup \{\rho_p(t) : t \in (0,\infty),\, \rho'_p(t) = 0\}.
        \end{equation}
        
        To achieve this, first fix \(s \in (0,1]\) and consider \(\gamma > 0\) large enough so that there exists a peak concentration of the solution of~\eqref{eqn:1d-agg-diff}. Further consider the family of operators \(\Phi_{p,\lambda}\) defined for each \(p \in [2,\infty)\) such that \(p > 1/s\) and \(\lambda \in (1-s,1-1/p]\) as
        \begin{equation}
            \Phi_{p,\lambda}(x) := \frac{pC_\Nash}{2} \norm{\nabla W_s}_{L^{d/\lambda}} \left(\frac{\gamma}{D}\right) x^\lambda.
        \end{equation}
        Since \(\Phi_{p,\lambda}\) is of the form~\eqref{defn:phi_intro} for the case \(M = d = 1,\, K = \gamma W_s,\, q = d/\lambda\) and \(A = C_\Nash/2\), it can be deduced from Theorem~\ref{thm:conc-dynam} that \(\rho_p'(t) \geq 0\) implies \(\rho_p(t) \leq \Phi_{p,\lambda}(\rho_p(t))\). Consequently, it follows that \(\hat{\rho}_p \leq \Phi_{p,\lambda}(\hat{\rho}_p)\), which is equivalent to 
        \begin{equation}\label{ineq:rho-hat-lambda-bdd}
            \hat{\rho}_p \leq C_{p,s,\lambda} \left(\gamma/D\right)^{1/(1-\lambda)}
        \end{equation}
        for some \(C_{p,s,\lambda} > 0\). As~\eqref{ineq:rho-hat-lambda-bdd} holds for any choice of \(\lambda \in (1-s,1-1/p]\), it becomes natural to ask if there exists some constant \(C_{p,s} > 0\) such that \(\hat{\rho}_p \leq C_{p,s} (\gamma/D)^{1/s}\) for sufficiently large \(\gamma\) and, moreover, whether the inequality is algebraically sharp. That is, for an arbitrary choice of $s \in (0,1]$ and \(p \in [2,\infty)\) such that \(p > 1/s\), does the peak concentration of solutions of~\eqref{eqn:1d-agg-diff} satisfy
        \begin{equation}\label{eqn:rho-gamma-power}
            \lim_{(\gamma/D) \to \infty} \frac{\hat{\rho}_p}{(\gamma/D)^{1/s}} = C_{p,s}
        \end{equation}
        for some constant \(C_{p,s} > 0\)?

        To determine if~\eqref{eqn:rho-gamma-power} holds, we can approximate \(\hat{\rho}_p\) by the peak concentration corresponding to~\eqref{eqn:1d-cauchy} since it appears to be attained in the short term. Numerical simulations of~\eqref{eqn:1d-cauchy} for \(s=0.5\) and \(L = 5\) along with different choices of \(\gamma/D\) seem to suggest that \(\hat{\rho}_p\) satisfies the asymptotic behaviour in~\eqref{eqn:rho-gamma-power}. The log--log plot of $\hat{\rho}_3$ against $\gamma/D$ given in Figure~\ref{subfig:p=3} illustrates this relationship clearly for $s = 0.5$ as the values of $\hat{\rho}_3$ tend towards $C_{3,0.5} (\gamma/D)^2$, where $C_{3,0.5} \approx 0.0469$.
        
        Even though Theorem~\ref{thm:conc-dynam} only gives the bound~\eqref{ineq:rho-hat-lambda-bdd} for \(s \in (0,1]\), the numerical solutions further suggest that~\eqref{eqn:rho-gamma-power} also holds for $s > 1$. For $s = 2$, Figure~\ref{subfig:p=3} shows the values of $\hat{\rho}_3$ tend towards $C_{3,2} (\gamma/D)^{0.5}$, where $C_{3,2} \approx 0.3712$. Due to this, it seems feasible to redefine the critical regularity of kernels of the form \(\gamma W_s\) to be \(\bar{\lambda} = 1 - s\) for all \(s \in (0,\infty)\). In particular, the authors believe that there exists some general definition of \(\bar{\lambda}\) such that if a probability density function \(W\) has critical regularity \(\bar{\lambda} \in (-\infty,1)\), then kernels of the form \(\gamma W\) satisfy
        \begin{equation}
            \lim_{(\gamma/D) \to \infty} \frac{\hat{\rho}_p}{(\gamma/D)^{1/(1-\bar{\lambda})}} = C_{p,\bar{\lambda}}
        \end{equation}
        for some \(C_{p,\bar{\lambda}} > 0\). Furthermore, while~\eqref{ineq:rho-hat-lambda-bdd} can be obtained via Theorem~\ref{thm:conc-dynam} for restricted choices of \(p,\, s\) and \(\lambda\), said restriction appears to only be technical. For instance, the edge case \(p = \infty\) falls outside the theory presented in this work, yet Figure~\ref{subfig:p=inf} demonstrates that~\eqref{eqn:rho-gamma-power} still follows with constants \(C_{\infty,0.5} \approx 0.0669\) and \(C_{\infty,2} \approx 0.4886\).

        \begin{figure}[!ht]
    	\centering
    	\begin{subfigure}[t]{0.47\linewidth}
            \centering
    			\begin{tikzpicture}
    				\begin{axis}[
    					xlabel={$\log(\gamma/D)$},
    					ylabel={$\log(\hat{\rho}_3)$},
    					xlabel style={font=\small\sffamily, yshift=0em},
    					ylabel style={font=\small\sffamily, yshift=0.0em},
    					tick label style={font=\footnotesize},
    					grid=major,
    					grid style={gray!30},
    					width=7.5cm,
    					height=6cm,
    					xmin=1.3,
    					xmax=3.6,
    					ymin=-0.6,
    					ymax=4.5,
                        ytick={0,1,2,3,4},
    					legend pos=north west,
    					legend style={font=\scriptsize\sffamily, inner sep=2pt, outer sep=0pt},
    					mark size=1.5pt,
    					]
    
    					\addplot[
    					color=black,
    					mark=*,
    					only marks
    					] coordinates {
                        (1.4, -0.28585)
                        (1.5, -0.19536)
                        (1.6, -0.03089)
                        (1.7, 0.20533)
                        (1.8, 0.45677)
                        (1.9, 0.69274)
                        (2.0, 0.91432)
                        (2.1, 1.1269)
                        (2.2, 1.33177)
                        (2.3, 1.5348)
                        (2.4, 1.73657)
                        (2.5, 1.93771)
                        (2.6, 2.13849)
                        (2.7, 2.33903)
                        (2.8, 2.53938)
                        (2.9, 2.73959)
                        (3.0, 2.93977)
                        (3.1, 3.1397)
                        (3.2, 3.3395)
                        (3.3, 3.53917)
                        (3.4, 3.73867)
                        (3.5, 3.93798)
    					};
    					\addplot[
    					color=black,
    					line width=1pt,
    					no marks,
                        dotted,
    					domain=1.0:4.0
    					] {-3.0602 + 2*x}
                        node[pos=0.67, below right, font=\footnotesize] {$s=0.5$};

                        \addplot[
    					color=black,
    					mark=diamond*,
                        mark size=2pt,
    					only marks
    					] coordinates {
                        (1.6, -0.34038)
                        (1.7, -0.31693)
                        (1.8, -0.27459)
                        (1.9, -0.20752)
                        (2.0, -0.11947)
                        (2.1, -0.02913)
                        (2.2, 0.05226)
                        (2.3, 0.12392)
                        (2.4, 0.18806)
                        (2.5, 0.24698)
                        (2.6, 0.30241)
                        (2.7, 0.35557)
                        (2.8, 0.40731)
                        (2.9, 0.45823)
                        (3.0, 0.50871)
                        (3.1, 0.55894)
                        (3.2, 0.60905)
                        (3.3, 0.6591)
                        (3.4, 0.70912)
                        (3.5, 0.75913)
    					};
    					\addplot[
    					color=black,
    					line width=1pt,
    					no marks,
                        dotted,
    					domain=1.0:4.0
    					] {-0.9909 + 0.5*x}
                        node[pos=0.67, below right, font=\footnotesize] {$s=2$};
    				\end{axis}
    			\end{tikzpicture}
    		\caption{\(p=3\)}
            \label{subfig:p=3}
    	\end{subfigure}
    	\hspace{0.5em}
        \begin{subfigure}[t]{0.47\linewidth}
            \centering
    			\begin{tikzpicture}
    				\begin{axis}[
    					xlabel={$\log(\gamma/D)$},
    					ylabel={$\log(\hat{\rho}_\infty)$},
    					xlabel style={font=\small\sffamily, yshift=0em},
    					ylabel style={font=\small\sffamily, yshift=0.0em},
    					tick label style={font=\footnotesize},
    					grid=major,
    					grid style={gray!30},
    					width=7.5cm,
    					height=6cm,
    					xmin=1.3,
    					xmax=3.6,
    					ymin=-0.6,
    					ymax=4.5,
                        ytick={0,1,2,3,4},
    					legend pos=north west,
    					legend style={font=\scriptsize\sffamily, inner sep=2pt, outer sep=0pt},
    					mark size=1.5pt,
    					]
    
    					\addplot[
    					color=black,
    					mark=*,
    					only marks
    					] coordinates {
                        (1.4, 0.06047)
                        (1.5, 0.1696)
                        (1.6, 0.34386)
                        (1.7, 0.57846)
                        (1.8, 0.8241)
                        (1.9, 1.05546)
                        (2.0, 1.27412)
                        (2.1, 1.48495)
                        (2.2, 1.68914)
                        (2.3, 1.89174)
                        (2.4, 2.09325)
                        (2.5, 2.29422)
                        (2.6, 2.49487)
                        (2.7, 2.69532)
                        (2.8, 2.89561)
                        (2.9, 3.09576)
                        (3.0, 3.2959)
                        (3.1, 3.49581)
                        (3.2, 3.69562)
                        (3.3, 3.89532)
                        (3.4, 4.09489)
                        (3.5, 4.29432)
    					};
    					\addplot[
    					color=black,
    					line width=1pt,
    					no marks,
                        dotted,
    					domain=1.0:4.0
    					] {-2.7041 + 2*x}
                        node[pos=0.67, below right, font=\footnotesize] {$s=0.5$};

                        \addplot[
    					color=black,
    					mark=diamond*,
                        mark size=2pt,
    					only marks
    					] coordinates {
                        (1.8, 0.02801)
                        (1.9, 0.09632)
                        (2.0, 0.17998)
                        (2.1, 0.26376)
                        (2.2, 0.33936)
                        (2.3, 0.4067)
                        (2.4, 0.46783)
                        (2.5, 0.52473)
                        (2.6, 0.57886)
                        (2.7, 0.63122)
                        (2.8, 0.68249)
                        (2.9, 0.73315)
                        (3.0, 0.78349)
                        (3.1, 0.83365)
                        (3.2, 0.88373)
                        (3.3, 0.93376)
                        (3.4, 0.98378)
                        (3.5, 1.03379)
    					};
    					\addplot[
    					color=black,
    					line width=1pt,
    					no marks,
                        dotted,
    					domain=1.0:4.0
    					] {-0.7162 + 0.5*x}
                        node[pos=0.67, below right, font=\footnotesize] {$s=2$};
    				\end{axis}
    			\end{tikzpicture}
    		\caption{\(p=\infty\)}
            \label{subfig:p=inf}
    	\end{subfigure}
		\caption{Plots of \((\log(\gamma/D),\, \log(\hat{\rho}_p))\) with respect to \(s \in \{0.5,2\}\), \(p \in \{3,\infty\}\), for values of \(\gamma/D\) for which \(\hat{\rho}_p\) exists, along with their apparent asymptotes. These simulations were run on a uniform grid of widths $2 \times 10^{-5}$ and uniform time steps of size $10^{-3}$ until \(\hat{\rho}_p\) was observed. Domain length and diffusion were fixed at \(L = 5\) and \(D=0.5\).}
	    \end{figure}
    \end{example}

    \begin{example}[Global existence due to interaction cycle property]\label{eg:num-good-cycle}

    In this example we illustrate the interaction cycle condition given in Theorem~\ref{thm:global-1} for the global existence of solutions to the following cross-aggregation--diffusion system, in \(d \in \{1,\, 2\}\) dimensions with \(\gamma_{12} = \gamma_{21} = 1\):
    \begin{equation}\label{eqn:x-agg-diff} \tag{X-Agg-Diff}
            \begin{dcases}
                \partial_t u_1 = D_1 \Delta u_1 - \nabla \cdot (u_1\, \nabla (W_{s_{12}} \ast u_2)), &\quad (t, x) \in  (0,\infty) \times \R^d\\
                \partial_t u_2 = D_2 \Delta u_2 - \nabla \cdot (u_2 \nabla (W_{s_{21}} \ast u_1)), &\quad (t, x) \in  (0,\infty) \times \R^d\\
                u_1(0,x) = u_2(0,x) = \max\{1 - \abs{x},\, 0\}, & \quad x \in \R^d,
            \end{dcases}
        \end{equation}
        with \(D_1,\, D_2 > 0\), and \(s_{12}, s_{21} > -1\).

        \begin{figure}[!ht]
        \centering
        \resizebox{0.5\linewidth}{!}{
            \begin{tikzpicture}
                \begin{axis}[
                    xmin=0, xmax=2,
                    ymin=0, ymax=2,
                    xlabel={$\bar{\lambda}_{12}$},
                    ylabel={$\bar{\lambda}_{21}$},
                    axis lines=box,
                    grid=major,
                    grid style={gray!10},
                    width=0.55\textwidth,
                    tick align=outside,
                    clip=true,
                    axis line style={-},
                    label style={font=\large},
                    tick label style={font=\small},
                ]

                \addplot[black, very thick, domain=0:2.2, samples=200]
                    ({x},{1/x});

                \addplot[black, thick, dashed, domain=0:2.2, samples=200]
                    ({x},{0.36/x});
        
                
                \filldraw[fill=black, opacity=0.1, draw=none]
                (axis cs:0,0) rectangle (axis cs:0.5,2);
                \addplot[fill=black,opacity=0.1,draw=none,domain=0.5:2,samples=600]
                ({x},{1/x})
                |- (axis cs:0.5,0)
                -- cycle;
                        
                \addplot[black,thick,dotted]
                        coordinates {(0,1) (2.2,1)};
                \addplot[black,thick,dotted]
                        coordinates {(1,0) (1,2.2)};
                        
                        
                
                
                \addplot[only marks, mark=*, mark size=2.5pt,black] coordinates {(0.6,0.6)};
                \node[anchor=north east,font=\small] at (axis cs:0.6,0.6)
                {$(0.6,0.6)$};
                
                \addplot[only marks, mark=*, mark size=2.5pt,black] coordinates {(1.08,0.33)};
                \node[anchor=south,font=\small] at (axis cs:1.2,0.35)
                {$(1.08,0.33)$};    

                \addplot[only marks, mark=*, mark size=2.5pt,black] coordinates {(1.2,0.3)};
                \node[anchor=north,font=\small] at (axis cs:1.25,0.25)
                {$(1.2,0.3)$}; 
                \end{axis}
                \end{tikzpicture}
                }
        \caption{A graph showing the region for which different pairs of critical regularities \((\bar{\lambda}_{12}, \bar{\lambda}_{21})\) imply a uniform-in-time bound for \(\frac{\rmd}{\rmd t} \mathcal{E}_P(t)\). The critical regularities corresponding to the choices $(s_{12}, s_{21}) \in \{(0.4, 0.4), (-0.08, 0.67), (-0.2,0.7)\}$ are included along with the level set \((\bar{\lambda}_{12} \bar{\lambda}_{21})^{1/2} = 0.6\).}
        \label{fig:good-lambda-range}
        \end{figure}
        
        Recall that Theorem~\ref{thm:global-1} is established by bounding the \(P\)-energy functional \(\mathcal{E}_P\) as defined in~\eqref{defn:E_P}. For the case \(N = 2\), we write the \((p_1,p_2)\)-energy functional as
        \begin{equation}
            \mathcal{E}_{(p_1,p_2)}(t) := \frac{1}{2(p_1-1)} \int_{\R^d} \abs{u_1(t)}^{p_1} \rmd x + \frac{1}{2(p_2-1)} \int_{\R^d} \abs{u_2(t)}^{p_2} \rmd x.
        \end{equation}

        For fixed \(d \in \{1,\,2\}\), consider \(s_{12}, s_{21} \in (-d+1, \infty)\) with corresponding critical regularities \(\bar{\lambda}_{12},\, \bar{\lambda}_{21}\). If \(\bar{\lambda}_{12}\bar{\lambda}_{21} < 1\), it can be demonstrated that there exists some sufficiently large \((p_1, p_2) \in [2,\infty)^2\) such that \(\frac{\rmd}{\rmd t}\mathcal{E}_{(p_1,p_2)}\) is uniformly bounded, hence~\eqref{eqn:x-agg-diff} can be continued in time up to any time \(T > 0\). To understand why, first notice that \(\bar{\lambda}_{12} \bar{\lambda}_{21} < 1\) implies the existence of some \(r > 0\) such that
        \begin{equation}\label{ineq:lambda-r}
            \bar{\lambda}_{12} < r < \frac{1}{\bar{\lambda}_{21}}.
        \end{equation} 
        Using the definition of \(\bar{\lambda}_{12}\) and \(\bar{\lambda}_{21}\), choose \(q_{12},\, q_{21} \in (1,\infty)\) such that \(\nabla W_{s_{12}} \in L^{q_{12}}(\R^d)\), \(\nabla W_{s_{21}} \in L^{q_{21}}(\R^d)\) and 
        \begin{equation}\label{ineq:q-chosen}
            \bar{\lambda}_{12} < \frac{d}{q_{12}} \leq r \leq \frac{q_{21}}{d} < \frac{1}{\bar{\lambda}_{21}}.
        \end{equation}
        By now choosing \((p_1, p_2) \in [2,\infty)\) such that \(p_1 \geq q_{21}'\), \(p_2 \geq q_{12}'\) and 
        \begin{equation}\label{ineq:p-lambda}
            r = \begin{dcases}
                \frac{p_2+1}{p_1+1}, \quad & d=1, \\
                \frac{p_2}{p_1}, \quad & d=2,
            \end{dcases}
        \end{equation}
        it can be deduced that~\eqref{ineq:r_ij} holds for \((i,j) \in \{(1,2),\, (2,1)\}\), thus it follows from Proposition~\ref{prop:L^p_est} that \(\frac{\rmd}{\rmd t}\mathcal{E}_{(p_1,p_2)}\) is uniformly bounded for our chosen \((p_1, p_2)\).
        
        When \(d = 2\) and \(s \in (-1,1)\), recall that \(W_{s}\) has critical regularity of \(\bar{\lambda} = 1-s\). In particular, we consider two examples, chosen to be \((s_{12}, s_{21}) \in \{(0.4, 0.4), (-0.08, 0.67)\}\). Observe that both choices of \((s_{12}, s_{21})\) satisfy \((\bar{\lambda}_{12} \bar{\lambda}_{21})^{1/2} \approx 0.6\), thus are in the ``good'' region in Figure~\ref{fig:good-lambda-range}. 
        
        For \((s_{12}, s_{21}) = (0.4,0.4)\), it follows that \(\bar{\lambda}_{12} = \bar{\lambda}_{21} = 0.6\), hence the choice of \(r = 1\) satisfies~\eqref{ineq:lambda-r}. By substituting in our values of \(\bar{\lambda}_{12}\), \(\bar{\lambda}_{21}\) and \(r\) into~\eqref{ineq:q-chosen} and~\eqref{ineq:p-lambda}, we choose \(q_{12},\, q_{21} \in [2, 10/3)\) and conclude that an admissible choice of \((p_1, p_2) \in [2,\infty)^2\) such that Proposition~\ref{prop:L^p_est} follows is \((p_1, p_2) = (p,p)\) for any \(p \geq 2\). For \((s_{12}, s_{21}) = (-0.08,0.67)\), we instead choose \(r = 1.25\). From~\eqref{ineq:q-chosen} we can then choose \(q_{12} \in [1.6, 1.\overline{851})\) and \(q_{21} \in [2.5,6.\overline{06})\). For the choice of \(q_{12} = 1.8\), we can then deduce that  \((p_1,p_2) = (2,2.5)\) is an admissible choice for Proposition~\ref{prop:L^p_est}.
        
        With our obtained admissible choices of \((p_1, p_2)\), we are now able to consider the numerical approximations of~\eqref{eqn:x-agg-diff}, which are obtained by computing solutions of the spatially truncated initial-boundary value problem 
        \begin{equation}\label{eqn:2d-cauchy}
            \begin{dcases}
                \partial_t u_1 = D_1 \Delta u_1 - \nabla \cdot (u_1\, \nabla (W_{s_{12}} \ast u_2)), &\quad (t, x) \in  (0,\infty) \times [-L,L]^d\\
                \partial_t u_2 = D_2 \Delta u_2 - \nabla \cdot (u_2 \nabla (W_{s_{21}} \ast u_1)), &\quad (t, x) \in  (0,\infty) \times [-L,L]^d\\
                u_1(0,x) = u_2(0,x) = \max\{1 - \abs{x},\, 0\}, & \quad x \in [-L,L]^d,
            \end{dcases}
        \end{equation}
        for some \(L > 1\), zero flux boundary conditions and spatial convolution defined by extending the solution by zero outside the domain. Similarly to the case in Example~\ref{eg:num-steady}, we suppose that the short-term behaviour of numerical solutions to~\eqref{eqn:2d-cauchy} is insensitive to the domain length, hence they are sufficiently good approximations of~\eqref{eqn:x-agg-diff} for the scope of this work. As expected from Proposition~\ref{prop:L^p_est}, we can observe from Figures~\ref{s_1-s_2=0.4,0.4} and~\ref{s_1-s_2=-0.08,0.67} that the value of \(\frac{\rmd}{\rmd t} \mathcal{E}_{(p_1, p_2)}\) appears uniformly bounded.

        \begin{figure}[!ht]
        	\centering
            \begin{subfigure}[t]{0.32\linewidth}
                \centering
        			\begin{tikzpicture}
        				\begin{axis}[
                            scale only axis,
        					xlabel={$t$},
        					ylabel={$\mathcal{E}_{(3,3)}(t)$},
        					xlabel style={font=\small\sffamily, yshift=0em},
        					ylabel style={font=\small\sffamily, yshift=0.0em},
        					tick label style={font=\footnotesize},
        					grid=major,
        					grid style={gray!30},
        					width=0.75\textwidth,
        					xmin=0.0,
        					xmax=15.0,
        					ymin=0.0,
        					ymax=14.0,
                            ytick={0,3,6,9,12},
        					legend pos=north west,
        					legend style={font=\scriptsize\sffamily, inner sep=2pt, outer sep=0pt},
        					mark size=2pt,
        					]
                            \addplot[thick, mark=none
        					] coordinates {
                            (0.0, 0.1571)
                            (0.1, 0.1925)
                            (0.2, 0.2446)
                            (0.3, 0.3178)
                            (0.4, 0.4175)
                            (0.5, 0.5493)
                            (0.6, 0.7183)
                            (0.7, 0.9291)
                            (0.8, 1.1854)
                            (0.9, 1.4897)
                            (1.0, 1.8434)
                            (1.1, 2.246)
                            (1.2, 2.6957)
                            (1.3, 3.1888)
                            (1.4, 3.7202)
                            (1.5, 4.283)
                            (1.6, 4.8695)
                            (1.7, 5.4712)
                            (1.8, 6.079)
                            (1.9, 6.6843)
                            (2.0, 7.2788)
                            (2.1, 7.8551)
                            (2.2, 8.4068)
                            (2.3, 8.9288)
                            (2.4, 9.4175)
                            (2.5, 9.8701)
                            (2.6, 10.2851)
                            (2.7, 10.6622)
                            (2.8, 11.0018)
                            (2.9, 11.3048)
                            (3.0, 11.5729)
                            (3.1, 11.8079)
                            (3.2, 12.0118)
                            (3.3, 12.187)
                            (3.4, 12.3357)
                            (3.5, 12.4601)
                            (3.6, 12.5625)
                            (3.7, 12.6449)
                            (3.8, 12.7093)
                            (3.9, 12.7574)
                            (4.0, 12.7911)
                            (4.1, 12.8116)
                            (4.2, 12.8206)
                            (4.3, 12.8192)
                            (4.4, 12.8086)
                            (4.5, 12.7898)
                            (4.6, 12.7636)
                            (4.7, 12.731)
                            (4.8, 12.6927)
                            (4.9, 12.6492)
                            (5.0, 12.6013)
                            (5.1, 12.5494)
                            (5.2, 12.494)
                            (5.3, 12.4354)
                            (5.4, 12.3742)
                            (5.5, 12.3105)
                            (5.6, 12.2448)
                            (5.7, 12.1772)
                            (5.8, 12.1081)
                            (5.9, 12.0375)
                            (6.0, 11.9658)
                            (6.1, 11.893)
                            (6.2, 11.8194)
                            (6.3, 11.745)
                            (6.4, 11.67)
                            (6.5, 11.5946)
                            (6.6, 11.5187)
                            (6.7, 11.4426)
                            (6.8, 11.3662)
                            (6.9, 11.2897)
                            (7.0, 11.2131)
                            (7.1, 11.1365)
                            (7.2, 11.06)
                            (7.3, 10.9835)
                            (7.4, 10.9071)
                            (7.5, 10.831)
                            (7.6, 10.755)
                            (7.7, 10.6793)
                            (7.8, 10.6038)
                            (7.9, 10.5286)
                            (8.0, 10.4538)
                            (8.1, 10.3793)
                            (8.2, 10.3052)
                            (8.3, 10.2314)
                            (8.4, 10.1581)
                            (8.5, 10.0851)
                            (8.6, 10.0126)
                            (8.7, 9.9405)
                            (8.8, 9.8689)
                            (8.9, 9.7977)
                            (9.0, 9.727)
                            (9.1, 9.6567)
                            (9.2, 9.5869)
                            (9.3, 9.5176)
                            (9.4, 9.4488)
                            (9.5, 9.3804)
                            (9.6, 9.3126)
                            (9.7, 9.2452)
                            (9.8, 9.1784)
                            (9.9, 9.112)
                            (10.0, 9.0461)
                            (10.1, 8.9807)
                            (10.2, 8.9159)
                            (10.3, 8.8515)
                            (10.4, 8.7876)
                            (10.5, 8.7242)
                            (10.6, 8.6613)
                            (10.7, 8.5989)
                            (10.8, 8.537)
                            (10.9, 8.4755)
                            (11.0, 8.4146)
                            (11.1, 8.3541)
                            (11.2, 8.2942)
                            (11.3, 8.2347)
                            (11.4, 8.1757)
                            (11.5, 8.1171)
                            (11.6, 8.059)
                            (11.7, 8.0014)
                            (11.8, 7.9443)
                            (11.9, 7.8876)
                            (12.0, 7.8314)
                            (12.1, 7.7757)
                            (12.2, 7.7204)
                            (12.3, 7.6655)
                            (12.4, 7.6111)
                            (12.5, 7.5572)
                            (12.6, 7.5037)
                            (12.7, 7.4506)
                            (12.8, 7.398)
                            (12.9, 7.3457)
                            (13.0, 7.294)
                            (13.1, 7.2426)
                            (13.2, 7.1916)
                            (13.3, 7.1411)
                            (13.4, 7.091)
                            (13.5, 7.0413)
                            (13.6, 6.992)
                            (13.7, 6.9431)
                            (13.8, 6.8946)
                            (13.9, 6.8465)
                            (14.0, 6.7988)
                            (14.1, 6.7514)
                            (14.2, 6.7045)
                            (14.3, 6.6579)
                            (14.4, 6.6117)
                            (14.5, 6.5659)
                            (14.6, 6.5205)
                            (14.7, 6.4754)
                            (14.8, 6.4307)
                            (14.9, 6.3863)
                            (15.0, 6.3423)
        					};
        				\end{axis}
        			\end{tikzpicture}
                
                    \begin{tikzpicture}
                        \begin{axis}[
                            scale only axis,
                            xlabel={$t$},
                            ylabel={$\frac{\rmd}{\rmd t} \mathcal{E}_{(3,3)}(t)$},
                            xlabel style={font=\small\sffamily, yshift=0em},
                            ylabel style={font=\small\sffamily, yshift=0.0em},
                            tick label style={font=\footnotesize},
                            grid=major,
                            grid style={gray!30},
                            width=0.75\textwidth,
                            xmin=0.0,
                            xmax=15.0,
                            ymin=-1.5,
                            ymax=7.0,
                            legend pos=north west,
                            legend style={font=\scriptsize\sffamily, inner sep=2pt, outer sep=0pt},
                            mark size=2pt,
                            ]
                            \addplot[thick, mark=none]
                            coordinates {
                            (0.0, 0.3011)
                            (0.1, 0.44829)
                            (0.2, 0.63942)
                            (0.3, 0.88229)
                            (0.4, 1.18076)
                            (0.5, 1.53319)
                            (0.6, 1.9343)
                            (0.7, 2.37611)
                            (0.8, 2.84835)
                            (0.9, 3.33852)
                            (1.0, 3.83221)
                            (1.1, 4.31331)
                            (1.2, 4.7649)
                            (1.3, 5.17033)
                            (1.4, 5.51423)
                            (1.5, 5.78386)
                            (1.6, 5.97008)
                            (1.7, 6.0679)
                            (1.8, 6.07683)
                            (1.9, 6.00055)
                            (2.0, 5.84639)
                            (2.1, 5.62439)
                            (2.2, 5.34638)
                            (2.3, 5.02508)
                            (2.4, 4.67307)
                            (2.5, 4.30218)
                            (2.6, 3.92297)
                            (2.7, 3.54443)
                            (2.8, 3.17421)
                            (2.9, 2.81772)
                            (3.0, 2.47882)
                            (3.1, 2.16014)
                            (3.2, 1.86338)
                            (3.3, 1.58928)
                            (3.4, 1.33788)
                            (3.5, 1.10866)
                            (3.6, 0.90069)
                            (3.7, 0.71278)
                            (3.8, 0.54359)
                            (3.9, 0.39169)
                            (4.0, 0.25565)
                            (4.1, 0.13407)
                            (4.2, 0.02559)
                            (4.3, -0.07106)
                            (4.4, -0.15706)
                            (4.5, -0.2335)
                            (4.6, -0.30137)
                            (4.7, -0.36159)
                            (4.8, -0.41495)
                            (4.9, -0.46221)
                            (5.0, -0.504)
                            (5.1, -0.54091)
                            (5.2, -0.57348)
                            (5.3, -0.60217)
                            (5.4, -0.62738)
                            (5.5, -0.6495)
                            (5.6, -0.66884)
                            (5.7, -0.6857)
                            (5.8, -0.70033)
                            (5.9, -0.71298)
                            (6.0, -0.72383)
                            (6.1, -0.73309)
                            (6.2, -0.7409)
                            (6.3, -0.74741)
                            (6.4, -0.75276)
                            (6.5, -0.75705)
                            (6.6, -0.7604)
                            (6.7, -0.7629)
                            (6.8, -0.76462)
                            (6.9, -0.76565)
                            (7.0, -0.76605)
                            (7.1, -0.76589)
                            (7.2, -0.76521)
                            (7.3, -0.76406)
                            (7.4, -0.7625)
                            (7.5, -0.76056)
                            (7.6, -0.75828)
                            (7.7, -0.75568)
                            (7.8, -0.75281)
                            (7.9, -0.74969)
                            (8.0, -0.74634)
                            (8.1, -0.74278)
                            (8.2, -0.73905)
                            (8.3, -0.73515)
                            (8.4, -0.7311)
                            (8.5, -0.72691)
                            (8.6, -0.72262)
                            (8.7, -0.71821)
                            (8.8, -0.71372)
                            (8.9, -0.70914)
                            (9.0, -0.70449)
                            (9.1, -0.69978)
                            (9.2, -0.69502)
                            (9.3, -0.69021)
                            (9.4, -0.68536)
                            (9.5, -0.68048)
                            (9.6, -0.67557)
                            (9.7, -0.67064)
                            (9.8, -0.66569)
                            (9.9, -0.66074)
                            (10.0, -0.65578)
                            (10.1, -0.65081)
                            (10.2, -0.64585)
                            (10.3, -0.64089)
                            (10.4, -0.63593)
                            (10.5, -0.63099)
                            (10.6, -0.62606)
                            (10.7, -0.62114)
                            (10.8, -0.61624)
                            (10.9, -0.61136)
                            (11.0, -0.6065)
                            (11.1, -0.60166)
                            (11.2, -0.59685)
                            (11.3, -0.59206)
                            (11.4, -0.5873)
                            (11.5, -0.58256)
                            (11.6, -0.57786)
                            (11.7, -0.57318)
                            (11.8, -0.56853)
                            (11.9, -0.56391)
                            (12.0, -0.55933)
                            (12.1, -0.55478)
                            (12.2, -0.55026)
                            (12.3, -0.54577)
                            (12.4, -0.54132)
                            (12.5, -0.5369)
                            (12.6, -0.53252)
                            (12.7, -0.52817)
                            (12.8, -0.52385)
                            (12.9, -0.51957)
                            (13.0, -0.51533)
                            (13.1, -0.51112)
                            (13.2, -0.50695)
                            (13.3, -0.50281)
                            (13.4, -0.49871)
                            (13.5, -0.49464)
                            (13.6, -0.49061)
                            (13.7, -0.48661)
                            (13.8, -0.48265)
                            (13.9, -0.47873)
                            (14.0, -0.47484)
                            (14.1, -0.47098)
                            (14.2, -0.46716)
                            (14.3, -0.46338)
                            (14.4, -0.45963)
                            (14.5, -0.45591)
                            (14.6, -0.45223)
                            (14.7, -0.44858)
                            (14.8, -0.44497)
                            (14.9, -0.44139)
                            };
                        \end{axis}
                    \end{tikzpicture}
            		  \caption{{\footnotesize $d = 2$, $D_2 = 0.05$, $(s_{12}, s_{21}) = (0.4, 0.4)$}}
                \label{s_1-s_2=0.4,0.4}
            \end{subfigure}
            \hfill
            \begin{subfigure}[t]{0.32\linewidth}
                \centering
        			\begin{tikzpicture}
        				\begin{axis}[
                            scale only axis,
        					xlabel={$t$},
        					ylabel={$\mathcal{E}_{(2,2.5)}(t)$},
        					xlabel style={font=\small\sffamily, yshift=0em},
        					ylabel style={font=\small\sffamily, yshift=0.0em},
        					tick label style={font=\footnotesize},
        					grid=major,
        					grid style={gray!30},
        					width=0.75\textwidth,
        					xmin=0.0,
        					xmax=15.0,
        					ymin=0.0,
        					ymax=24.0,
                            ytick={0,5,10,15,20},
        					legend pos=north west,
        					legend style={font=\scriptsize\sffamily, inner sep=2pt, outer sep=0pt},
        					mark size=2pt,
        					]
                            \addplot[thick, mark=none
        					] coordinates {
                            (0.0, 0.3948)
                            (0.1, 0.427)
                            (0.2, 0.4773)
                            (0.3, 0.5455)
                            (0.4, 0.6338)
                            (0.5, 0.7454)
                            (0.6, 0.8835)
                            (0.7, 1.052)
                            (0.8, 1.2551)
                            (0.9, 1.4972)
                            (1.0, 1.7828)
                            (1.1, 2.1165)
                            (1.2, 2.5029)
                            (1.3, 2.9456)
                            (1.4, 3.4478)
                            (1.5, 4.0109)
                            (1.6, 4.6351)
                            (1.7, 5.318)
                            (1.8, 6.0553)
                            (1.9, 6.8404)
                            (2.0, 7.6644)
                            (2.1, 8.5167)
                            (2.2, 9.3855)
                            (2.3, 10.2587)
                            (2.4, 11.1242)
                            (2.5, 11.971)
                            (2.6, 12.7892)
                            (2.7, 13.5712)
                            (2.8, 14.3108)
                            (2.9, 15.0043)
                            (3.0, 15.6495)
                            (3.1, 16.2457)
                            (3.2, 16.7938)
                            (3.3, 17.2954)
                            (3.4, 17.7529)
                            (3.5, 18.1691)
                            (3.6, 18.5472)
                            (3.7, 18.8903)
                            (3.8, 19.2015)
                            (3.9, 19.4839)
                            (4.0, 19.7402)
                            (4.1, 19.9731)
                            (4.2, 20.185)
                            (4.3, 20.378)
                            (4.4, 20.554)
                            (4.5, 20.7148)
                            (4.6, 20.862)
                            (4.7, 20.9969)
                            (4.8, 21.1208)
                            (4.9, 21.2347)
                            (5.0, 21.3396)
                            (5.1, 21.4363)
                            (5.2, 21.5257)
                            (5.3, 21.6083)
                            (5.4, 21.6848)
                            (5.5, 21.7556)
                            (5.6, 21.8213)
                            (5.7, 21.8824)
                            (5.8, 21.9391)
                            (5.9, 21.9918)
                            (6.0, 22.0408)
                            (6.1, 22.0864)
                            (6.2, 22.1289)
                            (6.3, 22.1685)
                            (6.4, 22.2053)
                            (6.5, 22.2396)
                            (6.6, 22.2716)
                            (6.7, 22.3013)
                            (6.8, 22.3291)
                            (6.9, 22.3549)
                            (7.0, 22.3789)
                            (7.1, 22.4012)
                            (7.2, 22.4219)
                            (7.3, 22.4412)
                            (7.4, 22.4591)
                            (7.5, 22.4757)
                            (7.6, 22.491)
                            (7.7, 22.5052)
                            (7.8, 22.5183)
                            (7.9, 22.5303)
                            (8.0, 22.5414)
                            (8.1, 22.5516)
                            (8.2, 22.5608)
                            (8.3, 22.5693)
                            (8.4, 22.577)
                            (8.5, 22.584)
                            (8.6, 22.5902)
                            (8.7, 22.5958)
                            (8.8, 22.6008)
                            (8.9, 22.6052)
                            (9.0, 22.609)
                            (9.1, 22.6123)
                            (9.2, 22.6151)
                            (9.3, 22.6174)
                            (9.4, 22.6193)
                            (9.5, 22.6207)
                            (9.6, 22.6217)
                            (9.7, 22.6223)
                            (9.8, 22.6225)
                            (9.9, 22.6224)
                            (10.0, 22.6219)
                            (10.1, 22.6211)
                            (10.2, 22.62)
                            (10.3, 22.6186)
                            (10.4, 22.617)
                            (10.5, 22.615)
                            (10.6, 22.6128)
                            (10.7, 22.6104)
                            (10.8, 22.6077)
                            (10.9, 22.6049)
                            (11.0, 22.6018)
                            (11.1, 22.5985)
                            (11.2, 22.595)
                            (11.3, 22.5913)
                            (11.4, 22.5875)
                            (11.5, 22.5835)
                            (11.6, 22.5793)
                            (11.7, 22.575)
                            (11.8, 22.5705)
                            (11.9, 22.5659)
                            (12.0, 22.5612)
                            (12.1, 22.5564)
                            (12.2, 22.5514)
                            (12.3, 22.5463)
                            (12.4, 22.541)
                            (12.5, 22.5356)
                            (12.6, 22.5301)
                            (12.7, 22.5244)
                            (12.8, 22.5187)
                            (12.9, 22.5129)
                            (13.0, 22.5071)
                            (13.1, 22.5012)
                            (13.2, 22.4952)
                            (13.3, 22.4892)
                            (13.4, 22.4831)
                            (13.5, 22.4769)
                            (13.6, 22.4707)
                            (13.7, 22.4645)
                            (13.8, 22.4582)
                            (13.9, 22.4519)
                            (14.0, 22.4455)
                            (14.1, 22.439)
                            (14.2, 22.4325)
                            (14.3, 22.426)
                            (14.4, 22.4195)
                            (14.5, 22.4129)
                            (14.6, 22.4062)
                            (14.7, 22.3995)
                            (14.8, 22.3928)
                            (14.9, 22.3861)
                            (15.0, 22.3794)
                            };
        				\end{axis}
        			\end{tikzpicture}
                
                    \begin{tikzpicture}
                        \begin{axis}[
                            scale only axis,
                            xlabel={$t$},
                            ylabel={$\frac{\rmd}{\rmd t} \mathcal{E}_{(2,2.5)}(t)$},
                            xlabel style={font=\small\sffamily, yshift=0em},
                            ylabel style={font=\small\sffamily, yshift=0.0em},
                            tick label style={font=\footnotesize},
                            grid=major,
                            grid style={gray!30},
                            width=0.75\textwidth,
                            xmin=0.0,
                            xmax=15.0,
                            ymin=-1.0,
                            ymax=9.5,
                            ytick={0,2,4,6,8},
                            legend pos=north west,
                            legend style={font=\scriptsize\sffamily, inner sep=2pt, outer sep=0pt},
                            mark size=2pt,
                            ]
                            \addplot[thick, mark=none]
                            coordinates {
                            (0.0, 0.24936)
                            (0.1, 0.43478)
                            (0.2, 0.60749)
                            (0.3, 0.7983)
                            (0.4, 1.01714)
                            (0.5, 1.26897)
                            (0.6, 1.55745)
                            (0.7, 1.8859)
                            (0.8, 2.25741)
                            (0.9, 2.67433)
                            (1.0, 3.13767)
                            (1.1, 3.64632)
                            (1.2, 4.19633)
                            (1.3, 4.78019)
                            (1.4, 5.38638)
                            (1.5, 5.99926)
                            (1.6, 6.59948)
                            (1.7, 7.16493)
                            (1.8, 7.67253)
                            (1.9, 8.09988)
                            (2.0, 8.42775)
                            (2.1, 8.64178)
                            (2.2, 8.73419)
                            (2.3, 8.70444)
                            (2.4, 8.55905)
                            (2.5, 8.31049)
                            (2.6, 7.97557)
                            (2.7, 7.57355)
                            (2.8, 7.12432)
                            (2.9, 6.64678)
                            (3.0, 6.15774)
                            (3.1, 5.67124)
                            (3.2, 5.19829)
                            (3.3, 4.74696)
                            (3.4, 4.32264)
                            (3.5, 3.92852)
                            (3.6, 3.56595)
                            (3.7, 3.23496)
                            (3.8, 2.93456)
                            (3.9, 2.66316)
                            (4.0, 2.41873)
                            (4.1, 2.19909)
                            (4.2, 2.00196)
                            (4.3, 1.82514)
                            (4.4, 1.66652)
                            (4.5, 1.52413)
                            (4.6, 1.39618)
                            (4.7, 1.28104)
                            (4.8, 1.17726)
                            (4.9, 1.08355)
                            (5.0, 0.99875)
                            (5.1, 0.92186)
                            (5.2, 0.85198)
                            (5.3, 0.78835)
                            (5.4, 0.73026)
                            (5.5, 0.67713)
                            (5.6, 0.62842)
                            (5.7, 0.58368)
                            (5.8, 0.5425)
                            (5.9, 0.50453)
                            (6.0, 0.46945)
                            (6.1, 0.43698)
                            (6.2, 0.40689)
                            (6.3, 0.37894)
                            (6.4, 0.35296)
                            (6.5, 0.32877)
                            (6.6, 0.30621)
                            (6.7, 0.28515)
                            (6.8, 0.26547)
                            (6.9, 0.24706)
                            (7.0, 0.22981)
                            (7.1, 0.21364)
                            (7.2, 0.19846)
                            (7.3, 0.1842)
                            (7.4, 0.1708)
                            (7.5, 0.15818)
                            (7.6, 0.14631)
                            (7.7, 0.13512)
                            (7.8, 0.12457)
                            (7.9, 0.11461)
                            (8.0, 0.10521)
                            (8.1, 0.09633)
                            (8.2, 0.08794)
                            (8.3, 0.08)
                            (8.4, 0.07249)
                            (8.5, 0.06537)
                            (8.6, 0.05863)
                            (8.7, 0.05225)
                            (8.8, 0.04619)
                            (8.9, 0.04045)
                            (9.0, 0.035)
                            (9.1, 0.02982)
                            (9.2, 0.02491)
                            (9.3, 0.02024)
                            (9.4, 0.0158)
                            (9.5, 0.01159)
                            (9.6, 0.00757)
                            (9.7, 0.00376)
                            (9.8, 0.00012)
                            (9.9, -0.00333)
                            (10.0, -0.00663)
                            (10.1, -0.00976)
                            (10.2, -0.01275)
                            (10.3, -0.0156)
                            (10.4, -0.01831)
                            (10.5, -0.0209)
                            (10.6, -0.02336)
                            (10.7, -0.02571)
                            (10.8, -0.02796)
                            (10.9, -0.0301)
                            (11.0, -0.03214)
                            (11.1, -0.03408)
                            (11.2, -0.03594)
                            (11.3, -0.03772)
                            (11.4, -0.03941)
                            (11.5, -0.04103)
                            (11.6, -0.04257)
                            (11.7, -0.04404)
                            (11.8, -0.04545)
                            (11.9, -0.04679)
                            (12.0, -0.04807)
                            (12.1, -0.0493)
                            (12.2, -0.05047)
                            (12.3, -0.05183)
                            (12.4, -0.05356)
                            (12.5, -0.055)
                            (12.6, -0.05608)
                            (12.7, -0.05691)
                            (12.8, -0.05762)
                            (12.9, -0.05826)
                            (13.0, -0.05887)
                            (13.1, -0.05946)
                            (13.2, -0.06004)
                            (13.3, -0.06061)
                            (13.4, -0.06117)
                            (13.5, -0.06172)
                            (13.6, -0.06225)
                            (13.7, -0.06276)
                            (13.8, -0.06326)
                            (13.9, -0.06374)
                            (14.0, -0.0642)
                            (14.1, -0.06464)
                            (14.2, -0.06507)
                            (14.3, -0.06547)
                            (14.4, -0.06586)
                            (14.5, -0.06622)
                            (14.6, -0.06657)
                            (14.7, -0.0669)
                            (14.8, -0.06721)
                            (14.9, -0.06751)
                            };
                        \end{axis}
                    \end{tikzpicture}
            		  \caption{{\footnotesize $d = 2$, $D_2 = 0.05$, $(s_{12},s_{21}) = (-0.08,0.67)$}}
                \label{s_1-s_2=-0.08,0.67}
            \end{subfigure}
            \hfill
            \begin{subfigure}[t]{0.32\linewidth}
                \centering
        			\begin{tikzpicture}
        				\begin{axis}[
                            scale only axis,
        					xlabel={$t$},
        					ylabel={$\mathcal{E}_{(2,2)}(t)$},
        					xlabel style={font=\small\sffamily, yshift=0em},
        					ylabel style={font=\small\sffamily, yshift=0.0em},
        					tick label style={font=\footnotesize},
        					grid=major,
        					grid style={gray!30},
        					width=0.75\textwidth,
        					xmin=0.0,
        					xmax=15.0,
        					ymin=0.0,
        					ymax=2.9,
                            ytick={0.,0.5,1.0,1.5,2.0,2.5},
        					legend pos=north west,
        					legend style={font=\scriptsize\sffamily, inner sep=2pt, outer sep=0pt},
        					mark size=2pt,
        					]
                            \addplot[thick, mark=none
        					] coordinates {
                            (0.0, 0.667)
                            (0.1, 0.713)
                            (0.2, 0.774)
                            (0.3, 0.85)
                            (0.4, 0.94)
                            (0.5, 1.041)
                            (0.6, 1.152)
                            (0.7, 1.269)
                            (0.8, 1.388)
                            (0.9, 1.507)
                            (1.0, 1.622)
                            (1.1, 1.73)
                            (1.2, 1.831)
                            (1.3, 1.922)
                            (1.4, 2.004)
                            (1.5, 2.076)
                            (1.6, 2.14)
                            (1.7, 2.196)
                            (1.8, 2.244)
                            (1.9, 2.287)
                            (2.0, 2.323)
                            (2.1, 2.355)
                            (2.2, 2.383)
                            (2.3, 2.407)
                            (2.4, 2.428)
                            (2.5, 2.447)
                            (2.6, 2.463)
                            (2.7, 2.478)
                            (2.8, 2.491)
                            (2.9, 2.502)
                            (3.0, 2.512)
                            (3.1, 2.521)
                            (3.2, 2.529)
                            (3.3, 2.537)
                            (3.4, 2.543)
                            (3.5, 2.549)
                            (3.6, 2.555)
                            (3.7, 2.56)
                            (3.8, 2.564)
                            (3.9, 2.568)
                            (4.0, 2.572)
                            (4.1, 2.576)
                            (4.2, 2.579)
                            (4.3, 2.582)
                            (4.4, 2.585)
                            (4.5, 2.588)
                            (4.6, 2.59)
                            (4.7, 2.592)
                            (4.8, 2.594)
                            (4.9, 2.596)
                            (5.0, 2.598)
                            (5.1, 2.6)
                            (5.2, 2.602)
                            (5.3, 2.603)
                            (5.4, 2.605)
                            (5.5, 2.606)
                            (5.6, 2.607)
                            (5.7, 2.608)
                            (5.8, 2.609)
                            (5.9, 2.61)
                            (6.0, 2.611)
                            (6.1, 2.612)
                            (6.2, 2.613)
                            (6.3, 2.614)
                            (6.4, 2.615)
                            (6.5, 2.615)
                            (6.6, 2.616)
                            (6.7, 2.616)
                            (6.8, 2.617)
                            (6.9, 2.618)
                            (7.0, 2.618)
                            (7.1, 2.619)
                            (7.2, 2.619)
                            (7.3, 2.619)
                            (7.4, 2.62)
                            (7.5, 2.62)
                            (7.6, 2.62)
                            (7.7, 2.621)
                            (7.8, 2.621)
                            (7.9, 2.621)
                            (8.0, 2.621)
                            (8.1, 2.622)
                            (8.2, 2.622)
                            (8.3, 2.622)
                            (8.4, 2.622)
                            (8.5, 2.622)
                            (8.6, 2.622)
                            (8.7, 2.622)
                            (8.8, 2.622)
                            (8.9, 2.622)
                            (9.0, 2.622)
                            (9.1, 2.622)
                            (9.2, 2.622)
                            (9.3, 2.622)
                            (9.4, 2.622)
                            (9.5, 2.622)
                            (9.6, 2.622)
                            (9.7, 2.622)
                            (9.8, 2.622)
                            (9.9, 2.622)
                            (10.0, 2.622)
                            (10.1, 2.622)
                            (10.2, 2.622)
                            (10.3, 2.622)
                            (10.4, 2.622)
                            (10.5, 2.622)
                            (10.6, 2.621)
                            (10.7, 2.621)
                            (10.8, 2.621)
                            (10.9, 2.621)
                            (11.0, 2.621)
                            (11.1, 2.621)
                            (11.2, 2.621)
                            (11.3, 2.62)
                            (11.4, 2.62)
                            (11.5, 2.62)
                            (11.6, 2.62)
                            (11.7, 2.62)
                            (11.8, 2.619)
                            (11.9, 2.619)
                            (12.0, 2.619)
                            (12.1, 2.619)
                            (12.2, 2.618)
                            (12.3, 2.618)
                            (12.4, 2.618)
                            (12.5, 2.618)
                            (12.6, 2.618)
                            (12.7, 2.617)
                            (12.8, 2.617)
                            (12.9, 2.617)
                            (13.0, 2.616)
                            (13.1, 2.616)
                            (13.2, 2.616)
                            (13.3, 2.616)
                            (13.4, 2.615)
                            (13.5, 2.615)
                            (13.6, 2.615)
                            (13.7, 2.615)
                            (13.8, 2.614)
                            (13.9, 2.614)
                            (14.0, 2.614)
                            (14.1, 2.613)
                            (14.2, 2.613)
                            (14.3, 2.613)
                            (14.4, 2.612)
                            (14.5, 2.612)
                            (14.6, 2.612)
                            (14.7, 2.612)
                            (14.8, 2.611)
                            (14.9, 2.611)
        					};
        				\end{axis}
        			\end{tikzpicture}
                
                    \begin{tikzpicture}
                        \begin{axis}[
                            scale only axis,
                            xlabel={$t$},
                            ylabel={$\frac{\rmd}{\rmd t} \mathcal{E}_{(2,2)}(t)$},
                            xlabel style={font=\small\sffamily, yshift=0em},
                            ylabel style={font=\small\sffamily, yshift=0.0em},
                            tick label style={font=\footnotesize},
                            grid=major,
                            grid style={gray!30},
                            width=0.75\textwidth,
                            xmin=0.0,
                            xmax=15.0,
                            ymin=-0.05,
                            ymax=1.3,
                            ytick={0,0.25,0.5,0.75,1,1.25},
                            legend pos=north west,
                            legend style={font=\scriptsize\sffamily, inner sep=2pt, outer sep=0pt},
                            mark size=2pt,
                            ]
                            \addplot[thick, mark=none]
                            coordinates {
                            (0.0, 0.401)
                            (0.1, 0.558)
                            (0.2, 0.703)
                            (0.3, 0.843)
                            (0.4, 0.97)
                            (0.5, 1.075)
                            (0.6, 1.149)
                            (0.7, 1.19)
                            (0.8, 1.195)
                            (0.9, 1.168)
                            (1.0, 1.113)
                            (1.1, 1.038)
                            (1.2, 0.95)
                            (1.3, 0.856)
                            (1.4, 0.762)
                            (1.5, 0.672)
                            (1.6, 0.589)
                            (1.7, 0.514)
                            (1.8, 0.447)
                            (1.9, 0.388)
                            (2.0, 0.337)
                            (2.1, 0.293)
                            (2.2, 0.255)
                            (2.3, 0.223)
                            (2.4, 0.195)
                            (2.5, 0.172)
                            (2.6, 0.151)
                            (2.7, 0.134)
                            (2.8, 0.119)
                            (2.9, 0.106)
                            (3.0, 0.095)
                            (3.1, 0.085)
                            (3.2, 0.076)
                            (3.3, 0.069)
                            (3.4, 0.063)
                            (3.5, 0.057)
                            (3.6, 0.052)
                            (3.7, 0.047)
                            (3.8, 0.043)
                            (3.9, 0.04)
                            (4.0, 0.037)
                            (4.1, 0.034)
                            (4.2, 0.031)
                            (4.3, 0.029)
                            (4.4, 0.027)
                            (4.5, 0.025)
                            (4.6, 0.023)
                            (4.7, 0.022)
                            (4.8, 0.02)
                            (4.9, 0.019)
                            (5.0, 0.018)
                            (5.1, 0.016)
                            (5.2, 0.015)
                            (5.3, 0.014)
                            (5.4, 0.014)
                            (5.5, 0.013)
                            (5.6, 0.012)
                            (5.7, 0.011)
                            (5.8, 0.01)
                            (5.9, 0.01)
                            (6.0, 0.009)
                            (6.1, 0.009)
                            (6.2, 0.008)
                            (6.3, 0.008)
                            (6.4, 0.007)
                            (6.5, 0.007)
                            (6.6, 0.006)
                            (6.7, 0.006)
                            (6.8, 0.005)
                            (6.9, 0.005)
                            (7.0, 0.005)
                            (7.1, 0.004)
                            (7.2, 0.004)
                            (7.3, 0.004)
                            (7.4, 0.003)
                            (7.5, 0.003)
                            (7.6, 0.003)
                            (7.7, 0.003)
                            (7.8, 0.002)
                            (7.9, 0.002)
                            (8.0, 0.002)
                            (8.1, 0.002)
                            (8.2, 0.002)
                            (8.3, 0.001)
                            (8.4, 0.001)
                            (8.5, 0.001)
                            (8.6, 0.001)
                            (8.7, 0.001)
                            (8.8, 0.001)
                            (8.9, 0.0)
                            (9.0, 0.0)
                            (9.1, 0.0)
                            (9.2, 0.0)
                            (9.3, -0.0)
                            (9.4, -0.0)
                            (9.5, -0.0)
                            (9.6, -0.0)
                            (9.7, -0.001)
                            (9.8, -0.001)
                            (9.9, -0.001)
                            (10.0, -0.001)
                            (10.1, -0.001)
                            (10.2, -0.001)
                            (10.3, -0.001)
                            (10.4, -0.001)
                            (10.5, -0.001)
                            (10.6, -0.001)
                            (10.7, -0.001)
                            (10.8, -0.002)
                            (10.9, -0.002)
                            (11.0, -0.002)
                            (11.1, -0.002)
                            (11.2, -0.002)
                            (11.3, -0.002)
                            (11.4, -0.002)
                            (11.5, -0.002)
                            (11.6, -0.002)
                            (11.7, -0.002)
                            (11.8, -0.002)
                            (11.9, -0.002)
                            (12.0, -0.002)
                            (12.1, -0.002)
                            (12.2, -0.002)
                            (12.3, -0.002)
                            (12.4, -0.002)
                            (12.5, -0.002)
                            (12.6, -0.002)
                            (12.7, -0.003)
                            (12.8, -0.003)
                            (12.9, -0.003)
                            (13.0, -0.003)
                            (13.1, -0.003)
                            (13.2, -0.003)
                            (13.3, -0.003)
                            (13.4, -0.003)
                            (13.5, -0.003)
                            (13.6, -0.003)
                            (13.7, -0.003)
                            (13.8, -0.003)
                            (13.9, -0.003)
                            (14.0, -0.003)
                            (14.1, -0.003)
                            (14.2, -0.003)
                            (14.3, -0.003)
                            (14.4, -0.003)
                            (14.5, -0.003)
                            (14.6, -0.003)
                            (14.7, -0.003)
                            (14.8, -0.003)
                            (14.9, -0.003)
                            };
                        \end{axis}
                    \end{tikzpicture}
            		  \caption{{\footnotesize $d=1$, $D_2 = 0.15$, $(s_{12}, s_{21}) = (-0.2, 0.7)$}}
                \label{s_1-s_2=-0.2,0.7}
            \end{subfigure}
    		\caption{Plots of \(\mathcal{E}_{(p_1,p_2)}(t)\) and \(\frac{\rmd}{\rmd t} \mathcal{E}_{(p_1,p_2)}(t)\) for numerical solutions of~\eqref{eqn:2d-cauchy} for fixed \(D_1 = 0.2\) and varying choices of \((s_{12},s_{21})\), \(D_2\) and \(d\). Simulations were run with uniform time steps of size \(10^{-2}\). For \(d = 2\), we fixed \(L=3\) and used a uniform square grid with widths of \(5 \times 10^{-3}\) and for \(d=1\) we fixed \(L=5\) and used a uniform grid of widths \(8 \times 10^{-5}\).}
    	\end{figure}

        Again, similarly to the previous example, the numerical simulations suggest that the definition of the critical regularity parameter can be generalised such that \(\bar{\lambda} = 1 - s\), even when \(\nabla W_s \not\in L^q(\R^d)\) for any \(q \in (1,\infty)\). This is demonstrated in Figure~\ref{s_1-s_2=-0.2,0.7}, where the choice of \((s_{12}, s_{21}) = (-0.2,0.7)\) falls outside the conditions of Proposition~\ref{prop:L^p_est} when \(d = 1\), yet the numerical solution to~\eqref{eqn:2d-cauchy} seems to exhibit similar behaviour to what is seen in Figures~\ref{s_1-s_2=0.4,0.4} and~\ref{s_1-s_2=-0.08,0.67}.   
    \end{example}

\section{Discussion}\label{section:conclusion}
    
	By allowing each component of the initial data \(u_{0i}\) and each kernel \(K_{ij}\) to have their own regularity, we have demonstrated the importance of considering the interaction cycles in the system~\eqref{eqn:agg_diff} when establishing global-in-time estimates of the solution. Furthermore, by considering the dynamics of the concentration of each species \(u_i\), we have been able to extend the global well-posedness and small \(L^p\) estimates, as previously proven in~\cite{karch2011blow} for the single-species system, to a multi-species system.
    
	On the other hand, a result from~\cite{karch2011blow} that has not been extended to the heterogeneous system in this work is the demonstration of a finite-time blow-up for certain solutions to~\eqref{eqn:agg_diff}. For $N = 1$, it is proven that, for any \(q < d\), there exists some kernel satisfying \(\nabla K \in L^q(\R^d) \setminus L^d(\R^d)\) and some initial data \(u_0 \in L^1 \cap L^\infty(\R^d)\) such that the second moment of the solution vanishes in finite time, a sufficient condition for a chemotactic collapse. This raises the open problem of whether the geometric mean condition for every interaction cycle in Theorem~\ref{thm:global-1} is sharp. That is, if we assume that there is a choice of exponents \(q_{i_1 i_2}, \dots, q_{i_n i_1}\) such that \(\nabla K_{i_k i_{k+1}} \in L^{q_{i_k i_{k+1}}}(\R^d)\) for each \(k \in \{1,\dots,n\}\), \(i_{n+1} = i_1\), and \((q_{i_1 i_2} \cdots q_{i_n i_1})^{1/n} < d\), then there exists a choice of kernels \(K_{i_1 i_2}, \dots, K_{i_n i_1}\) and initial data \(u_0\) such that the solution to~\eqref{eqn:agg_diff} exhibits a finite-time blow-up. 

    Another open problem that follows on from Theorem~\ref{thm:small-conc} is obtaining a more complete classification of kernels and initial data for which the corresponding operator \(\Phi\), as defined in~\eqref{defn:phi}, has a positive fixed point. Motivated by Example~\ref{ex:no-self-perp} and the interaction cycle condition in Theorem~\ref{thm:global-1}, we believe that it should be possible to show that \(\Phi\) has a positive fixed point independent of the initial data if the perception graph is strongly connected and the interaction cycles either all satisfy \(\lambda_{i_1 i_2} \cdots \lambda_{i_n i_1} > 1\), or all satisfy \(\lambda_{i_1 i_2} \cdots \lambda_{i_n i_1} < 1\). 
    
	Finally, the examples considered in Section~\ref{section:numerics} indicate the existence of a broader regularity condition on the kernels. In particular, Example~\ref{eg:num-good-cycle} suggests that the condition \(\nabla K \in L^q(\R^d)\) for some \(q > 1\) is not necessary for global existence of solutions. By investigating weaker regularity conditions, it may be possible to unify the well-posedness results of this work with analogous results for top hat kernels, as studied in~\cite{carrillo2024well}.

    
	\appendix

	\section{Sobolev space inequalities}

	\begin{lemma}[\!\!{\cite[\(\S\)B.2]{evans2010}}]\label{lem:lp_int}
		Suppose \(\psi \in L^p \cap L^q(\R^d)\) for some \(1 \leq p \leq q \leq \infty\). Then \(\psi \in L^r(\R^d)\) for all \(r \in [p,q]\) with the \(L^r\) bound
		\begin{equation}\label{ineq:lp_int}
			\norm{\psi}_{L^r} \leq \norm{\psi}_{L^p}^{1-\theta} \norm{\psi}_{L^q}^\theta,
		\end{equation}
		where \(\theta \in [0,1]\) satisfies 
		\begin{equation}\label{cond:lp_int_theta}
			\frac{1}{r} = \frac{1-\theta}{p} + \frac{\theta}{q}.
		\end{equation}
	\end{lemma}
	\begin{remark}\label{rem:bessel_int}
		If \(\varphi \in H^{s,p} \cap H^{s,q}(\R^d)\) for some \(s \in [0,\infty)\) and \(1 < p \leq q < \infty\), then using the definition a Bessel potential space there exists some locally integrable \(\psi\) such that \(\mathcal{J}_s \psi = \varphi\). Therefore it can be shown that \(\varphi \in H^{s,r}\) for any \(r \in [p,q]\) with the bound
		\begin{equation}\label{ineq:bessel_int}
			\norm{\varphi}_{H^{s,r}} = \norm{\psi}_{L^r} \leq \norm{\psi}_{L^p}^{1-\theta} \norm{\psi}_{L^q}^\theta =  \norm{\varphi}_{H^{s,p}}^{1-\theta} \norm{\varphi}_{H^{s,q}}^\theta,
		\end{equation}
		where \(\theta \in [0,1]\) satisfies~\eqref{cond:lp_int_theta}. Furthermore, by considering the following inequality~\cite[\(\S\)B.2]{evans2010}, for any \(a,\, b \geq 0\) and \(\theta \in [0,1]\),
        \begin{equation}
            a^{1-\theta}b^{\theta} \leq (1-\theta)a + \theta b \leq a + b
        \end{equation}
        it follows via~\eqref{ineq:bessel_int} that \(H^{s,p} \cap H^{s,q}(\R^d) \subset H^{s,r}(\R^d)\) for all \(s \in [0,\infty)\) as
        \begin{equation}\label{ineq:inbtwn-space}
            \norm{\varphi}_{H^{s,r}} \leq \norm{\varphi}_{H^{s,p}} + \norm{\varphi}_{H^{s,q}} =: \norm{\varphi}_{H^{s,p} \cap H^{s,q}}.
        \end{equation}
	\end{remark}
    
	The next two results are simple generalisations of the well-known Hölder's and Young's inequalities~\cite[pp.~39,\,90--91]{lieb2001analysis}. Both results can be extended to hold over \(W^{k,p}\) spaces using the Leibniz rule for differentiation and \(\partial_{x_i} (\varphi \ast \psi) = (\partial_{x_i} \varphi) \ast \psi =  \varphi \ast (\partial_{x_i} \psi)\) respectively, where \(\displaystyle{(\varphi \ast \psi)(x) := \int_{\R^d} \varphi(x-y) \psi(y) \rmd y}\).

	\begin{lemma}\label{lem:sobolev_ineqs}
		Suppose \(\varphi \in W^{k,p}(\R^d)\) and \(\psi \in W^{m,q}(\R^d)\) for \(1 \leq p,\,q \leq \infty\) and \(k,\, m \in \N_0\). Then we have the following:
        \begin{itemize}
            \item If \(k=m\) and \(0 \leq 1/p + 1/q \leq 1\), then \(\varphi \psi \in W^{k,r}(\R^d)\) with the bound
    		\begin{equation}\label{ineq:holder}
    			\norm{\varphi \psi}_{W^{k,r}} \leq C_H \norm{\varphi}_{W^{k,p}}\norm{\psi}_{W^{k,q}}, \quad 1/r = 1/p + 1/q,
    		\end{equation}
    		where \(C_H > 0\) depends on \(k,\, d\) and \(C_H = 1\) when \(k \in \{0,\, 1\}\).
            \item If \(0 \leq 1/p + 1/q - 1 \leq 1\), then 
			\begin{equation}\label{ineq:young_conv}
				\norm{\varphi \ast \psi}_{W^{k+m,r}} \leq \norm{\varphi}_{W^{k,p}}\norm{\psi}_{W^{m,q}},\quad 1/r = 1/p + 1/q - 1.
			\end{equation}
        \end{itemize}      
	\end{lemma}

	As \((H^{s^0,p}(\R^d), H^{s^1,p}(\R^d))\) form a \emph{complex interpolation pair} for any choice of \(s^0, s^1 \in [0,\infty)\) and \(p \in (1,\infty)\), they satisfy the conditions of the multilinear interpolation result in~\cite[Theorem 4.4.1]{bergh1976}. That is, we can state the following. 

	\begin{lemma}\label{lem:multi_int}
		For some \(n \in \N\), let \((H^{s^0_{i},\,p_i}, H^{s^1_{i},\,p_i})_{i=1}^{n+1}\) be a collection of pairs of Bessel potential spaces, such that \(0 \leq s^0_{i} \leq s^1_{i} < \infty\) and \(p_i \in (1,\infty)\) for each \(i \in \{1,\dots,n+1\}\). Now suppose that \(T : \bigoplus_{i = 1}^n (H^{s^0_i,\,p_i} \cap H^{s^1_i,\,p_i}) \to H^{s_{n+1}^0,\,p_{n+1}} \cap H^{s_{n+1}^1,\,p_{n+1}}\) is a multilinear operator such that, for any \(\varphi = (\varphi_1, \dots, \varphi_n) \in \bigoplus_{i = 1}^n (H^{s^0_i,\,p_i} \cap H^{s^1_i,\,p_i})\), there exists constants \(M_0,\, M_1 > 0\) such that
		\begin{equation}
			\norm{T[\varphi]}_{H^{s_{n+1}^0,\,p_{n+1}}} \leq M_0 \prod_{i=1}^{n} \norm{\varphi_i}_{H^{s^0_{i},\,p_i}},\quad \text{and}\quad \norm{T[\varphi]}_{H^{s_{n+1}^1,\,p_{n+1}}} \leq M_1 \prod_{i=1}^{n} \norm{\varphi_i}_{H^{s^1_{i},\,p_i}}.
		\end{equation}	
		Then, for any \(\theta \in [0,1]\), there exists a unique extension \(T:\bigoplus_{i = 1}^n H^{s^\theta_{i},\, p_i} \to H^{s_{n+1}^\theta,p_{n+1}}\) with operator norm \(M_\theta \leq M_0^{1-\theta} M_1^\theta\), where \((s^\theta_i)_{i=1}^{n+1}\) satisfies \(s^\theta_i = (1-\theta)s^0_{i} + \theta s^1_{i}\) for each \(i\).
	\end{lemma}

	One application of the mutli-linear interpolation is the following trilinear bound.

	\begin{lemma}\label{lem:tri_lin_s_bound}
		Suppose \(\nabla \mathcal{K} \in L^q(\R^d)\), \(\varphi \in H^{s,r}(\R^d)\) and \(\psi \in H^{s,q'}(\R^d)\) for \(s \in [0,\infty)\) and \(r,\,q \in (1,\infty)\). Then, \(\abs{\varphi \nabla (\mathcal{K}\ast \psi)} \in H^{s,r}(\R^d)\) with the bound
		\begin{equation}\label{ineq:int_flux_bound}
			\norm{\varphi \nabla (\mathcal{K}\ast \psi)}_{H^{s,r}} \leq  C_H \norm{\nabla \mathcal{K}}_{L^q} \norm{\varphi}_{H^{s,r}} \norm{\psi}_{H^{s,q'}},
		\end{equation}
		where \(C_H > 0\) depends on \(s\) and \(d\).
	\end{lemma}

	\begin{proof}
		For some fixed \(\nabla \mathcal{K} \in L^q(\R^d)\) consider the bilinear map \(T(\varphi,\psi) := \varphi \nabla (\mathcal{K} \ast \psi)\), where \(\varphi \in W^{k,r}(\R^d)\) and  \(\psi \in W^{k,q'}(\R^d)\) for some \(k \in \N_0\). By applying~\eqref{ineq:holder} followed by~\eqref{ineq:young_conv} from Lemma~\ref{lem:sobolev_ineqs}, observe that
		\begin{equation}\label{ineq:tri-lin-k}
			\norm{\varphi \nabla (\mathcal{K} \ast \psi)}_{W^{k,r}} \leq C_H(k,d) \norm{\nabla \mathcal{K}}_{L^q} \norm{\varphi}_{W^{k,r}}  \norm{\psi}_{W^{k,q'}}.
		\end{equation}
		Now for some \(s \in [0,\infty)\), consider~\eqref{ineq:tri-lin-k} for \(k^0,\, k^1 \in \N_0\) such that \(k^0 \leq s \leq k^1\). By then applying Lemma~\ref{lem:multi_int} on the operator \(T\), conclude that~\eqref{ineq:int_flux_bound} holds for some constant \(C_H(s,d) \leq C_H(k^0,d)^{1-\theta} C_H(k^1,d)^\theta\), where \(\theta\) is chosen such that \(s = (1-\theta) k^0 + \theta k^1\). 
	\end{proof}

	\section{Heat semigroup}

	\begin{definition}[Heat semigroup]\label{defn:heat_semi}
		For some fixed \(\nu > 0\), we define \(\{S(t)\}_{t \geq 0}\) to be the \emph{heat semigroup} over \(\R^d\) with diffusion constant \(\nu\). For \(t=0\), \(S(0)\) is the identity operator and, for each \(t > 0\), the heat semigroup can be written as the convolution operator \(S(t) := \Gamma_\nu(t,\cdot)\, \ast\), where
		\begin{equation}\label{eqn:heat_semi_conv}
			 \Gamma_\nu(t,x) := (4 \pi \nu t)^{-d/2} \exp(-\abs{x}^2/4 \nu t)
		\end{equation}
		is the heat kernel for diffusion constant \(\nu\).
	\end{definition}

	As bounded continuous functions are dense in \(L^p(\R^n)\) for all \(1 \leq p < \infty\), the following properties hold from~\cite[\(\S 7\)]{jost2007partial}.
    
	\begin{lemma}\label{lem:heat_semi_props}
		Suppose \(1 \leq p < \infty\) and \(\nu > 0\). Then the heat semigroup \(S(t)\) with diffusion constant \(\nu > 0\) satisfies the following for all \(\varphi \in L^{p}(\R^d)\), and \(t,\,s \in [0,\infty)\):
		\begin{itemize}
			\item \(S(0) \varphi = \varphi\), \emph{(Identity at 0)}
			\item \(S(t+s) \varphi = S(t)S(s)\varphi\), \emph{(Semigroup property)}
			\item \(\lim_{t \to 0} \norm{S(t)\varphi - \varphi}_{L^{p}} = 0\), \emph{(Strong continuity)}
			\item \(\frac{\rmd}{\rmd t} S(t)\varphi = \nu \Delta S(t) \varphi,\) for all \(t > 0\). \emph{(Differentiability)}
		\end{itemize}
	\end{lemma}

	\begin{remark}\label{rem:semi_s,p}
        Note that the strong continuity of \(S(t)\) also holds for the \(H^{s,p}\) norm. First observe that \(\Gamma_\nu(t,\cdot) \ast J_s = J_s \ast \Gamma_\nu(t,\cdot)\), which implies that \(S(t) \mathcal{J}_s = \mathcal{J}_s S(t)\). By then considering \(\varphi = \mathcal{J}_s \psi \in H^{s,p}(\R^d)\) for some \(\psi \in L^p(\R^d)\) it then follows that
        \begin{equation}\label{eqn:strong-cty}
            \lim_{t \to 0} \norm{S(t)\varphi - \varphi}_{H^{s,p}} = \lim_{t \to 0} \norm{\mathcal{J}_s(S(t)\psi - \psi)}_{H^{s,p}} = \lim_{t \to 0} \norm{S(t)\psi - \psi}_{L^{p}} = 0.
        \end{equation}
    \end{remark}
    
	
	Of particular interest are the following inequalities for the heat semigroup, as stated in~\cite{karch2011blow, cozzi2025global}.\noeqref{ineq:semigrp_hess_bound}
	\begin{lemma}\label{lem:semi_lp_bounds}
		Let \(1 \leq p \leq \infty\) and suppose \(\varphi \in L^{p}(\R^d)\). Then, for all \(t > 0\),
		\begin{align}
			\norm{S(t)\varphi}_{L^{p}} & \leq \norm{\varphi}_{L^{p}} \label{ineq:semigrp_contraction}\\
			\norm{\nabla S(t) \varphi}_{L^{p}} & \leq C_1 (\nu t)^{-\frac{1}{2}} \norm{\varphi}_{L^{p}}, \label{ineq:semigrp_deriv_bund}
			\\ \norm{\nabla^2 S(t) \varphi}_{L^p} & \leq C_{2} (\nu t)^{-1} \norm{\varphi}_{L^p}, \label{ineq:semigrp_hess_bound}
		\end{align}
		where \(C_1,\, C_{2} > 0\) both depend on \(d\).
	\end{lemma}

    \begin{proof}
        Recall the definition of the heat semigroup, \( S(t) \varphi = \Gamma_\nu(t,\cdot) \ast \varphi\), where \(\Gamma_\nu(t,\cdot)\) is a Gaussian function for all \(t > 0\). Therefore, via Young's convolution inequality it can be seen that, for any \(\alpha \in \N_0^d\),
        \begin{equation}
            \norm{D^\alpha S(t) \varphi}_{L^p} \leq \norm{D^\alpha \Gamma_\nu(t,\cdot)}_{L^1} \norm{\varphi}_{L^p}.
        \end{equation}
        The result then follows by evaluating \(\norm{\Gamma_\nu(t,\cdot)}_{L^1}\), \(\norm{\nabla \Gamma_\nu(t,\cdot)}_{L^1}\) and \(\norm{\nabla^2 \Gamma_\nu(t,\cdot)}_{L^1}\) respectively.
    \end{proof}
    Due to the inequality in~\eqref{ineq:young_conv}, it follows that Lemmas~\ref{lem:heat_semi_props} and~\ref{lem:semi_lp_bounds} can be shown to also hold for \(W^{k,p}\) spaces. Furthermore, observe that from the definitions of \(\norm{\psi}_{W^{k,p}}\), \(\norm{\nabla \psi}_{W^{k,p}}\) and \(\norm{\nabla^2 \psi}_{W^{k,p}}\),
    \begin{equation}\label{ineq:semi_wkp}
        \begin{aligned}
            \norm{S(t)\varphi}_{W^{{k+1},p}} &= \norm{S(t)\varphi}_{W^{k,p}} + \norm{\nabla S(t) \varphi}_{W^{k,p}} \leq \left(1 + C_1(\nu t)^{-1/2}\right) \norm{\varphi}_{W^{k,p}}, \\
            \norm{\nabla S(t)\varphi}_{W^{{k+1},p}} &= \norm{\nabla S(t)\varphi}_{W^{k,p}} + \norm{\nabla^2 S(t) \varphi}_{W^{k,p}} \leq \left(C_1(\nu t)^{-1/2} + C_2(\nu t)^{-1}\right) \norm{\varphi}_{W^{k,p}}.
        \end{aligned}
    \end{equation}
    
    As \(1 \leq (\nu t)^{-1/2} \leq (\nu t)^{-1}\) for all \(t \leq \nu^{-1}\) and vice versa for \(t > \nu^{-1}\), we can apply Lemma~\ref{lem:multi_int} to conclude the following result~\cite{taylor2011vol3}.
    
	\begin{lemma}\label{lem:semi_embed}
		Suppose \(\nu > 0\), \(s \in [0,\infty)\), \(\sigma \in [0,1]\) and \(p \in (1,\infty)\). Then for any \(0 < t \leq \nu^{-1}\), \(S(t)\) and \(\nabla S(t)\) are bounded operators between \(H^{s,p}(\R^d)\) and \(H^{s+\sigma,p}(\R^d)\). That is, for any \(\varphi \in H^{s,p}(\R^d)\), \(t > 0\),
		\begin{align}
			\norm{S(t)\varphi}_{H^{s+\sigma,p}} &\leq C_{\sigma} \max\{(\nu t)^{-\sigma/2},\, 1\} \norm{\varphi}_{H^{s,p}} \label{ineq:snu_1}\\
			\norm{\nabla S(t)\varphi}_{H^{s+\sigma,p}} &\leq C_{\sigma+1} \max\{(\nu t)^{-(\sigma+1)/2},\, (\nu t)^{-1/2}\} \norm{\varphi}_{H^{s,p}}, \label{ineq:snu_2}
		\end{align}
		where \(C_{\sigma} ,\, C_{\sigma+1} > 0\) are constants depending on \(\sigma\) and \(d\).
	\end{lemma}




	{\bibliography{refs_abbrev}{}

@article{armstrong2006continuum,
  title = {A continuum approach to modelling cell-cell adhesion},
  journal = {J. Theoret. Biol.},
  volume = {243},
  number = {1},
  pages = {98--113},
  year = {2006},
  issn = {0022-5193},
  author = {Armstrong, Nicola J. and Painter, Kevin J.  and Sherratt,Jonathan A.}
}

@article{bernoff2013nonlocal,
  title={Nonlocal aggregation models: A primer of swarm equilibria},
  author={Bernoff, Andrew J. and Topaz, Chad M.},
  journal={SIAM Rev.},
  volume={55},
  number={4},
  pages={709--747},
  year={2013},
  publisher={SIAM}
}

@article{blanchet2006two,
  title={TWO-DIMENSIONAL {K}ELLER-{S}EGEL MODEL: OPTIMAL CRITICAL MASS AND QUALITATIVE PROPERTIES OF THE SOLUTIONS},
  author={Blanchet, Adrien and Dolbeault, Jean and Perthame, Beno{\^i}t},
  journal={Electron. J. Differential Equations},
  volume = {2006},
  number={44},
  pages={1--33},
  year={2006}
}

@article{carrillo2024well,
  title={Well-posedness of aggregation-diffusion systems with irregular kernels},
  author={Carrillo, Jos{\'e} A. and Salmaniw, Yurij and Skrzeczkowski, Jakub},
  journal={Ann. Inst. H. Poincar{\'e} C Anal. Non Lin{\'e}aire},
  year={2026},
  note={Online first}
}

@article{carrillo2026long,
  title={Long-time behaviour and bifurcation analysis of a two-species aggregation-diffusion system on the torus},
  author={Carrillo, Jos{\'e} A. and Salmaniw, Yurij},
  journal={Calc. Var. Partial Differential Equations},
  volume={65},
  number={1},
  pages={19},
  year={2026},
  publisher={Springer}
}

@article{cozzi2025global,
  title={Global existence for a nonlocal multi-species aggregation-diffusion equation},
  author={Cozzi, Elaine and Radke, Zachary},
  journal={Quart. Appl. Math.},
  volume = {84},
  number = {1},
  pages = {1--29},
  year={2026}
}

@article{fagan2017perceptual,
  title={Perceptual ranges, information gathering, and foraging success in dynamic landscapes},
  author={Fagan, William F. and Gurarie, Eliezer and Bewick, Sharon and Howard, Allison and Cantrell, Robert S. and Cosner, Chris},
  journal={Am. Nat.},
  volume={189},
  number={5},
  pages={474--489},
  year={2017},
  publisher={University of Chicago Press Chicago, IL}
}

@article{fetecau2011swarm,
  title={Swarm dynamics and equilibria for a nonlocal aggregation model},
  author={Fetecau, Razvan C and Huang, Yanghong and Kolokolnikov, Theodore},
  journal={Nonlinearity},
  volume={24},
  number={10},
  pages={2681},
  year={2011},
  publisher={IOP Publishing}
}

@article{giuntaLocalGlobalExistence2022,
  author = {Valeria Giunta and Thomas Hillen and Mark A. Lewis and Jonathan R. Potts},
  title = {Local and global existence for nonlocal multispecies advection-diffusion models},
  journal = {SIAM J. Appl. Dyn. Syst.},
  volume = {21},
  number = {3},
  pages = {1686--1708},
  year = {2022}
}

@misc{giunta2025phylogeny,
  title={A phylogeny of biological patterns formed by nonlocal advection},
  author={Valeria Giunta and Thomas Hillen and Mark A. Lewis and Jonathan R. Potts},
  year={2025},
  note={arXiv preprint arXiv:2506.00489}
}

@article{giunta2025positivity,
    title = {Positivity and global existence for nonlocal advection-diffusion models of interacting populations},
    journal = {AIMS Math.},
    volume = {10},
    number = {9},
    pages = {21254-21272},
    year = {2025},
    issn = {2473-6988},
    author = {Valeria Giunta and Thomas Hillen and Mark A. Lewis and Jonathan R. Potts},
}

@article{haubold2011mittag,
  title={{M}ittag-{L}effler functions and their applications},
  author={Haubold, Hans J. and Mathai, Arak M. and Saxena, Ram K.},
  journal={J. Appl. Math.},
  volume={2011},
  number={1},
  pages={298628},
  year={2011},
  publisher={Wiley Online Library}
}

@article{he2021multi,
  title={Multi-species {Patlak--Keller--Segel} system},
  author={He, Siming and Tadmor, Eitan},
  journal={Indiana Univ. Math. J.},
  volume={70},
  number={4},
  pages={1577--1624},
  year={2021},
  publisher={JSTOR}
}

@article{herrero1996chemotactic,
  title={Chemotactic collapse for the {Keller--Segel} model},
  author={Herrero, Miguel A. and Vel{\'a}zquez, Juan J. L.},
  journal={J. Math. Biol.},
  volume={35},
  number={2},
  pages={177--194},
  year={1996},
  publisher={Springer}
}

@article{jungel2022nonlocal,
  title = {Nonlocal cross-diffusion systems for multi-species populations and networks},
  journal = {Nonlinear Anal.},
  volume = {219},
  pages = {112800},
  year = {2022},
  issn = {0362-546X},
  author = {Ansgar J\"ungel and Stefan Portisch and Antoine Zurek}
}

@article{karch2011blow,
  title={Blow-up versus global existence of solutions to aggregation equations},
  author={Karch, Grzegorz and Suzuki, Kanako},
  journal={Appl. Math. (Warsaw)},
  volume={38},
  number={3},
  pages={243--258},
  year={2011},
  publisher={Institute of Mathematics}
}

@article{koertje2024collective,
  title={Collective group drift in a partial-differential-equation-based opinion dynamics model with biased perception kernels},
  author={Koertje, Christian and Sayama, Hiroki},
  journal={Phys. Rev. E},
  volume={109},
  number={3},
  pages={034304},
  year={2024},
  publisher={APS}
}

@article{li2010wellposedness,
  title={Wellposedness and regularity of solutions of an aggregation equation},
  journal ={Rev. Mat. Iberoam.},
  author={Li, Dong and Rodrigo, Jos{\'e} L.},
  volume = {26},
  number = {1},
  pages = {261--294},
  year={2010}
}

@article{milewski2008simple,
  title={A simple model for biological aggregation with asymmetric sensing},
  author={Milewski, Paul A. and Xu, Yang},
  journal={Commun. Math. Sci.},
  volume={6},
  number={2},
  pages={397--416},
  year={2008},
  publisher={International Press of Boston, Inc.}
}

@article{nash1958continuity,
  title={Continuity of solutions of parabolic and elliptic equations},
  author={Nash, John},
  journal={Amer. J. Math.},
  volume={80},
  number={4},
  pages={931--954},
  year={1958},
  publisher={JSTOR}
}

@article{painter2024biological,
  author = {Painter, Kevin J. and Hillen, Thomas and Potts, Jonathan R.},
  title = {Biological modeling with nonlocal advection-diffusion equations},
  journal = {Math. Models Methods Appl. Sci.},
  volume = {34},
  number = {1},
  pages = {57--107},
  year = {2024}
}

@article{painter2024variations,
  title={Variations in non-local interaction range lead to emergent chase-and-run in heterogeneous populations},
  author={Painter, Kevin J. and Giunta, Valeria and Potts, Jonathan R. and Bernardi, Sara},
  journal={J. R. Soc. Interface},
  volume={21},
  number={219},
  pages={20240409},
  year={2024},
  publisher={The Royal Society}
}

@article{potts2019spatial,
  title={Spatial memory and taxis-driven pattern formation in model ecosystems},
  author={Potts, Jonathan R. and Lewis, Mark A.},
  journal={Bull. Math. Biol.},
  volume={81},
  number={7},
  pages={2725--2747},
  year={2019},
  publisher={Springer}
}

@article{tello2012stabilization,
  title={Stabilization in a two-species chemotaxis system with a logistic source},
  author={Tello, Jos{\'e} I. and Winkler, Michael},
  journal={Nonlinearity},
  volume={25},
  number={5},
  pages={1413--1425},
  year={2012}
}

@article{wang2023open,
  title = {Open problems in {PDE} models for knowledge-based animal movement via nonlocal perception and cognitive mapping},
  journal = {J. Math. Biol.},
  volume = {86},
  number = {5},
  pages = {71},
  year = {2023},
  issn = {0303-6812},
  author = {Wang, Hao and Salmaniw, Yurij}
}

@article{ye2007generalized,
  title={A generalized {G}r{\"o}nwall inequality and its application to a fractional differential equation},
  author={Ye, Haiping and Gao, Jianming and Ding, Yongsheng},
  journal={J. Math. Anal. Appl.},
  volume={328},
  number={2},
  pages={1075--1081},
  year={2007},
  publisher={Elsevier}
}

@book{bergh1976,
  title = {Interpolation {{Spaces}}: {{An Introduction}}},
  shorttitle = {Interpolation {{Spaces}}},
  author = {Bergh, J{\"o}ran and L{\"o}fstr{\"o}m, J{\"o}rgen},
  year = {1976},
  series = {Grundlehren der mathematischen {{Wissenschaften}}},
  volume = {223},
  publisher = {Springer},
  doi = {10.1007/978-3-642-66451-9},
  urldate = {2024-11-12},
  copyright = {http://www.springer.com/tdm},
  isbn = {978-3-642-66453-3 978-3-642-66451-9},
  langid = {english}
}

@book{cormen2009,
  author = {Cormen, Thomas H. and Leiserson, Charles E. and Rivest, Ronald L. and Stein, Clifford},
  publisher = {The MIT Press, Cambridge, MA},
  edition = {3rd},
  title = {Introduction to Algorithms},
  year = {2009}
}

@book{evans2010,
  title = {Partial Differential Equations},
  author = {Evans, Lawrence C.},
  year = {2010},
  series = {Graduate Studies in Mathematics},
  edition = {2nd},
  volume = {19},
  publisher = {American mathematical society},
  address = {Providence (R.I.)},
  isbn = {978-0-8218-4974-3},
  langid = {english},
  lccn = {515.353}
}

@book{grafakos2009modern,
  title={Modern Fourier Analysis},
  author={Grafakos, Loukas},
  edition = {2nd},
  volume={250},
  year={2009},
  publisher={Springer},
  doi = {https://doi.org/10.1007/978-0-387-09434-2},
}

@book{jost2007partial,
  title={Partial differential equations},
  author={Jost, J{\"u}rgen},
  year={2007},
  publisher={Springer}
}

@book{lieb2001analysis,
  title={Analysis},
  author={Lieb, Elliott H. and Loss, Michael},
  volume={14},
  year={1997},
  edition={1st},
  publisher={American Mathematical Soc.}
}

@book{stein1970singular,
  title={Singular integrals and differentiability properties of functions},
  author={Stein, Elias M.},
  year={1970},
  publisher={Princeton University Press}
}

@book{guo1988nonlinear,
  title={Nonlinear problems in abstract cones},
  author={Guo, Dajun and Lakshmikantham, Vangipuram},
  volume={1},
  year={1988},
  publisher={Academic press}
}

@book{taylor2011vol3,
  title = {Partial {{Differential Equations III}}: {{Nonlinear Equations}}},
  shorttitle = {Partial {{Differential Equations III}}},
  author = {Taylor, Michael E.},
  year = {2011},
  edition = {3rd},
  series = {Applied {{Mathematical Sciences}}},
  volume = {117},
  publisher = {Springer New York},
  doi = {10.1007/978-1-4419-7049-7},
  urldate = {2024-10-22},
  copyright = {http://www.springer.com/tdm},
  isbn = {978-1-4419-7048-0 978-1-4419-7049-7},
  langid = {english}
}

@misc{BarattaEtal2023,
  title     = {{DOLFINx}: the next generation {FEniCS} problem solving environment},
  author    = {Baratta, Igor A. and Dean, Joseph P. and Dokken, J{\o}rgen S. and Habera, Michal and Hale, Jack S. and Richardson, Chris N. and Rognes, Marie E. and Scroggs, Matthew W. and Sime, Nathan and Wells, Garth N.},
  doi       = {10.5281/zenodo.10447666},
  year      = {2023},
  howpublished = {preprint}
}

@article{ScroggsEtal2022,
  title     = {Construction of arbitrary order finite element degree-of-freedom maps on polygonal and polyhedral cell meshes},
  author    = {Scroggs, Matthew W. and Dokken, J{\o}rgen S. and Richardson, Chris N. and Wells, Garth N.},
  journal   = {ACM Trans. Math. Software},
  year      = {2022},
  volume    = {48},
  number    = {2},
  doi       = {10.1145/3524456},
  pages     = {{18:1--18:23}},
}

@article{BasixJoss,
  title     = {Basix: a runtime finite element basis evaluation library},
  author    = {Scroggs, Matthew W. and Baratta, Igor A. and Richardson, Chris N. and Wells, Garth N.},
  journal   = {J. Open Source Softw.},
  year      = {2022},
  volume    = {7},
  number    = {73},
  doi       = {10.21105/joss.03982},
  pages     = {3982}
}

@article{AlnaesEtal2014,
  title     = {Unified Form Language: A domain-specific language for weak formulations of partial differential equations},
  author    = {Alnaes, Martin S. and Logg, Anders and {\O}lgaard, Kristian B. and Rognes, Marie E. and Wells, Garth N.},
  journal   = {ACM Trans. Math. Software},
  year      = {2014},
  volume    = {40},
  doi       = {10.1145/2566630},
}

@article{chen2020mathematical,
  title={Mathematical models for cell migration: a non-local perspective},
  author={Chen, Li and Painter, Kevin and Surulescu, Christina and Zhigun, Anna},
  journal={Philos. Trans. Roy. Soc. B},
  volume={375},
  number={1807},
  pages={20190379},
  year={2020},
  publisher={The Royal Society}
}

@article{giunta2024weakly,
  title={Weakly nonlinear analysis of a two-species non-local advection--diffusion system},
  author={Valeria Giunta and Thomas Hillen and Mark A. Lewis and Jonathan R. Potts},
  journal={Nonlinear Anal. Real World Appl.},
  volume={78},
  pages={104086},
  year={2024},
  publisher={Elsevier}
}
	\bibliographystyle{plain}
    }

\end{document}